\documentclass[12pt,leqno]{amsart}

\usepackage{amsmath,amssymb,amsfonts,amsopn,color}
\usepackage{pdfsync}
\usepackage{amssymb}
\usepackage{enumerate}
\usepackage{amsthm}
\usepackage{anysize}
\marginsize{3cm}{2.5cm}{2cm}{4.5cm} %{left}{right}{top}{bottom}

\theoremstyle{plain}% default 
\newtheorem{theorem}{Theorem}
\newtheorem{lemma}{Lemma} 
\newtheorem{proposition}{Proposition} 
\newtheorem{corollary}{Corollary} 

\theoremstyle{definition}

\newtheorem{example}{Example}

\theoremstyle{remark} 
\newtheorem{remark}{Remark} 

\newcommand{\N}{\mathbb N}
\newcommand{\R}{\mathbb R}
\newcommand{\C}{\mathbb C}
\newcommand{\T}{\mathbb T}

\providecommand{\scal}[2]{\langle#1,#2\rangle}

\newcommand{\id}{\operatorname{id}}
\newcommand{\dime}{\operatorname{dim}}
\newcommand{\trace}{\operatorname{tr}}
\newcommand{\vol}{\operatorname{vol}}
\newcommand{\supp}{\operatorname{supp}}
\newcommand{\card}{\operatorname{card}}

\newcommand{\eps}{\varepsilon}

\newcommand{\cB}{{\mathcal B}}
\newcommand{\cC}{{\mathcal C}}
\newcommand{\cD}{{\mathcal D}}

\newcommand{\cF}{{\mathcal F}}
\newcommand{\cG}{{\mathcal G}}
\newcommand{\cH}{{\mathcal H}}

\newcommand{\cK}{{\mathcal K}}
\newcommand{\cL}{{\mathcal L}}

\newcommand{\cN}{{\mathcal N}}

\newcommand{\cR}{{\mathcal R}}

\newcommand{\cT}{{\mathcal T}}

\newcommand{\cW}{{\mathcal W}}

\begin{document}

\title{A mock metaplectic representation}
\bigskip\bigskip\bigskip
\author{Filippo De Mari}
\address{Filippo De Mari , DIMA,  Via Dodecaneso 35, 16146 Genova, Italy}
\email{demari@dima.unige.it}
\author{Ernesto  De Vito}
\address{Ernesto  De Vito, DIMA,  Via Dodecaneso 35, 16146 Genova, Italy}
\email{devito@dima.unige.it}

\date{\today}
%\thanks{The authors were partially supported by PRIN 2007 ``Analisi
 % armonica su spazi metrici di misura, su gruppi di Lie, sullo spazio
 % delle fasi e su strutture discrete''. The second author was
 % partially supported by the FIRB project RBIN04PARL}
\date{\today}
\maketitle

\begin{abstract} 
We obtain necessary and sufficient conditions
for the admissible vectors of  a new  unitary non irreducible representation
 $U$. The group $G$ is an arbitrary semidirect
product whose normal factor  $A$ is abelian and whose homogeneous
factor $H$ is a locally compact second countable group acting on a
Riemannian manifold $M$. The key ingredient in the construction of $U$ 
 is a $C^1$ intertwining map between the actions of $H$ on the dual group $\hat A$ 
 and on $M$. The representation $U$ generalizes the restriction of the metaplectic 
 representation to triangular subgroups of $Sp(d,\R)$, whence the name ``mock metaplectic''.  For simplicity,  we content
ourselves with the case where $A=\R^n$ and $M=\R^d$. The main technical point
is the  decomposition of $U$ as direct integral of its irreducible components.
This theory is motivated by some recent developments in signal analysis, notably 
shearlets. Many related examples are discussed. 
\end{abstract}

%%%%%%%%%%%%%%%% END TOP MATTER %%%%%%%%%%%%%%%%%%%%%%%%%%%%%%%%%%%%

%%%%%%%%%%%%%%%%%%%%%%%%%%%%%%%%%%%%%%%%%%%%%%%%
%%%%%%%%%%%%% ACTUAL DOCUMENT BEGINS HERE %%%%%%%%%%%%%%%%

\section{Introduction} Unitary representations of semidirect products have been thoroughly studied  
by many authors and are useful in a wide variety of applications. 
% {\bf mettere REFS}.
In particular, they play a central  r\^ole in the harmonic analysis of the continuous wavelet transform, as discussed in \cite{fuhr05}. From the point of view of applications,  a unitary  representation $U$ of a locally compact group $G$ (with Haar measure $dg$) is particularly useful if it yields a reproducing formula, that is, a weak reconstruction of the form
\begin{equation}
f=\int_G
\scal{f}{U_g\eta}\;U_g\eta\;dg,
\label{RF}
\end{equation}
valid for every $f$ in the representation space $\cH$, for some {\it admissible} vector $\eta\in\cH$. In this case $(G,U,\eta)$ is called a {\it reproducing} system. Alternatively, we simply say that $G$ is a reproducing group. If $U$ is irreducible, this is nothing else but the classical concept of square integrable representation~\cite{dix64}, \cite{dumo76}. Typically, $\cH=L^2(\R^d)$, and in this case an admissible vector $\eta$ is sometimes called a {\it generating function} or {\it wavelet}. Apart from direct use, formula~\eqref{RF} is important also because it is the starting point for its discrete counterparts, an aspect that we shall not develop in the present paper.  It is actually rather interesting to observe that most formulae of the above type that appear in applications, either in their continuous or discrete versions, turn out to be expressible by taking the restriction of the metaplectic representation to some parabolic subgroup $G$ of the symplectic group $Sp(d,\R)$. This is the main theme in the papers \cite{codenota06a}, \cite{codenota06b}, \cite{codenota10} and the present contribution is an outgrowth thereof.

We will be concerned with groups $G$ that are semidirect products, where the normal factor is  an abelian group $A$  and the homogeneous factor is a locally compact second countable group $H$. Our main object of study is a unitary representation $U$ of $G$ whose construction is based on the following ingredients: a Riemannian manifold $M$ on which 
$H$  acts by $C^1$ diffeomorphisms and a $C^1$ map $\Phi: M\to \hat A$ (the dual group of $A$)
that intertwines the actions of $H$ on $M$ and on $\hat A$. The  representation $g\mapsto U_g$ acts   on $L^2(M)$ as pointwise multiplication by  the character $\langle \Phi(\cdot),g\rangle$  if $g\in A$ and  quasi regularly if $g\in H$, as clarified below in \eqref{mock}.  
For simplicity, we take $A=\R^n$ and $M=\R^d$ and we also suppose that the Jacobian of the
action on $\R^d$  is constant. We call $U$  the ``mock'' metaplectic representation because its definition is inspired by the case where $\R^n$ is a vector space of $d\times d$ symmetric matrices on which a closed subgroup $H$ of $GL(d,\R)$ acts by $\sigma\mapsto{}^th^{-1}\sigma h^{-1}$. Under these circumstances, $G$ can be identified with a triangular subgroup of $Sp(d,\R)$ and $U$ is the restriction to $G$ of the metaplectic representation (see Example~1).

General admissibility criteria for type-I groups have been proved in \cite{fuhr05}.  Given the representation $U$ on $\cH$, his theory stems from knowledge of a direct
integral decomposition  $U=\int_{\widehat{G}}m_{\sigma}\sigma\,d\nu(\sigma)$ into irreducible components, and the corresponding decomposition
$\cH=\int_{\widehat{G}}m_{\sigma}\cH_{\sigma}\,d\nu(\sigma)$.  With these data at hand, F\"uhr proves that if $G$ is non-unimodular, then \eqref{RF} holds true for some $\eta$  if and only if $\nu$ has density with respect to $\mu_{\widehat{G}}$, the Plancherel measure of $G$; if $G$ is unimodular, then one has to add the extra conditions
that $m_{\sigma}\leq\dim\cH_{\sigma}$ for $\nu$-almost every $\sigma$ and $\int_{\widehat{G}}m_{\sigma}\,d\nu(\sigma)<+\infty$.  Observe that the measure $\nu$ is known to exist \cite{dix64}, but one has to find it, together with the measurable field $\{\cH_\sigma\}$ and the multiplicity function $\sigma\mapsto m_\sigma$. The explicit knowledge of $\mu_{\widehat{G}}$ is also non trivial, in general, but is  understood for semidirect products 
\cite{klli72}.
Without using the remarkable machinery of \cite{fuhr05}, we explicitly decompose $U$ and thereby obtain, as a byproduct, computable admissibility criteria in terms of the intertwining map $\Phi$.

Our finer results  are Theorem~\ref{MAIN1} and Theorem~\ref{MAIN2}, which  deal with the cases where $G$ is unimodular or non-unimodular, respectively.  They both hold under the standard technical  assumption that the $H$-orbits are locally closed in $\Phi(\R^d)$ and assuming also that almost all $H$-stabilizers in $\Phi(\R^d)$ are compact. 
The latter assumption may be removed and yields the weaker conclusion given in Theorem~\ref{adm2}.  Theorem~\ref{MAIN2}  actually contains the following result:  if $G$ is  non-unimodular $U$ is reproducing if and only if  the set of critical points of $\Phi$ has Lebesgue measure zero. This is of course very easy to check in the examples in which $\Phi$ is explicitly known.
In the case where $n=d$ and where $\Phi$ is a homogeneous polynomial, circumstances that happen in many examples, then $U$ is reproducing if and only if $G$ is non unimodular and the stabilizers are almost all compact (see Theorem~\ref{neqd}). This last result settles the problem that was the original motivation of this work.

Here is an outline of the other results contained in the paper.
\begin{itemize}
\item Theorem~\ref{nleqd}, which establishes an important necessary condition for  a reproducing formula \eqref{RF} to hold true: $\Phi$ must map sets of positive measure into sets of positive measure, hence the critical points $\cC$ have zero Lebesgue measure and $n\leq d$. Thus we introduce an open $H$-invariant  subset $X$ of $\R^d$  with negligible Lebesgue complement  whose image is denoted by $Y=\Phi(X)\subseteq\R^n$ in such a way that $\Phi$ is a submersion of $X$ onto $Y$. 
The fibers $\Phi^{-1}(y)$ are therefore Riemannian submanifolds of $X$ and play a crucial r\^ole in what follows. All the results except Theorem~\ref{nleqd} will be formulated for $X$ and $Y$, namely, for the map $\Phi:X\to Y$, and hold true under the assumption that  $\cC$ has zero Lebesgue measure (see Assumption~1).

\item Theorem~\ref{corcoar}, based on the classical coarea formula, shows how the Lebesgue measure of $X$ disintegrates into a family of measures $\{\nu_y\}$ concentrated on the fibers  $\Phi^{-1}(y)$, whose covariance with respect to the $H$-action is explicitly calculated in \eqref{alfabeta}. 
\item Theorem~\ref{adm1}, where a first  reduction criterion for admissible vectors is given. One looks at the $H$-orbits in $Y$ and takes their preimages under $\Phi$ in $X$. Upon selecting an origin $y$ in each  $H$-orbit in $Y$, one gets  the fiber $\Phi^{-1}(y)$.  The theorem states that  it is necessary and sufficient to test that, for almost every  $H$-orbit in $Y$,  the $L^2$-norm with respect to $\nu_y$ of any $u\in L^2(X,\nu_y)$ can be reproduced by the (weighted) $H$-integral of  the square modulus  $|\scal{u}{\eta_y^h}_{\nu_y}|^2$ of the components of $u$  along the $H$-translates of the restriction to $\Phi^{-1}(y)$ of the admissible vector $\eta$. This is  formula~\eqref{FIRST}.
\item Theorem~\ref{intertwine}, which exhibits a direct integral decomposition of $U$ in terms of induced representations of isotropy subgroups of $H$, and is independent of any admissibility issue. This is achieved as follows.
\begin{itemize}
 \item
First of all, we assume that the $H$-orbits are locally closed in $Y$. This is  a standard assumption, without which most results in the current literature on these themes cannot be applied.
In Section~\ref{TD} we make some technical comments on this in relation to the recent results in \cite{fu09}.
\item
Secondly, we derive  a disintegration of the Lebesgue measure on $Y$ {\it \`a la Mackey}, that is, $dy=\int_{ Z}\tau_{z}\,d\lambda(z)$. Here $\lambda$ is a pseudo-image measure on the locally compact second countable space $ Z$  which is a nice parametrization of the orbits (better than $Y/H$) and $\tau_{z}$ is concentrated on the orbit corresponding to $z\in{ Z}$. This preliminary disintegration  is carried out  in Theorem~\ref{Tmackey}, where the covariance of $\{\tau_{z}\}$ with respect to the $H$-action is also calculated in \eqref{mackey_inv}.
\item
In Proposition~\ref{double} we use the measures $\{\tau_{z}\}$ in order  to ``glue'' together the measures $\nu_y$ for all $y$ in the same orbit, thereby producing new measures $\mu_{z}=\int_Y\nu_y\,d\tau_{z}(y)$ on $X$ which, in turn, allow to disintegrate the Lebesgue measure on $X$ as $dx=\int_{ Z}\mu_{z}\,d\lambda(z)$. As before, 
 the covariance of $\{\mu_{z}\}$ with respect to the $H$-action is calculated.
 The reason for introducing these measures are formulae \eqref{idemackey} and \eqref{idecoarea}: the representation space of $U$, namely $L^2(X)$, is formally the  double direct integral
 $$
 L^2(X)=\int_{ Z}\bigl(\int_Y L^2(X,\nu_y)\,d\tau_{z}(y)\bigr)
 \,d\lambda({z}),
 $$
 where the inner integral is $L^2(X,\mu_{z})$. 
\item
Next we show in Lemma~\ref{S} that $L^2(X,\mu_{z})$ is unitarily equivalent to the representation space $\cH_{z}$ of the representation $W_{z}$ which is unitarily induced to $G$ by the quasi regular representation of the stabilizer $H_{o(z)}$ (naturally extended to the  semidirect product $\R^n\rtimes H_{o(z)}$). Here it is important to select an origin $o(z)$ of the orbit in $Y$ whose label is $z$.
\end{itemize}

The conclusion of Theorem~\ref{intertwine} is that $U$ is equivalent to 
$\int_{ Z}W_{z}\,d\lambda({z})$, with an explicit intertwining isometry. The main technical ingredient of this part is the theory of disintegration of measures, as developed by Bourbaki
and it is reviewed in Appendix~\ref{DIS} under the simplifying assumption that the spaces are second countable.

%\item Theorem~\ref{adm2}, as mentioned earlier, is a general characterization of the admissible vectors.

\item Theorem~\ref{pazzini} assumes that the stabilizers of the $H$ action on $Y$ are almost all compact and it is based on the theory of von~Neumann algebras. It takes care of a nontrivial measurability issue involved in the decomposition of the map $z\mapsto W_{z}$ as direct sum of its irreducible components.
\end{itemize}

Finally, in Section~\ref{examples} we illustrate several examples.

\section{Notation and assumptions}
In this section we fix the notation and describe the setup.  We start by recalling the notions of {\em reproducing group} and {\em admissible vector}. For a thorough discussion the reader is referred to \cite{fuhr05}. 

Let $G$ be a locally compact group with (left) Haar measure $dg$ and $U$ be a strongly continuous unitary representation of $G$ acting on the complex separable Hilbert space $\cH$. A vector $\eta\in\cH$ is called
admissible if 
$$
\|f\|^2=\int_{G}|\scal{f}{U_g\eta}|^2\,
dg\qquad\text{for all }f\in\cH.
$$
If such a vector exists, we say that $G$ is a reproducing group and that $U$ is a {\em reproducing representation}.  Clearly, if $U$ is reproducing, then it is a cyclic representation, but in general
it is not irreducible.  When $U$ is irreducible, the representation is reproducing if and only if it is square integrable \cite{dumo76}.

\subsection{The semidirect product} Let $H$ be a locally compact second countable group acting on $\R^n$ by
means of the continuous representation
\begin{equation}
  \label{repn}
y\mapsto h[y],\qquad h\in H.
\end{equation}
Let $G$ be  the semidirect product $G=\R^n\rtimes H$ with  group law 
$$
(a_1,h_1)(a_2,h_2)=(a_1+h^\dagger_1[a_2],h_1h_2)\qquad
a_1,a_2\in\R^n,h_1,h_2\in H,
$$
where $h^\dagger[\cdot]$ is the action given by the contragredient representation of $H$ on $\R^n$ defined via the usual inner product  by 
\begin{equation}
\scal{h^\dagger[a]}{y}=\scal{a}{h^{-1}[y]},
\qquad a,y\in\R^n.
\label{contra}
\end{equation}
Since $h[\cdot]$ is linear, the semidirect product is well defined and $G$ is a locally compact second countable group. Conversely, any locally compact second countable group $G$ that is the semidirect product
of a closed subgroup  and a normal subgroup, which is a real vector space of dimension $n$, is of the above form.

The (left) Haar measures of $G$ and $H$ are written $dg$ and $dh$,
and, similarly, $da$ is the Lebesgue measure on $\R^n$.  The modular functions
of $G$ and $H$ are denoted by $\Delta_G$ and $\Delta_H$,
respectively. The following relations are easily established
\begin{align}
 \label{eq:1}
 dg & = \frac{1}{\alpha(h)}da\,dh \\
 \Delta_G(a,h) & = \frac{\Delta_H(h)}{\alpha(h)}\label{modularG}
\end{align}
where $\alpha:H\to (0,+\infty)$ is the character of $H$ defined by
\begin{equation}
  \label{eq:40}
  \alpha(h)  =|\det(a\mapsto h^\dagger[a])|=|\det(y\mapsto h^{-1}[y])|.
\end{equation}
The Fourier transform $\cF:L^2(\R^n)\to L^2(\R^n)$ is defined by
$$
(\cF f)(y)=\int_{\R^n} e^{-2\pi i \scal{y}{a}} f(a)\ da,
\qquad f\in L^2(\R^n)\cap L^1(\R^n).
$$
In general,  if $G$ is any locally compact second countable group, $L^2(G)$ will denote the Hilbert space of square integrable functions with respect to left Haar measure. Finally, if $X$ is a locally compact second countable  topological space, the Borel $\sigma$-algebra on $X$ is denoted $\cB(X)$ and $C_c(X)$ denotes the space of complex continuous functions on $X$ with compact support. By  {\it measure} we mean a $\sigma$-additive  function $\mu$ on $\cB(X)$ with values in $[0,+\infty]$ which is finite on compact sets. The hypothesis on $X$ implies that any such measure is automatically inner and outer regular \cite{lang95}. A function $f:X\to X'$ between two  such spaces will be called Borel measurable if $f^{-1}(B)\in\cB(X)$ for every $B\in\cB(X')$ and $\mu$-measurable if  
$f^{-1}(B)\in\cB_{\mu}(X)$, where $\cB_{\mu}(X)$ denotes  the completion of $\cB(X)$ with respect to $\mu$. When dealing with open subsets of Euclidean spaces endowed with the Lebesgue measure, however, we say measurable to mean Lebesgue measurable. Finally, if $E\in\cB(\R^d)$ we write $|E|_d$ for it Lebesgue measure or simply $|E|$ if no confusion arises.

\subsection{The mock metaplectic representation} Suppose we are given:
\begin{itemize}
\item[(H1)] a  continuous action  of $H$ on $\R^d$ by smooth maps denoted $x\mapsto h.x$, whose Jacobian is constant and equal to $\beta(h)$; for $h\in H$ and  $E\in\cB(\R^d)$ we thus have
\begin{equation}
  \label{beta}
  |h.E|= \beta(h)|E|,
\end{equation}
that is, for every $\varphi\in C_c(\R^d)$
\begin{equation}
  \label{betabeta}
\int_{\R^d}\varphi(h^{-1}.x)\,dx=\beta(h)\int_{\R^d}\varphi(x)\,dx.
\end{equation}

\item[(H2)] a $C^1$-map $\Phi:\R^d\to\R^n$ intertwining the two actions of $H$ on $\R^d$ and $\R^n$:
\begin{equation}
  \label{PHI}
  \Phi(h.x)=h[\Phi(x)]\qquad x\in\R^d,\,h\in H.
\end{equation}
\end{itemize}
For $g=(a,h)\in G$ we define $U_g:L^2(\R^d)\to L^2(\R^d)$ by 
\begin{equation}
(U_g f)(x) =\beta(h)^{-\frac{1}{2}} e^{-2\pi i \scal{\Phi(x)}{a}}f(h^{-1}.x)
\label{mock}
\end{equation}
for almost every $x\in\R^d$. We show below that this is indeed a representation, that we call the {\it mock metaplectic} representation. For a motivation for the choice of this name, see Example~\ref{eclass} below.

\begin{remark} The representation \eqref{repn} of $H$ on $\R^n$ plays
  no direct r\^ole in the definition of $U$; its purpose is to
  construct the semidirect product $G$.
\end{remark}

\begin{remark} Occasionally, we shall write $f^h(x)$ for $f(h^{-1}.x)$.\end{remark}

\begin{remark} At this stage there are no limitations on the relative sizes of $n$ and $d$, but we shall see later (Theorem~\ref{nleqd}) that in the situations that are of interest to us $n\leq d$.
\end{remark}
The next proposition records that \eqref{mock} is a good definition.
\begin{proposition}\label{mock_prop}
 The map $g\mapsto U_g$ is a strongly continuous unitary
 representation of $G$ acting on $L^2(\R^d)$. 
\end{proposition}
\begin{proof}
 Clearly, $U_g$ is a unitary operator and $U$ is a representation of  $\R^n$ and $H$ separately.  In order to prove that it is a representation of $G$, it is enough to show that 
 $U_hU_aU_{h^{-1}}=U_{h^\dagger[a]}$ for $a\in\R^n$ and $h\in H$. For  $f\in L^2(\R^d)$,
and almost every $x\in \R^d$
\begin{align*}
  \left(U_hU_a U_{h^{-1}} f\right)(x) 
& =\beta(h)^{-\frac{1}{2}}  e^{-2\pi i \scal{\Phi(h^{-1}.x)}{a}}\,(U_{h^{-1}} f)(h^{-1}.x) \\
& =  e^{-2\pi i \scal{\Phi(h^{-1}.x)}{a}}\, f(x) = e^{-2\pi i\scal{h^{-1}[\Phi(x)]}{a}}\, f(x) \\
& =  e^{-2\pi i \scal{\Phi(x)}{h^\dagger[a]}}\, f(x) = (U_{h^\dagger[a]} f)(x)
\end{align*}
To show strong continuity, it is enough to prove that $g\mapsto\scal{U_gf_1}{f_2}$ is continuous at the identity  whenever $f_1,f_2$  are continuous functions with compact
support, and this is an easy consequence of the dominated convergence theorem.
\end{proof}

\subsection{Examples} There are many interesting examples of the setup we are considering.
We will focus on some situations in which  most relevant features occur.

\begin{example}\label{fuhr} Let $H$ be a closed subgroup of $GL(d,\R)$ and assume $n=d$. Since the group $H$ acts naturally on $\R^d$, define 
$$
h.x= h[x]=\,^th^{-1}\,x \qquad x\in\R^d,\ h\in H.
$$
Choosing  $\Phi(x)=x$, the representation $U$ is equivalent  to the quasi
 regular representation of $G$ via the Fourier transform. Necessary and sufficient conditions
for $U$ to be reproducing are given in \cite{fuhr05}. It is worth observing that if $n=1$, then $H=\R_+$ and hence $G$ is  the ``$ax+b$'' group, whereby the dilations are parametrized by $H$. In this case $U$ is 
$$
U_{(b,a)}f(x)=\sqrt{a}e^{-2\pi ibx}f(ax)
$$
which,  after conjugation with the Fourier transform, is the usual wavelet  representation. It may be generalized to higher dimension \cite{LWWW}.
\end{example}

\begin{example}\label{scro} 
The Schr\"odinger representation of the  Heisenberg group ${\mathbb H}^1$ may be included in this setup, by regarding ${\mathbb H}^1$ as a closed subgroup of $GL(3,\R)$:
$$
{\mathbb H}^1
=\Bigl\{
\begin{bmatrix}
1&q&t\\
0&1&p\\
0&0&1
\end{bmatrix}
:q,p,t\in{\R}
\Bigr\}.
$$
It is easy to see that  ${\mathbb H}^1$ is isomorphic to the semidirect product $A\rtimes H$, where
$A=\bigl\{\left[\begin{smallmatrix}p\\t\end{smallmatrix}\right]:p,t\in{\mathbb R}\bigr\}$
and
 $H=\bigl\{\left[\begin{smallmatrix}1&0\\q&1\end{smallmatrix}\right]:q\in{\mathbb R}\bigr\}$.
Indeed, the group $H$ has the natural representation on $\R^2$:
$$
q\mapsto \;^t\!\begin{bmatrix}1&0\\q&1\end{bmatrix}^{-1}
=\begin{bmatrix}1&-q\\0&1\end{bmatrix}
$$
and acts on $\R$ via the  translations $q.x=x+q$.
The smooth map $\Phi:\R\to\R^2$ defined by 
$\Phi(x)=\left[\begin{smallmatrix} -x\\\;1\end{smallmatrix}\right]$
satisfies the intertwining property \eqref{PHI}. The mock metaplectic representation takes the form
$$
U_{(q,p,t)}\,f(x)
=
e^{-2\pi i\langle\Phi(x),\left[\begin{smallmatrix} p\\t\end{smallmatrix}\right]\rangle}
\hskip0.1truecmf(q^{-1}.\,x)
=e^{-2\pi i(t-px)}\hskip0.1truecmf(x-q)
$$
and it thus coincides with the Schr\"odinger representation, which is irreducible but notoriously not square integrable (i.e. not reproducing). Notice that $n>d$.
\end{example}

\begin{example}\label{eclass}
This class of examples is where our investigation started. It will be transparent that the mock metaplectic representation is a generalization of the metaplectic representation as restricted to this class of subgroups of $Sp(d,\R)$ which includes all the parabolic subgroups.
 Let $G=\Sigma\rtimes H\subset Sp(d,\R)$  be a subgroup of the form
\begin{equation}
G=\Bigl\{\begin{bmatrix}h & 0 \\\sigma h& ^th^{-1} \end{bmatrix}: h\in
  H,\ \sigma\in \Sigma\Bigr\},
\label{parabolic}
\end{equation}
where $H$ is a closed subgroup of $GL(d,\R)$ and $\Sigma$ is an
$n$-dimensional subspace of
$\text{Sym}(d,\R)$, the space of symmetric $d\times d$ matrices.  We call any such group a triangular subgroup.

Inner conjugation  within $G$ yields the $H$-action on $\Sigma$
\begin{equation}
h^\dagger[\sigma]:={^th}^{-1}\sigma h^{-1}\qquad \sigma\in\Sigma,h\in H,
\label{dagger1}
\end{equation}
under which $\Sigma$ must be invariant. As the notation suggests, \eqref{dagger1} can be seen as a contragredient action.
Indeed, we endow $\text{Sym}(d,\R)$ with  the natural  inner
product $\scal{\sigma_1}{\sigma_2}=\trace(\sigma_1\sigma_2)$,  whose restriction to $\Sigma$ will be denoted $\scal{\cdot}{\cdot}_{\Sigma}$. If $\sigma\mapsto h[\sigma]$ is the representation  whose contragredient version is \eqref{dagger1}, then for $\sigma,\tau\in\Sigma$ we have
$$
\scal{\tau}{h[\sigma]}_\Sigma
=\scal{^t\!h\tau h}{\sigma}_\Sigma
%=\trace(^t\!h\tau h\sigma)
=\trace(\tau h\sigma^t\!h)
=\scal{\tau}{P_\Sigma(h\sigma ^t\!h)}_\Sigma,
$$
where $P_{\Sigma}$ is the orthogonal projection from $\text{Sym}(d,\R)$
onto $\Sigma$. Thus
\begin{equation}
h[\sigma]=P_{\Sigma} (h \sigma\, {^th})\qquad \sigma\in\Sigma,h\in H,
\label{semidirect}
\end{equation}
and if ${^t}H=H$ there is no need of  the projection.

The group $H$ acts naturally on $\R^d$, that is, $h.x=hx$. Given $x\in\R^d$, let $\Phi(x)\in\Sigma$ be defined by
\begin{equation}
 \trace(\Phi(x)\sigma)=-\frac{1}{2}\,\scal{\sigma x}{x}
 \qquad x\in\R^d.
\label{riesz}
\end{equation}
Identifying $\R^n\simeq\widehat{\Sigma}\simeq\Sigma$, we can interpret $\Phi(x)$ either as the linear functional on $\Sigma$ whose action on $\sigma$ is $-\frac{1}{2}\,\scal{\sigma x}{x}$ or as the symmetric matrix associated to it via the usual inner product on symmetric matrices.
Condition~\eqref{PHI} is satisfied, since, upon observing that $\sigma=P_{\Sigma}(\sigma)$ and 
that $P_{\Sigma}$ is self-adjoint,
$$
\trace(\Phi(h.x)\sigma)
=-\frac{1}{2}\,\scal{{^th}\sigma h x}{x}
=\trace(\Phi(x) {^t}h\sigma h)
=\trace(h\Phi(x){^th}\sigma)
= \trace(h[\Phi(x)]\sigma).
$$
The representation~\eqref{mock} is 
\begin{equation}
U_{(\sigma,h)}f(x)=|\det h|^{-1/2}\;e^{\pi i\scal{\sigma x}{x}}f(h^{-1}x)
\label{metaplectic}
\end{equation}
and hence it coincides with the restriction of the metaplectic representation to the group $G$. 
Various properties of  $U$  are analyzed in \cite{codenota06b,codenota06a}.

An important  explicit example in this class  is connected to the theory of shearlets initiated in \cite{gukula06}. Here the group
$G$ parametrizes the  two-dimensional phase-space operations of translation, dilation and shear and is thus sometimes denoted $TDS(2)$. We shall do so and call it the shearlet group.

Precisely, $G=\R^2\rtimes H$ in the following way. 
Fix a parameter $\gamma>0$ (usually $\gamma=1/2$). 
The abelian normal subgroup $\Sigma\simeq\R^2$ consists of the $2\times2$ symmetric matrices
$\left[\begin{smallmatrix}a_1& a_2/2\\ a_2/2&0 \end{smallmatrix}\right]$. The homogeneous group $H$ contains all the $2\times2$ matrices of the form $S_\ell A_t$ where $\ell\in\R$, $t\in\R_+$ and
\[
S_\ell=\begin{bmatrix} 1& 0 \\-\ell &1\end{bmatrix}, \qquad 
A_t =\begin{bmatrix} t^{-\frac{1}{2}}& 0 \\0 &t^{\frac{1}{2}-\gamma}\end{bmatrix}
\]
with Haar measure $dh=t^{\gamma-2} d\ell dt$ and modular
function $\Delta_H(\ell,t)=t^{\gamma-1}$.
 For any $h=(\ell,t)$ the linear action on the abelian normal factor $\R^2$ is
\[
h^\dagger[\cdot]=\begin{bmatrix} 1&\ell\\0&1\end{bmatrix} \begin{bmatrix} t & 0\\0&t^\gamma\end{bmatrix} 
\]
and the group law of $G$ is
\[
(a,\ell,t)(a',\ell',t')=(a+
\left[\begin{smallmatrix}
t& t^{\gamma}\ell \\ 0  &t^{\gamma}
\end{smallmatrix}\right]a',\ell+t^{1-\gamma}\ell',tt').\]
It is easy to see that formula \eqref{riesz} implies that $\Phi(x_1,x_2)=-\frac{1}{2}(x_1^2,x_1x_2)$. The
mock metaplectic representation $U$ restricted to $\Sigma$ is
equivalent to translations and restricted to $\{A_t\}$ it amounts to
dilations, as shown in \cite{codenota06b}, where necessary and
sufficient conditions for admissible vectors are given in the case $\gamma=1$. 
Admissibility conditions are also given in \cite{dakustte09} for $\gamma=1/2$.
Observe that $d=n$. 

\end{example}

\begin{example}\label{n>d} This is a case where $n<d$. Let    $H=\R_+\times {\mathbb T}$. Here
${\mathbb T}$ is the one-dimensional torus, parametrized by $\theta\in [0,2\pi)$, with Haar measure $d\theta/2\pi$, and  $\R_+$ is the multiplicative group with Haar measure $t^{-1}dt$ where   $dt$ is the restriction to $\R_+$ of the Lebesgue measure on the real line. 
Hence $H$ has  Haar measure $dt \,d\theta/2\pi t$  and modular function $\Delta_H(h)=1$.
The representation of $H$  on $\R$ is
$$
h[y]=t^2 y\qquad y\in\R,
$$
where $h=(t,\theta)$. Hence in particular $\alpha(h)=t^{-2}$. The group law in $G=\R\rtimes H$  is
$$ 
(a_1,t_1,\theta_1)(a_2,t_2,\theta_2)= (a_1+t_1^{-2}
a_2,t_1t_2,\theta_1+\theta_2 ). 
$$
The resulting Haar measure is $t dt \,d\theta/2\pi$ and the modular function is easily seen to be 
$\Delta_G(a,t,\theta)=t^2$. The action of $h=(t,\theta)\in H$ on $\R^2$ is given by
$$
h.(x_1,x_2)=t (\cos\theta\, x_1-\sin\theta\, x_2, \sin\theta\,
x_1+\cos\theta\, x_2)\qquad (x_1,x_2)\in\R^2
$$ 
so that $\beta(h)=t^2$. Finally,  $\Phi:\R^2\to \R$ is given by
$\Phi(x_1,x_2)=x_1^2+x_2^2$.
The mock metaplectic representation $U$of $G$ on $L^2(\R^2)$ is
\begin{align*}
U_{(a,t,\theta)}f(x_1,x_2)&= t^{-1}e^{-2\pi i (x_1^2+x_2^2)a}\times\\
&\times f\left(t^{-1} (\cos\theta\, x_1+\sin\theta\, x_2), t^{-1}(-\sin\theta\, 
x_1+\cos\theta\, x_2)\right).
\end{align*}
\end{example}

\begin{example}\label{n<dnocom}
The point of this example, where again $n<d$, will become clearer later, when $H$-stabilizers enter into the picture: this is a case where they are not compact.
Let $H=\R^*\times \R$ where $\R^*$ is the (non-connected) multiplicative group of non-zero real numbers and
$\R$ is the additive group with Haar measures $\lvert t\rvert^{-1} dt$ and
$db$ respectively. The Haar measure of $H$ is $\lvert t\rvert^{-1}dtdb$ and $\Delta_H=1$.  An element $h=(t,b)\in H$ acts on $\R$ and $\R^2$ by
means of 
\begin{align*}
  h[y] & = t y && y\in \R \\
  h.(x_1,x_2)& =(x_1+b,t x_2)  && (x_1,x_2)\in\R^2
\end{align*}
so that $\alpha(h)=\lvert t\rvert^{-1}$ and $\beta(h)=\lvert
t\rvert$. Finally $\Phi:\R^2\to\R$ is defined by $\Phi(x_1,x_2)=x_2$,
which clearly satisfies~\eqref{PHI}.  
\end{example}

%\begin{example}\label{mome} MOMENTUM MAP.\end{example}

\section{Main results}\label{mainresults}

\subsection{Dimensional constraints}
Our first result, Theorem~\ref{nleqd}, states that if $G$ is reproducing, then $n\leq d$. The interpretation of this statement in the case of wavelets  is that the dimension of the space of translations cannot exceed that of the ``ground'' space.  In order to prove the theorem we need a technical lemma, in the proof of which we use a standard result in harmonic analysis on locally compact abelian groups (see Theorem~(31.33) in \cite{HeRossII79}).  This is the fact that if a bounded measure $\nu$ on the locally compact abelian group $\cG$ has Fourier transform that coincides almost everywhere (on the character group $\widehat{\cG}$) with the Fourier transform of an $L^p(\cG)$-function $F$, with $1\leq p\leq 2$, then $F\in L^1(\cG)$, $\nu$ is absolutely continuous with respect to Haar measure and its Radon-Nikodym derivative is $F$. We apply this to a bounded measure on $\R^n$.

\begin{lemma}\label{basic} For any  $f,\eta\in L^2(\R^d)$ the following facts are equivalent:
\begin{itemize}
\item[(i)] $\int_G |\scal{f}{U_g\eta}|^2\, dg<+\infty$;
\item[(ii)] for almost every $h\in H$ the bounded measure on $\R^n$ 
\begin{equation}
\Omega_h(E)=\int_{\Phi^{-1}(E)}
  f(x)\overline{\eta(h^{-1}.x)}\, dx,
  \qquad
E\in\cB(\R^n),
\label{Omega}
\end{equation}
has a density $\omega_{h}\in L^2(\R^n)$
for which
\begin{equation}
\int_H\left(\int_{\R^n}|\omega_{h}(y)|^2\,dy\right)\,\frac{dh}{\alpha(h)\beta(h)}<+\infty.
\label{finitenorm}
\end{equation}
\end{itemize}
Under the above circumstances
\begin{equation}
\int_G |\scal{f}{U_g\eta}|^2\, dg
=\int_H\left(\int_{\R^n}|\omega_{h}(y)|^2\,dy\right)\,\frac{dh}{\alpha(h)\beta(h)}.
\label{2norm}
\end{equation}
\end{lemma}
\begin{proof}
Observe that $\Omega_{h}$ is the image measure,  induced by $\Phi$, of the bounded measure with density $f\overline{\eta^{h}}\in L^1(\R^d)$ with respect to $dx$ (see e.g. Sec.~39 in \cite{Halmos50}).   Since $\Omega_{h}$ is bounded, the basic integration formula for image measures, (see Theorem~C, p.161 in \cite{Halmos50}) and~\eqref{mock}  imply that
\begin{align*}
\scal{f}{U_{(a,h)}\eta}
&=\beta^{-\frac{1}{2}}(h)\int_{\R^d} e^{2\pi i\scal{\Phi(x)}{a}}f(x)\overline{\eta^{h}}(x)\,dx \\
&=\beta^{-\frac{1}{2}}(h) \int_{\R^n} e^{2\pi i \scal{y}{a}} \,d \Omega_{h}(y).
\end{align*}
Assume that  $ \int_G |\scal{f}{U_g\eta}|^2\, dg<\infty$. Since $dg=\frac{da\,dh}{\alpha(h)}$, Fubini's theorem implies that, for almost every $h\in H$, 
$$
\int_{\R^n}|{\scal{f}{U_{(a,h)}\eta}|^2\,da
=\beta(h)^{-1}\int_{\R^n} | \int_{\R^n} e^{2\pi i \scal{y}{a}} \,d\Omega_{h}(y)}|^2\,da<+\infty.
  $$
This says that the inverse Fourier transform of $\Omega_{h}$ is in $L^2(\R^n)$, and the aforementioned Theorem~(31.33) in \cite{HeRossII79} ensures that the latter condition is equivalent to saying that 
$\Omega_{h}$ has  an $L^2(\R^n)$-density $\omega_{h}$ with respect to $dy$. Furthermore, by Plancherel's theorem
$$
\int_{\R^n} |\int_{\R^n} e^{2\pi i \scal{y}{a}} \,d\Omega_{h}(y)|^2\,da
=\int_{\R^n} |\omega_{h}(y)|^2\,dy.
$$
Applying again Fubini's theorem, \eqref{2norm} follows and hence \eqref{finitenorm} holds. Therefore (i) implies (ii). The converse statement is shown by applying the same argument backwards.
\end{proof}
We are now in a position to state our first result.
\begin{theorem}\label{nleqd}
If  $U$ is a reproducing representation, then the image under $\Phi$ of any Borel subset of $\R^d$ with positive measure has positive measure.  Hence
\begin{itemize}
\item[(i)] $n\leq d$;
\item[(ii)] the set $\cC$ of critical points\footnote{A  point $x\in\R^d$ is critical for $\Phi:\R^d\to\R^n$ if the rank of the differential map $\Phi_{*x}$ is less than $n$.} of $\Phi$ has measure zero.
\end{itemize}
\end{theorem}
\begin{proof}
By contradiction, suppose that there exists a Borel subset $A$ of $\R^d$ with positive measure such that $\Phi(A)$ is negligible. Since $|A|_d>0$ and the Lebesgue measure is regular, there exists a compact subset $K\subset A$ with  $|K|_d>0$. Clearly, 
 $\Phi(K)$ is also compact, but $|\Phi(K)|_n=0$. Take an admissible vector $\eta$ for $U$. The reproducing formula for $f=\chi_{K}$ and~\eqref{2norm} imply that
$$
0<|K|_d
=\int_H\left(\int_{\R^n}|\omega_{h}(y)|^2\,dy\right)\,\frac{dh}{\alpha(h)\beta(h)},
$$
so that, on a subset of $H$ of positive Haar measure we have $\omega_{h}\neq 0$.  Take then $h\in H$ such that  
$\Omega_{h}=\omega_{h} dy\neq0$. Now, if ${E}$ is a Borel subset of $\R^n$, the
definition of $\Omega_{h}$ gives
$$
\Omega_{h}({E})=\Omega_{h}({E}\cap\Phi(K))=\int_{{E}\cap\Phi(K)}\omega_{h}(y) dy=0
$$
because $|\Phi(K)|_n=0$. Hence $\Omega_{h}=0$, a contradiction.

To show (i), assume that $n>d$ and apply the above result to $A=\R^d$. Since $\Phi$ is of class $C^1$ we have $|\Phi(A)|_n=0$, so that $U$ cannot be reproducing. 

To show (ii),  denote by $\cC$ the set  of critical points of $\Phi$. Sard's theorem \cite{st64}
 implies that $\Phi(\cC)$ has measure zero. But then, by (i), also $\cC$ has measure zero. 
\end{proof}

\subsection{Measures concentrated on the preimages under $\Phi$}
Given any $x\in\R^d$, let  $J(\Phi)(x)=\sqrt{\det(\Phi_{*x}\cdot {^t}\Phi_{*x})}$ be the Jacobian of  $\Phi$ at $x$ and denote by $\cR$ the set of regular points of $\Phi$, namely
$$
\cR=\left\{x\in \R^d: J(\Phi)(x)>0\right\}=\R^d\setminus\cC.
$$

\begin{lemma}\label{regular} The set $\cR$  satisfies the following properties:
\begin{itemize}
\item[(i)] it is open;
\item[(ii)] it is $H$-invariant and has $H$-invariant image under $\Phi$;
\item[(iii)] the restriction of $\Phi$ to it is an open mapping;
\item[(iv)] for every $y$ in its image,  $\Phi^{-1}(y)$ is a Riemannian submanifold of $\R^d$;
\item[(v)] a subset $E\subset \Phi(\cR)$ is negligible if and only
if $\Phi^{-1}(E)$ is negligible.
\end{itemize}
\end{lemma}
\begin{proof}
(i) Since $\Phi$ has continuous derivatives,  $\cR$ is an open
set. (ii) The $H$-invariance follows from
\begin{equation}
\Phi_{*h.x}(h_*.v)=h[\Phi_{*x}v],
\qquad
x,v\in\R^d,
\label{Hinvariance}
\end{equation}
where $h_*$ denotes the differential of the action $x\mapsto h.x$ and is therefore linear.
Indeed,  \eqref{Hinvariance} and the fact that $u\mapsto h[u]$ is a linear isomorphism,  show that $v\in\ker\Phi_{*x}$ if and only $h_*.v\in\ker\Phi_{*h.x}$, so that $\dim\ker\Phi_{*x}=\dim\ker\Phi_{*h.x}$. Since  $x\in \cR$ if and only if $\dim\ker\Phi_{*x}=d-n$, the claim follows. To prove \eqref{Hinvariance}, fix $x\in\R^d$, a tangent vector $v\in T_x(\R^d)\simeq\R^d$ and a smooth curve $v(t)$ passing through $x$ at time zero with tangent vector $v$. Evidently, $h.v(t)$ is smooth and has tangent $h_*.v$ at time zero. By \eqref{PHI} and again by the linearity of $u\mapsto h[u]$
\begin{align*}
\Phi_{*h.x}(h_*.v)&=\frac{d}{dt}\Phi(h.v(t))\Bigr|_{t=0} 
=\frac{d}{dt}h[\Phi(v(t))]\Bigr|_{t=0} \\
&=h[\frac{d}{dt}\Phi(v(t))\Bigr|_{t=0}]
=h[\Phi_{*x}v],
\end{align*}
as desired.

Finally, (iii)
and (iv) are standard consequences of  the fact that, by definition of
$J(\Phi)$, the differential $\Phi_{*x}$ is surjective whenever
$x\in\cR$.\\
In order to prove (v), put  $X=\cR$ and $Y=\Phi(\cR)$.  Since $\Phi$ is a submersion from $X$ 
onto $Y$ and since $X$ is locally compact second
  countable space, there exists a countable family of diffeomorphisms
  $\Psi_i:U_i\times V_i\to W_i$ such that $\{W_i\}$ is an open
  covering of $X$, $\{V_i\}$ is an open
  covering of $Y$, $\{U_i\}$ is a family of  open sets of $\R^{d-n}$
  and
\begin{equation}
\Phi(\Psi_i(z,y))=y \qquad (z,y)\in U_i\times V_i.\label{projection}
\end{equation}
Assume that $E$ is a Borel subset of $Y$, then $\lvert\Phi^{-1}(E)\rvert_d=0$
 if and only if $\lvert\Phi^{-1}(E)\cap W_i\rvert_d=0$  for
all $i$.  Since $\Psi_i$ is a diffeomorphism, by the chain rule this is equivalent to 
$\lvert\Psi_i^{-1}(\Phi^{-1}(E)\cap W_i)\rvert_d=0$, that
is, by~\eqref{projection}, $U_i\times (E\cap V_i)$ is a negligible set
of $\R^{d-n}\times \R^n$. Since $U_i$ is an open non-void set, the last condition is equivalent to 
$\lvert E\cap V_i\rvert_n=0$, that is, to $\lvert E\rvert_n=0$.
\end{proof}

\vskip0.2truecm
\noindent
{\sc Assumption 1.}  Motivated by  Theorem~\ref{nleqd}, in the
following we assume that $\cC$ has Lebesgue
measure zero. In particular, we assume that $n\le d$. Furthermore, we
fix an open $H$-invariant subset $X$ of $\cR$ whose complement also
has measure zero and we denote by $Y$ its image under $\Phi$, namely
$Y=\Phi(X)$. Clearly,  $X$ satisfies the properties (i)--(v)
described in Lemma~\ref{regular} and so its complement is negligible.
\vskip0.2truecm

The next results are based on several kinds of disintegration formulae and their covariance properties with respect to the $H$-action. In Section~\ref{DIS} we review the general theory of disintegration of measures and introduce the pertinent notation. As for the induced $H$-action on measures, and the resulting covariance properties, we recall that,  if $\nu$ is a  measure on $X$ and $h\in H$,  $\nu^h$ is the  measure given by $\nu^h(E)=\nu(h.E)$ whenever  $E\in\cB(X)$. Equivalently,
\begin{equation}
  \label{shiftmeasure}
  \int_X \varphi(x)\,d\nu^h(x)=\int_X \varphi(h^{-1}.x)\,d\nu(x)
\end{equation}
for every $\varphi\in C_c(X)$.
The first disintegration we discuss arises from the Coarea Formula for submersions. 
\begin{theorem}\label{corcoar} There exists a unique family $\{\nu_ y\}$ of  measures on $X$, labeled by the points of $Y$, with  the following properties: 
\begin{itemize}
\item[(i)]   $\nu_y$  is concentrated on  $\Phi^{-1}( y)$ for all $y\in Y$;
\item[(ii)]  $dx=\int_{Y}\nu_{y}dy$;
\item[(iii)] for any $\varphi\in C_c(X)$ the map
$y\mapsto \int_X \varphi(x)\, d\nu_{y}(x)\in\C$
is continuous.
\end{itemize}
Furthermore, 
\begin{equation}
    \label{alfabeta}
    \nu_{h[y]}^h= \alpha(h)\beta(h)\,\nu_{y}
  \end{equation}
  for all $h\in H$ and all $y\in Y$.
\end{theorem}
\begin{proof}
The proof is based on the classical Coarea Formula. In  Section~\ref{CF} we give a short proof adapted to the situation at hand and we introduce the notation used  in this proof. The reader is thus referred to Theorem~\ref{coar-form} below.

For every $y\in Y$, define $\nu_y$ by  \eqref{eq:20}.  Property (i) is then  obvious and (ii) is  the content of  Theorem~\ref{coar-form}.
 
To prove (iii), fix $\varphi\in C_c(X)$ and $y_0\in Y$. If $y_0\not\in\Phi(\supp{\varphi})$, there is an open neighborhood $V$ of $y_0$ such that $V\cap \Phi(\supp{\varphi})=\emptyset$. Thus 
$\int_X \varphi(x)\, d\nu_{y}(x)=0$ for all $y\in V$ because $\nu_y$ is concentrated on $\Phi^{-1}(y)$.
If $y_0\in\Phi(\supp{\varphi})$, taking a finite covering if necessary, we can
always assume that there exists a diffeomorphism $\Psi:U\times V \mapsto W$ such that
\eqref{eq:10} holds, where $U$ is an open subset of $\R^{d-n}$,  $ V$ is an open neighborhood  of $y_0$ and $W$ is an open subset of $X$ containing $\supp{\varphi}$. The definition of $\nu_y$ gives
$$
\int_X \varphi(x)\, d\nu_{y}(x)=\int_{U} \varphi(\Psi(z,y)) (J\Psi)(z,y)\,dz
$$
and the map 
$y\mapsto\int_{U} \varphi(\Psi(z,y)) (J\Psi)(z,y)\,dz$
is continuous on $V$ by the dominated convergence theorem. 

In order to show~\eqref{alfabeta}, fix $h\in H$. 
Since the action of $H$ on $X$ is continuous,  $\{\nu^h_{h[y]}\}_{y\in Y}$ is a family of
measures on $X$ and each of them is concentrated on $\Phi^{-1}(y)$,
as shown by
\[
\nu^h_{h[y]}(X\setminus\Phi^{-1}(y))=\nu_{h[y]}(X \setminus
\Phi^{-1}(h[y]))=0,
\]
where the last equality is due to (i). Furthermore the family  $\{\nu^h_{h[y]}\}_{y\in Y}$ is scalarly integrable with respect to
$dy$ because for all $\varphi\in C_c(X)$
\begin{align*}
 \int_Y \bigl(\int_X \varphi(x)\,d\nu_{h[y]}^h(x) \bigr)dy & = \int_Y
 \bigl(\int_X \varphi(h^{-1}.x)\,d\nu_{h[y]}(x) \bigr)dy  \\
(~y\mapsto h^{-1}[y]~) & = \alpha(h) \int_Y
 \bigl(\int_X \varphi(h^{-1}.x)\,d\nu_{y}(x) \bigr)dy  \\
& = \alpha(h) \int_X \varphi(h^{-1}.x)\,dx \\
(~x\mapsto h.x~)  & =\alpha(h)\beta(h) \int_X \varphi(x)\,dx,
\end{align*}
where the third line  follows from (ii). Hence
\[
dx=\int_X \alpha(h^{-1})\beta(h^{-1}) \nu_{h[y]}^h\,dy
\] 
and (iv) of Theorem~\ref{th:2} implies that for almost all $y\in Y$~\eqref{alfabeta}
holds true. Item (iii) tells us that for any fixed $\varphi\in C_c(X)$, the mappings
$y\mapsto\int_X\varphi(x)\,d\nu_y(x)$ and $y\mapsto\int_X\varphi(x)\,d\nu^h_{h[y]}(x)$
are continuous and hence the almost everywhere equality is really an equality.
\end{proof}

In view of the previous result, we may apply the theory developed in Section~\ref{DI}.
In particular we obtain  \eqref{L2} in the case in which $\omega$  and
$\rho$ are the Lebesgue measures:  
\begin{equation}
L^2(X)=\int_Y L^2(X,\nu_y)dy,
\qquad
f=\int_Y f_y dy.
\label{Leb2}
\end{equation}
Here the equalities must be interpreted in $M(X)$ and the second integral is a scalar integral relative to the duality of $M(X)$ and $C_c(X)$. For a discussion of the details see the Appendix, where it is also explained that in particular
\begin{equation}
  \label{coareB}
  \|f\|^2=  \int_{Y} \|f_y\|_{\nu_y}^2\, dy.
\end{equation}

One of the reasons for introducing the measures $\{\nu_y\}$ is because, via the coarea formula, they provide a very useful description of the density $\omega_h$ discussed in Lemma~\ref{basic}. 
\begin{corollary}\label{density}
Given $f,\eta\in L^2(X)$, the function
$
y\mapsto
%\omega_h(y)
%=\int_{X}f(x)\bar\eta(h^{-1}.x)\,d\nu_y(x)
\scal{f_y}{\eta_y^h}_{\nu_y}
$
coincides almost everywhere with the density $\omega_h$ of the measure  $\Omega_h$ defined by \eqref{Omega}.
\end{corollary}
\begin{proof}
Item (iii) of Theorem~\ref{th:2}, together with Theorem~\ref{corcoar}, applied to $f\bar\eta\in L^1(X)$ and any $\xi\in C_c(Y)$ gives 
$$
\int_X\xi\left(\Phi(x)\right)f(x)\bar\eta(h^{-1}.x)\,dx
=\int_Y\xi(y)\int_{X}f(x)\bar\eta(h^{-1}.x)\,d\nu_y(x)\,dy.
$$
The left hand side is nothing else but the integral
$
\int_Y\xi(y)\,d\Omega_h(y)
$
because $\Omega_h$ is the image measure,  induced by $\Phi$, of  $f\overline{\eta^{h}}\,dx$. The corollary follows.
\end{proof}

\subsection{Reduction to fibers}\label{fibers} Much of our analysis stems from decomposing the representation space $L^2(\R^d)$ in terms of the measures $\{\nu_y\}$, and from 
a rather detailed understanding of the $H$-action on $Y$.  We thus introduce the usual notation for group actions: if $y\in Y$, then $H_y$ is the stabilizer of $y$, $H[y] =\left\{h[y]:h\in H\right\}$ is the corresponding orbit and $Y/H$ the orbit space. At this stage we   need a hypothesis ensuring that the $Y/H$ is not a pathological measurable space. It is worth mentioning that this hypothesis is satisfied in all the significant examples that we are aware of.  Below we further comment on this.

\vskip0.2truecm
\noindent {\sc Assumption 2.} We assume that for every $y\in Y$ the
$H$-orbit $H[y]$ is locally closed in $Y$, i.e., that it is open in
its closure or, equivalently, that $H[y]$ is the intersection of an
open and a closed set.
\vskip0.2truecm

The above assumption is not enough to guarantee that the orbit space $Y/H$ is a Hausdorff space, hence locally compact, with respect to the quotient topology. However, it is possible to bypass this topological obstruction by choosing a different parametrization of the $H$-orbits of $Y$. Indeed, a result of Effros (Theorem~2.9 in \cite{effros65}) shows that Assumption~2 is equivalent to the fact that the orbit space $Y/H$ is a standard Borel space. Hence there is a locally compact second countable space $ Z$ and a Borel measurable (hence Lebesgue measurable) map $\pi:Y\to  Z$ such that $\pi(y)=\pi(y')$ if and only if $y$ and $y'$ belong to the same orbit. To see this, observe that, by definition of standard Borel space, $Y/H$ with the quotient $\sigma$-algebra is Borel isomorphic to a Borel subset of a Polish space $Z$. By Kuratowski's theorem \cite{kur92}, we may assume that $Z=[0,1]$. Define $\pi(y)=i(\dot y)$, where $\dot y$ is the equivalence class of $y$ in $Y/H$ and $i$ is the Borel isomorphism of $Y/H$ into $[0,1]$.

In the following we fix the space $ Z$ whose points will label the orbits of $Y$ and we choose on $ Z$ a pseudo-image measure\footnote{\label{pseudo}
It is a measure on $ Z$ whose sets of measure zero are exactly the sets whose preimage with respect to $\pi$ have measure zero in $Y$. It always exists since $Y$ is $\sigma$-compact: it is enough to take first a finite measure on $Y$ equivalent to the Lebesgue measure (just choose a positive $L^1$ density), and then to consider the image measure on $ Z$ induced by $\pi$ (see e.g. Chap.~VI, Sect.~3.2 in~\cite{bourbaki_int}).} $\lambda$ of the Lebesgue measure under the map $\pi$. We note that $\lambda$ is concentrated on $\pi(Y)$ and a subset $E$ is $\lambda$-negligible if and only if $\lvert\pi^{-1}(E)\rvert_n=0$, which is equivalent to  $\lvert(\pi\circ\Phi)^{-1}(E)\rvert_d=0$ (item (v) in Lemma~\ref{regular}).

\begin{theorem}\label{adm1} The following  facts are equivalent:
\begin{itemize}
\item[(i)] the vector $\eta\in L^2(\R^d)$ is admissible for $U$;
\item[(ii)] for $\lambda$-almost every $z\in  Z$,  there exists a point 
$y\in\pi^{-1}(z)$ such that
\begin{equation}
\|u\|_{\nu_y}^2=\int_H|\scal{u}{\eta_y^h}_{\nu_y}|^2\frac{dh}{\alpha(h)\beta(h)},
\qquad
u\in L^2(X,\nu_y).
\label{FIRST}
\end{equation}
\end{itemize}
If  ~\eqref{FIRST} holds true for $y$, then it holds true for every point in $H[y]$.
\end{theorem}
\begin{proof} Given $\eta\in L^2(X,dx)$, write $\eta=\int_Y \eta_y\,dy$ where
$\eta_y\in L^2(X,\nu_y)$.  Fix $y\in Y$ and put
\[
\cD_y=\{ u\in L^2(X,\nu_y):
\int_H|\scal{u}{\eta_y^h}_{\nu_y}|^2\frac{dh}{\alpha(h)\beta(h)}<+\infty\}.\]
The map $\cW_y: \cD_y\to L^2(H, \alpha(h^{-1})\beta(h^{-1})dh)$,
defined by $(\cW_y u)(h)= \scal{u}{\eta^{h}_y}$ for almost all $h\in H$, 
is a closed linear operator (the proof is
standard \cite{dumo76}). Hence it is enough to prove~\eqref{FIRST} for a
dense countable subset of $L^2(X,\nu_y)$.  
Hence we fix a  countable family of functions $\{\varphi_\ell\}$ in $C_c(X)$ with
the following property: given an arbitrary $\varphi\in C_c(X)$, there
exists a subsequence
$(\varphi_{\ell_k})_{k\in\N}$ such that  
\begin{equation}
|\varphi_{\ell_k}|\leq|\varphi_0|,
\qquad
\lim_{k\to\infty}\sup_{x\in X}|\varphi_{\ell_k}(x)-\varphi(x)|=0.
\label{separable}
\end{equation}
The existence of such a family is clarified in Footnote~\ref{separa}
in the Appendix.  Clearly, for any $y\in Y$, the family
$\{\varphi_\ell\}$ is dense in $L^2(X,\nu_{y})$. \\
Assume that  $U$ is reproducing and take an admissible $\eta\in
L^2(X)$. For any $\ell$ we thus have
$$
\int_G |\scal{\varphi_\ell}{U_g\eta}|^2\, dg=
\int_X|\varphi_\ell(x)|^2\,dx=
\int_Y\bigl(\int_{X}|\varphi_\ell(x)|^2\,d\nu_y(x)\bigr)\,dy,
$$
the latter being a consequence of the coarea formula \eqref{coareB}. 
By Lemma~\ref{basic} the measure $\Omega_h^\ell$ in~\eqref{Omega}
has an $L^2$-density $\omega_h^\ell$ for almost every $h\in H$ and
formula~\eqref{2norm} holds true; furthermore, Corollary~\ref{density} tells us
that $\omega_h$ can be expressed in terms of the measures
$\{\nu_y\}$. Therefore  
\begin{align*}
\int_Y\int_{X}|\varphi_\ell(x)|^2\,d\nu_y(x)\,dy
%&=\int_X|f(x)|^2\,dx\\
&=\int_G |\scal{\varphi_\ell}{U_g\eta}|^2\, dg\\
%\eqref{2norm}
&=\int_H\bigl(\int_{Y}|\omega_{h}^\ell(y)|^2\,dy\bigr)
\,\frac{dh}{\alpha(h)\beta(h)}\\
%\eqref{density}
&=\int_H\bigl(\int_{Y}|\scal{\varphi_\ell}{\eta_y^h}_{\nu_y}|^2\,dy\bigr)
\,\frac{dh}{\alpha(h)\beta(h)}\\
&=\int_{Y}\bigl(\int_H|\scal{\varphi_\ell}{\eta_y^h}_{\nu_y}|^2\bigr)
\,\frac{dh}{\alpha(h)\beta(h)}\,dy,
\end{align*}
where in the last line we have applied Fubini's theorem. 
Let $N_\ell\subset Y$ be the set of $y\in Y$ where the equality
\begin{equation} 
\|\varphi_\ell\|_{\nu_y}^2=\int_H|\scal{\varphi_\ell}{\eta_y^h}_{\nu_y}|^2\frac{dh}{\alpha(h)\beta(h)}.
\label{FIRSTell}  
\end{equation}
does not hold. Reasoning
as in the proof of Corollary~\ref{density}, the equality of the first
and last term of the above string  is equivalent to saying that $N_\ell$ is negligible. \\
Put $N=\cup_{\ell}N_{\ell}$,   a negligible set.  For any
$y\not\in N$, \eqref{FIRSTell}~shows that
$\{\varphi_\ell\}\subset \mathcal D_y$ and $\mathcal W_y$
is an isometry on this dense subset. Since $\mathcal W_y$ is a closed
operator, it follows that $\mathcal D_y=L^2(X,\nu_y)$
and~\eqref{FIRST} holds true for every $u\in L^2(X,\nu_y)$. \\
Now, $N$ is the set consisting of those $y\in Y$ for which  the equality~\eqref{FIRST} does
not hold for at least a $u\in L^2(X,\nu_y)$.  We show that  $N$
is $H$-invariant. Take  $h\in H$ and 
$y\not\in N$.  For any $\varphi\in C_c(X)$,   both $\varphi$
and $\varphi^{h^{-1}}$ are in $L^2(X,\nu_y)$. Hence~\eqref{FIRST}  does hold  for  $u=\varphi$ and $u=\varphi^{h^{-1}}$.
Using~\eqref{shiftmeasure} and \eqref{alfabeta}, we obtain
\begin{align*}
\int_X|\varphi(x)|^2\,d\nu_{h[y]}(x)
&=\int_X|\varphi(h.x)|^2\,d\nu^h_{h[y]}(x)\\
&=\int_X|\varphi(h.x)|^2\,\alpha(h)\beta(h)d\nu_y(x)\\
&=\alpha(h)\beta(h)
\int_H\Bigl|\int_X \varphi(h.x)\bar\eta(k^{-1}.x)\,d\nu_y(x)\Bigr|^2\frac{dk}{\alpha(k)\beta(k)}\\
(h.x=z)\hskip0.2truecm
&=\alpha(h)\beta(h)
\int_H\Bigl|\int_X \varphi(z)\bar\eta((hk)^{-1}.z)\,d\nu^{h^{-1}}_y(z)\Bigr|^2\frac{dk}{\alpha(k)\beta(k)}\\
(hk=s)\hskip0.2truecm
&=\alpha^2(h)\beta^2(h)
\int_H\Bigl|\int_X \varphi(z)\bar\eta(s^{-1}.z)\,d\nu^{h^{-1}}_y(z)\Bigr|^2
\frac{ds}{\alpha(s)\beta(s)}\\
&=\int_H\Bigl|\int_X \varphi(z)\bar\eta(s^{-1}.z)\alpha(h)\beta(h)
\,d\nu^{h^{-1}}_{y}(z)\Bigr|^2
\frac{ds}{\alpha(s)\beta(s)}\\
&=\int_H\Bigl|\int_X \varphi(z)\bar\eta(s^{-1}.z)\,d\nu_{h[y]}(z)\Bigr|^2
\frac{ds}{\alpha(s)\beta(s)}\\
&=\int_H|\scal{\varphi}{\eta^s}_{\nu_{h[y]}}|^2\,\frac{ds}{\alpha(s)\beta(s)},
\end{align*}
that is, $h[y]\not\in N$, as desired. 
Finally, since $N$ is $H$-invariant and
negligible,  $\pi(\cup_{\ell}N_{\ell})$ is $\lambda$-negligible and
(ii) follows.

The fact that (ii) implies that $U$ is reproducing is proved by reversing the argument.
\end{proof}
\begin{remark}
Since $\pi$ induces a Borel isomorphism between the orbit space $Y/H$ and $\pi(Y)$, 
in the above statement and in the theorems of the following section it would be  possible to avoid the  space $ Z$ by considering on $Y/H$ a $\sigma$-finite measure defined on the quotient $\sigma$-algebra, which, by Assumption~2 (Theorem~2.9 in \cite{effros65}), coincides with  the Borel $\sigma$-algebra induced by the quotient topology. However, this measure could fail to be  finite on compact subsets.
\end{remark}

\subsection{Disintegration formulae}\label{sec-ass2}
 Our next result, Theorem~\ref{adm2}, is based on some classical  formulae that allow both a geometric interpretation of the integral~\eqref{FIRST} and a computational reduction that in the known examples is indeed significant. 
This is inspired by the irreducible case, where it is known that  $U$ is reproducing (i.e. square integrable) if and only if the $H$-orbit, unique by irreducibility, has full measure and the inducing representation of the stabilizer $H_y$ is square integrable~\cite{ernie98}.

We allude to  formulae that express an integral over $Y$ as a double integral,  first along the single $H$-orbits and then   with respect to the measure $\lambda$ on the space $Z$. Although these kinds of formulae can be traced back to Bourbaki~\cite{bourbaki_intII} and Mackey~\cite{mackey52}, perhaps one of the most famous occurrences of such a disintegration procedure appears in the celebrated paper of Kleppner and Lipsman \cite{klli72}; for a recent review see~\cite{fu09}. Much in the same spirit, we shall also need to decompose integrals over $H$ by integrating along a closed  subgroup $H_0$ first, and then over the homogeneous space $H/H_0$, which we identify with a suitable orbit of $Y$. 
The topological hypothesis formulated in Assumption~2 is needed in order that these decomposition formulae  can be safely applied.

Recall that in the beginning of Section~\ref{fibers} we fixed a space
$ Z$ that labels the orbits of~$Y$ and a measure $\lambda$ on $ Z$
whose null sets are in one-to-one correspondence with the
$H$-invariant null sets of $Y$.

\begin{theorem}\label{Tmackey} There exists a  family $\{\tau_ {z}\}$ of  measures on $Y$, labeled by the points of ${ Z}$, with  the following properties: 
\begin{itemize}
\item[(i)]   $\tau_ {z}$  is concentrated on  $\pi^{-1}(z)$ for all $z\in{ Z}$;
\item[(ii)]  $dy=\int_{ Z} \tau_{z}\,d\lambda(z) $.
\end{itemize}
Furthermore, for  almost every $z\in{ Z}$ the measure $\tau_z$ is relatively invariant and 
\begin{equation}
     \label{mackey_inv} 
    \tau_{z}^h=\alpha(h)^{-1}  \tau_{z}
  \end{equation}
  holds for every $h\in H$.
The family $\{\tau_{z}\}$ is unique in the sense that if
$\{\tau_{z}'\}$ is another family satisfying {\rm (i)} and {\rm (ii)}, then $\tau_{z}'=\tau_{z}$ for almost every ${z}\in{ Z}$.
\end{theorem}
\begin{proof}
The content of the theorem can be found in many different papers, such as
Lemmas~11.1 and 11.5 in \cite{mackey52} and Theorem~2.1 of
\cite{klli72}, in slightly different contexts.  The cited results are both
based on Bourbaki's treatment of disintegration of measures.  Here we simply adapt this theory to our setting.\\
Theorem~2 Ch.VI \S~3.3 of
\cite{bourbaki_int} yields a family $\{\tau_{z}\}$ of measures on $Y$ labeled by the points $z\in  Z$, unique in the sense of the statement, such that 
\begin{itemize}
\item $\tau_z\neq 0$ if and only if $z\in\pi(Y)$
\item $\tau_z$ is concentrated on $\pi^{-1}(z)$
\item $dy=\int_ Z \tau_z d\lambda(z)$.
\end{itemize}
The proof of Lemma~11.5 in \cite{mackey52}  shows, under the circumstances that we are considering,  that for almost all $z\in Z$  \eqref{mackey_inv} holds true for all $h\in
  H$; the density appearing in Lemma~11.4 of \cite{mackey52} is precisely
  $\alpha^{-1}$. 
\end{proof}

\subsubsection{A topological detour}\label{TD} Assumption~2 is needed
in order to prove Theorem~\ref{Tmackey} because we apply results on
disintegration of measures that use it, as developed
in~\cite{bourbaki_int}. The same theorem actually holds under the
(weaker) conditions that are described in the proposition below. Their
equivalence does not seem to be a known fact. In \cite{fu09},
Theorem~12, it is shown that (ii) in Lemma~\ref{toplemma} below is a
necessary condition for the disintegration in Theorem~\ref{Tmackey} to
hold true. In the next statement ${\hat\pi}$ denotes the canonical
projection from $Y$ onto $Y/H$.

\begin{proposition}\label{toplemma}
The following two conditions are equivalent:
  \begin{enumerate}
\item[(i)]  there exists an increasing sequence of compact subset $\{K_n\}$ of
    $Y$ such that the complement of $\cup K_n$ is Lebesgue negligible
    and ${\hat\pi}(K_n)$ endowed
    with the relative topology is a Hausdorff space; 
\item[(ii)] there exists an $H$-invariant null set $N\subset Y$ such that
  $(Y\setminus N)/H$ is a standard Borel space with respect to the $\sigma$-algebra induced by $\hat\pi$. 
\end{enumerate}
\end{proposition}
\begin{proof}
  First we show that (i) implies (ii).  Denote by $R$ the equivalence
  relation induced by the action of $H$ on $Y$, that is, $y\sim_R y'$ if
  and only if ${\hat\pi}(y)={\hat\pi}(y')$. \\
{\bf Claim 1:} there exists a  Lebesgue measurable map $p$ from $Y$ into
a locally compact second countable space 
  $\Omega$  with the property
  \begin{equation}
p(y)=p(y') \qquad \iff\qquad  y\sim_R y'\label{comprato}.
\end{equation}
By assumption
  for each $n$ the space ${\hat\pi}(K_n)$ is Hausdorff and, by Prop.~3 Ch.1 \S~5.3 of
      \cite{bourbaki_GT}, this is  equivalent to the fact the quotient
      space $K_n/R_n$ is Hausdorff with respect to the  quotient
      topology, where $R_n$ the restriction of $R$ 
      to $K_n\times K_n$.  Since $Y$ is $\sigma$-compact, the above property implies that 
      $R$ is a Lebesgue measurable  equivalence relation according to the definition in
  Ch.~VI \S~3.4 of \cite{bourbaki_int}. By Proposition~2 Ch.~VI
  \S~3.4 of \cite{bourbaki_int} there exists a map $p:Y\to\Omega$
  with the desired properties.\\
{\bf Claim 2:} for any compact set $K$ of $Y$, the set $H[K]$ is Borel measurable. Indeed, since $H$ is
$\sigma$-compact, there exists a countable family $\{H_m\}$ of compact
subsets of $H$ such that $H=\cup_{m} H_m$ and, hence,   $H[K]=\cup_m
H_m[K]$. Hence $H[K]$ is countable union of compact subsets, hence
Borel measurable, since  the action of $H$ on $Y$ is continuous
and $H_m\times K$ is compact.\\
{\bf Claim 3:} there exists an $H$-invariant Borel set $Y_1$ whose
complement is Lebesgue negligible and  such that the restriction
$p_{\mid Y_1}$ is Borel measurable. The proof of Proposition~2 Ch.~VI
  \S~3.4 of \cite{bourbaki_int} actually implies the claim. For completeness, however,
   we present a direct proof.  Lusin's
  theorem\footnote{ See, for example, Theorem 5.6.23 \cite{sch97} or
    the definition of measurable function given in
    \cite{bourbaki_int}.} yields an
  increasing sequence of compact subsets  $\{K'_m\}$ of $Y$ such that the
  complement of $\cup K'_m$ is Lebesgue negligible and the restriction
  of $p$ to each $K_m$ is continuous.   By Claim~2 the set $Y_1=H[\cup_m K'_m]$ and its complement
$N_1=Y\setminus Y_1$ are both $H$-invariant
Borel subsets,  and $N_1$ is Lebesgue negligible since $N_1\subset Y
\setminus \cup_m K'_m$.  To prove that $p_{\mid Y_1}$ is Borel
measurable,  for any closed subset $C\subset \Omega$
\begin{align*}
p_{\mid Y_1}^{-1}(C)&=p^{-1}(C)\cap Y_1
=\cup_m p^{-1}(C)\cap H[K'_m]\\
&= \cup_mH[C\cap K'_m]
=\cup_m H[p^{-1}_{\mid K_m}(C)],
\end{align*}
since $p^{-1}(C)=H[p^{-1}(C)]$ by~\eqref{comprato}. Since
$p^{-1}_{\mid K_m}(C)$ is compact, Claim~2 implies that $p_{\mid
  Y_1}^{-1}(C)$ is Borel measurable. \\
{\bf Claim 4:}  the quotient space $Y_1/H$ is analytic. Since $Y_1$ is a
Borel subset of a locally compact second countable space, it is
standard and, hence, analytic. By Theorem~5.1 of \cite{mac57},
if a quotient space of an analytic Borel space is countably separated, then it is analytic. 
Hence, it is enough to exhibit a countable family
$\{A_m\}$ of $H$-invariant Borel sets of $Y_1$ with the property that
for any pair of  points $y,y'\in Y_1$ such that $y\not\sim_R y'$, there exists $A_m$
such that $y\in A_m$ and $y'\not\in A_m$.  To find  such a family, choose
a countable base $\{V_m\}$ for the second countable topology of
$\Omega$ and define $A_m=p_{\mid Y_1}^{-1}(V_m)$, which is an
$H$-invariant Borel subset of $Y_1$ by~\eqref{comprato} and Claim~3. If
$y\not\sim_R y'$, then $p(y)\neq p(y')$ and, since $\Omega$ is
Hausdorff, there exists $V_m$ such
that $p(y)\in V_m$ and $p(y')\not\in V_m$, that is, $y\in A_m$ and $y'\not\in A_m$. \\
{\bf Claim 5:}  there exists an $H$-invariant  Borel set $Y_2\subset
Y_1$ whose complement is Lebesgue negligible and 
  $Y_2/H$ is a standard Borel space.  Since $Y$ is second countable,
  there exists a finite measure on the analytic space $Y_1/H$, which is the pseudo-image
  measure of  the Lebesgue measure of $Y$.  By Theorem~6.1 of
  \cite{mac57}, there exists a Borel subset $E\subset
  Y_1/H$ whose complement is negligible and $E$ is a standard Borel
  space.  The set $Y_2=\hat{\pi}^{-1}(E)$ has the desired properties. 

Item~(ii) is  proved by setting $N=Y\setminus Y_2=N_1\cup (Y_1\setminus
Y_2)$ and observing that $(Y\setminus N)/H$ is Borel isomorphic to $E$.
\vskip0.2truecm

We now  show that (ii) implies (i). By assumption there exists a Borel
$H$-invariant Borel set $N\subset Y$ with zero Lebesgue measure such that
$(Y\setminus N)/H$ is Borel isomorphic to a Borel subset of $[0,1]$ and, hence,
there exists a Borel injective map
$j:(Y\setminus N)/H\to \mathbb R$. If $N\neq\emptyset$, fix a section
$s:N/H\to N$,  a point  $y_0\in Y\setminus N$, and define $p:Y\to Y\times [0,1]$ by
$$
p(y)=
\begin{cases}
 (y_0, i({\hat\pi}(y))) & y\not\in N \\
  (s({\hat\pi}(y)),0) & y\in N.
\end{cases}
$$
Clearly, the map $p$ is Lebesgue measurable and $p(y')=p(y)$ if and
only if ${\hat\pi}(y)={\hat\pi}(y')$.  Lusin's theorem implies that there exists an
  increasing sequence of compact subsets $\{K_m\}$ such that the
  complement of $\cup K_m$ is Lebesgue negligible and the restriction
  of $p$ to each $K_m$ is continuous.  By a  standard result in topology,
  (see e.g. Corollary~1 of Proposition~8 \S~10.6 of \cite{bourbaki_GT}), 
${\hat\pi}(K_m)$ is homeomorphic to $p(K_m)$ which is a
  compact subset of a Hausdorff space, so it is Hausdorff.
\end{proof}
In the statement of the above proposition  $Y$ can be replaced by any locally
compact second countable space, the Lebesgue measure by  a  measure on
$Y$ and the equivalence relation induced by $H$ by any other equivalence
relation. 

\subsection{The integral decomposition of $U$} From now on Assumptions~1 and~2 are taken for granted. The main result here is that Theorems~\ref{corcoar} and \ref{Tmackey}, which hold both true, yield an integral decomposition of the  mock metaplectic representation  in terms of induced representations of the  isotropy subgroups of $H$. This fact, which is of independent interest, is at the root of Theorem~\ref{adm2}, where  the admissible vectors for $U$ are characterized.

\begin{proposition}\label{double}
For almost every $z\in  Z$  the family of measures $\{\nu_y\}$ is scalarly  integrable with respect to $\tau_{z}$,  the measure on $X$
$$
\mu_{z}=\int_Y\nu_y\,d\tau_{z}(y)
$$
is concentrated on the $H$-invariant  subset
$\Phi^{-1}(\pi^{-1}(z))$ and for all $h\in H$
\begin{equation}
 \mu_{z}^h=\beta(h)\mu_{z}.
\label{mucov}
\end{equation}
Furthermore, the  family of measures $\{\mu_{z}\} $
is  scalarly integrable with respect to $\lambda$ and 
\begin{equation}
dx=\int_{ Z} \mu_{z}\,d\lambda(z).
\label{mudis}
\end{equation}
\end{proposition}
\begin{proof}
The map $\pi\circ\Phi$ is a Lebesgue measurable map from $X$ to $ Z$ and
$\lambda$ is a pseudo-image measure of the Lebesgue measure restricted
to $X$ under $\pi\circ\Phi$ by construction of $\lambda$ and Assumption~1. Hence, Theorem~2 Ch.~VI \S~3.3
of~\cite{bourbaki_int} yields a family $\{\mu_z\}$  of positive measures on $X$ such
that each $\mu_z$ is concentrated on $\Phi^{-1}(\pi^{-1}(z))$ and, for
all $\varphi\in C_c(X)$
\begin{equation}
\int_X \varphi(x) dx =\int_ Z\bigl(\int_X
\varphi(x)d\mu_z(x)\bigr)d\lambda(z).\label{muprime}
\end{equation}
For any fixed $\varphi\in C_c(X)$,   $y\mapsto
\int_X \varphi(x)d\nu_y(x)$ is Lebesgue integrable by (ii) of Theorem~\ref{corcoar}.
Hence, appealing to (ii) of Theorem~\ref{Tmackey} and to (iii) of Theorem~\ref{th:2}, we know that for almost all $z\in  Z$,  the map $y\mapsto\int_X \varphi(x)d\nu_y(x)$ is
$\tau_z$-integrable,    the map $z\mapsto \int_Y (\int_X
\varphi(x)d\nu_y(x) )d\tau_z(y)$ is $\lambda$-integrable, and
\[
\int_ Z\bigl( \int_Y \bigl(\int_X \varphi(x)d\nu_y(x)
\bigr)d\tau_z(y)\bigr)d\lambda(z) = \int_Y (\int_X
\varphi(x)d\nu_y(x)) dy =\int_X \varphi(x)dx.
\]
Comparing this with~\eqref{muprime} we infer that for almost every $z\in  Z$
\begin{equation}
\int_Y \bigl(\int_X \varphi(x)d\nu_y(x)
\bigr)d\tau_z(y)= \int_X \varphi(x)d\mu_z(x).
\label{munutau}
\end{equation}
The set $N$ of $z\in Z$ where the above inequality does not
hold is $\lambda$-negligible and, can be chosen independently of $\varphi$.
Indeed, as explained in Footnote~\ref{separa} we may find a countable subset $\mathcal S$ of $C_c(X)$ such that, for any
$\varphi\in C_c(X)$, there is a sequence
$(\varphi_i)$ in $\mathcal S$ converging to $\varphi$ uniformly and
    $\lvert \varphi_i\rvert\leq \lvert \varphi_0\rvert$ for all $i$. For each $\varphi\in\mathcal S$ 
there is a negligible  set $N_\varphi\subset  Z$ such that the map $y\mapsto \int_X\varphi(x)\,d\nu_y(x)$  is
integrable with respect to $\tau_z$ for all $z\not\in N_\varphi$. Denote by $N$ the  $\lambda$-negligible set
$\cup_{\varphi\in\mathcal  S}N_\varphi$.
We now claim that the family $\{\nu_y\}$ is scalarly
integrable with respect to $\tau_{z}$ for all $z\not\in N$.  Indeed, given $\varphi\in
C_c(X)$,  there is a sequence
$(\varphi_i)$ in $\mathcal S$ converging to $\varphi$ uniformly and
    $\lvert \varphi_i\rvert\leq \lvert \varphi_0\rvert$ for all $i$. Write \eqref{munutau} for each $\varphi_i$.
    Since  $\lvert \varphi_i\rvert\leq \lvert \varphi_0\rvert$ we may apply the dominated convergence theorem to the right hand side. As for the left hand side, for the same reason we may apply the dominated convergence theorem to the inner integral. Further, since $y\mapsto \nu_y({\rm supp}\,\varphi_0)$  is $\tau_{z}$-integrable we may apply  dominated convergence to the outer integral. The claimed independence of $\varphi$ is proved.

Hence for all $z\not\in N$, the family $\{\nu_y\}$ is scalarly
integrable with respect to $\tau_z$ and
$\mu_z=\int_Y\nu_y\,d\tau_z(y)$. Finally, fix $z\not\in N$ and $h\in H$.  For all $\varphi\in C_c(X)$
\begin{align*}
  \int_X\varphi(h^{-1}.x)d\mu_{z}(x)
  &= \int_Y \left(\int_X\varphi(h^{-1}.x)\,d\nu_y(x) \right) \,d\tau_{z}(y) \\
 & = \alpha(h)\beta(h)\int_Y \left(\int_X\varphi(x) \,d\nu_{h^{-1}[y]}(x)
  \right) \,d\tau_{z}(y)\\
& = \beta(h)\int_Y \left(\int_X\varphi(x) \,d\nu_{y}(x)
  \right) \,d\tau_{z}(y)\\
  &= \beta(h)\int_X\varphi(x)d\mu_{z}(x)
\end{align*}
where the second line is due to the change of variables $x\mapsto h.x$
and~\eqref{alfabeta}, and the third line to $y\mapsto h.y$
and~\eqref{mackey_inv}.  This proves that $ \mu_{z}^h=\beta(h)\mu_{z}$.
\end{proof}

By virtue of Proposition~\ref{double} we may consider the Hilbert space $L^2(X,\mu_{z})$ for almost every $z\in Z$. Whenever $\mu_z$ is not defined, we redefine $\tau_z=0$ and $\mu_z=0$, and  set $L^2(X,\mu_{z})=\{0\}$. Proposition~\ref{directintegral} below, or equation~\eqref{L2}, both based on Proposition~\ref{double}, will allow the following  Hilbert space identifications
\begin{align}
L^2(X)& = \int_{ Z} L^2(X,\mu_{z})\,d\lambda(z)   && f=
  \int_{ Z} f_{z} \,d\lambda(z)\label{idemackey}\\
 L^2(X,\mu_{z} )& = \int_Y L^2(X,\nu_y)\,d\tau_{z}(y)  &&
 f_{z}=
 \int_Y f_{z,y}\,d\tau_{z}(y), \label{idecoarea}
\end{align}
where $f\in L^2(X)$, $f_{z}\in L^2(X,\mu_{z})$ for all $z\in  Z$ and, fixed $z$, $f_{z,y}\in L^2(X,\nu_y)$ for all
$y\in Y$. The integrals of Hilbert spaces are direct integrals with
respect to the measurable field associated with $C_c(X)$, and the
integral of functions are  scalar integrals of
vector valued functions taking value in $M(X)$. Indeed, as explained in the Appendix, we shall regard $L^2(X)$, $L^2(X,\mu_{z})$ and $L^2(X,\nu_y)$ as  subspaces of $M(X)$ in the
natural way. In particular, if $f\in C_c(X)$, $f_{z}$ is the
restriction of $f$ to $\Phi^{-1}(\pi^{-1}(z))$ and   $f_{z,y}$ is the restriction to $\Phi^{-1}(y)$.  Furthermore, for any
$f\in L^2(X)$ 
 \begin{equation}
\|f\|^2=\int_{ Z}\int_{Y}\|f_{z,y}\|_{\nu_{y}}^2
\,d\tau_{z}(y)\,d\lambda(z).
\label{decompose}
\end{equation}
Formula \eqref{idemackey} induces the following decomposition of $U$.
\begin{lemma}\label{Udirect}
The representation $U$  is the direct integral of the family $\{U_z\}$
of representations acting on $L^2(X,\mu_z)$ by
\[
(U_{z,g} f)(x) =\beta(h)^{-\frac{1}{2}} e^{-2\pi i \scal{\Phi(x)}{a}}f(h^{-1}.x)
\]
for $g=(a,h)\in G$ and $f\in L^2(X,\mu_z)$. 
\end{lemma}
\begin{proof}
 For each $z\in Z$,  the map $g\mapsto U_{z,g}$ is a strongly continuous unitary
 representation of $G$ by the same proof of
 Proposition~\ref{mock_prop} since $\mu_z$ and the Lebesgue measure
 are both  relatively invariant with the same character $\beta$,
 (compare~\eqref{beta} with \eqref{mucov}).  We now prove that
 $\{U_z\}$ is a $\lambda$-measurable field of representations. Indeed,
 for any $g\in G$ and $\varphi,\varphi'\in C_c(X)$, 
\[ \scal{U_{z}\varphi}{\varphi'}_{\mu_z}= \int_X
\beta(h)^{-\frac{1}{2}} e^{-2\pi i \scal{\Phi(x)}{a}}
\varphi(h^{-1}.x)\overline{\varphi'(x)}\,d\mu_z(x).
\]
Since $x\mapsto e^{-2\pi i \scal{\Phi(x)}{a}}
\varphi(h^{-1}.x)\overline{\varphi'(x)}$ is a compactly supported
continuous function and the family $\{\mu_z\}$ is $\lambda$-scalarly
integrable, the map  $x\mapsto\scal{U_{z}\varphi}{\varphi'}_{\mu_z}$ is $\lambda$-integrable,
hence $\lambda$-measurable. \\
Finally, to prove that $U=\int_Z U_z\,dz$ it is enough to test the
equality on $C_c(X)$. For any $g\in G$ and $\varphi\in C_c(X)$, we
regard $U_g \varphi$ and $U_{z,g}\varphi$ as elements of
$M(X)$. Hence, \eqref{mudis} gives
\[
U_g \varphi \cdot dx = \int_Z ( U_g\varphi \cdot \mu_z )dz =\int_Z (
U_{z,g}\varphi \cdot \mu_z )dz
\]
by definition of $U_z$.
\end{proof}
The next technical lemma is needed in order  to prove that $U_z$ is equivalent to an induced representation.
\begin{lemma}\label{T}
Fix $y\in Y$ and $h\in H$. The map
 $T_{y,h}:L^2(X,\nu_y)\to  L^2(X,\nu_{h[y]})$
 defined  for $\nu_{h[y]}$-almost every $x\in X$ by
$$ 
(T_{y,h} f)(x)=\sqrt{\alpha(h^{-1})\beta(h^{-1})} f(h^{-1}.x)\qquad
$$
is a unitary operator. Furthermore,  for every $h,h'\in H$ and every $y\in Y$
\begin{align}
 T_{h[y],h'} T_{y,h} & =T_{y,h'h} \label{cociclo}\\
    T_{y,h}^{-1} & = T_{h[y],h^{-1}}.\label{cociclo2}
\end{align}
\end{lemma}
\begin{proof}
 Given a Borel measurable function $f$ which is square-integrable with
 respect to $\nu_{y}$, the map $x\mapsto (T_{y,h} f)(x)$ is also Borel measurable and it is square-integrable with respect to $\nu_{h[y]}$ since
$$ \alpha(h^{-1})\beta(h^{-1})\int_X \lvert f(h^{-1}.x)\rvert^2\,d\nu_{h[y]}(x)= \int_X
\lvert f(x)\rvert^2\,d\nu_{y}(x),$$
by the change of variables $x\mapsto h.x$
and~\eqref{alfabeta}. The above equation implies that $T_{y,h}$ is a
well-defined isometry from  $L^2(X,\nu_y)$ to
$L^2(X,\nu_{h[y]})$. Equality~\eqref{cociclo} is clear and, as a
consequence, $ T_{h[y],h^{-1}}T_{y,h}=T_{y,e}$ is the identity on
$L^2(X,\nu_{h[y]})$ so that $T_{y,h}$ is surjective, thereby
showing~\eqref{cociclo2}.
\end{proof}
For any $z\in\pi(Y)$, we fix an origin $y_0$ in the orbit
$\pi^{-1}(z)=H[y_0]$ and we denote by  $H_{z}$ the
stabilizer at $y_0$.
We denote by $\cK_{z}=L^2(X,\nu_{y_0})$.

By \eqref{alfabeta} we know that
$\nu_{y_0}$ is  relatively invariant under   $H_{z}$. It follows that
it make sense to look at the quasi-regular representation $\Lambda_{z}$ of $H_{z}$
acting on $\cK_{z}$, whose value at $s\in H_{z}$ is $\Lambda_{z,s}= T_{y_0,s}$.  As usual, we extend   $\Lambda_{z}$ to a representation of $\R^n\rtimes
  H_{z}$ by setting $\Lambda_{z,a} = e^{-2\pi i
    \scal{y_0}{a}} \id$ for all $a\in \R^n$. Finally, we denote by $W_{z}$ the representation of $G$ unitarily induced by
    $\Lambda_{z}$ from $\R^n\rtimes H_{z}$ to $G$.  We realize $W_{z}$ as a representation acting on the space $\cH_{z}$ of those  functions $F:G\to \cK_{z}$ that satisfy
    \begin{enumerate}
    \item[(K1)] $F$ is $dg$-measurable;
     \item[(K2)] For all $g\in G$ and $(a,s)\in \R^n\rtimes H_{z}$
$$
F(gas)= \sqrt{\alpha(s^{-1})}\ \Lambda_{z,as}^{-1}\, F(g);
$$
   \item[(K3)] $\displaystyle{\lVert F\rVert^2_{\cH_{z}} := \int_{Y} \lVert
       F(h(y))\rVert^2_{\cK_{z}}\,\alpha(h(y)) d\tau_{z}(y)}<+\infty$. 
      \end{enumerate}
Here  $h(y)\in H$ is any element in $H$ that satisfies  $h(y)[y_0]=y$
for $\tau_z$- almost all $y\in Y$. Since $\tau_{z}$ is concentrated on
$H[y_0]$, it is enough to define $h(y)$ for $y\in H[y_0]$ and, due to
the covariance property in (K2), the integral does not depend on the
choice of $h(y)$ in the coset $hH_{z}$. Furthermore, (K2) implies that
it is enough to know these functions on $H$. 
 Two functions $F$ and $F'$ are identified if $\lVert F-F'\rVert^2_{\cH_{z}} =0$. 
 The induced representation on $\cH_{z}$ is defined  for $g\in G$ by the equality
$$ (W_{z,g}F)(g')=F(g^{-1}g')$$
valid for $dg$-almost every $g'\in G$. 

For the sake of precision,  if $z\in Z\setminus\pi(Y)$ we put $\nu_z=0$, $\cK_z=\{0\}$
and $H_z=\{e\}$; recall that $\tau_z=0$ and that $\lambda(Z\setminus\pi(Y))=0$.

\begin{lemma}\label{S} Fix $z\in  Z$ such that $\tau_z\neq0$. The map 
$S_{z}:L^2(X,\mu_{z})\to \cH_{z}$ whose value at 
$ f_{z}=\int_Y f_{z,y}\,d\tau_{z}(y)$ is given by
$$ 
(S_{z} f_{z})(a,h)=\sqrt{\alpha(h^{-1})}\ e^{2\pi i
\scal{h[y_0]}{a}} \,\,T_{y_0,h}^{-1}( f_{z,h[y_0]})
$$
is a unitary operator intertwining $U_z$ with $S_z$.
\end{lemma}
\begin{proof}
For any $(a,h)\in G$,   $f_{z,h[y_0]}\in L^2(X,\nu_{h[y_0]})$. Hence $T_{y_0,h}^{-1}( f_{z,h[y_0]})\in \cK_{z}$.
In order to prove that  $S_{z} f_{z}$ is $dg$-measurable it is enough to
show that
$$ 
h\mapsto  \scal{T_{y_0,h}^{-1}( f_{z,h[y_0]})}{\varphi}_{\cK_{z}}=\sqrt{\alpha(h)\beta(h)}\int_X  f_{z,h[y_0]}(h.x) \varphi(x)d\nu_{y_0}(x)
 $$
is $dh$-measurable for every $\varphi\in C_c(X)$ because $C_c(X)$ is a
dense subspace of the separable Hilbert  space $\cK_{z}$.
Since $f_{z}=\int_Y f_{z,y}\,d\tau_{z}(y)$, there exists a
square-integrable function $\tilde{f}:X\to\C$ and a
$\tau_{z}$-negligible set $N\subset Y$ such that, for all $y\not\in N$,  
$\tilde f$ belongs to the equivalence class of $f_{z,y}\in
L^2(X,\nu_y)$.  Define $N'=\{h\in H\mid h[y_0]\in N\}$, a negligible
set with respect to the Haar measure $dh$ because, by
\eqref{mackey_inv},  $\tau_{z}$ is  non-zero relatively invariant on the orbit $H[y_0]$. Then for all $h\not\in N'$
\begin{align*}
h&\mapsto\sqrt{\alpha(h)\beta(h)}\int_X  f_{z,h[y_0]}(h.x) \varphi(x)d\nu_{y_0}(x)\\
& = \sqrt{\alpha(h)\beta(h)}\int_X  \tilde f(h.x) \varphi(x)d\nu_{y_0}(x),
\end{align*}
which is clearly $dh$-measurable. Next we prove the covariance property (K2). For $g=(a,h)=ah$
and $(b,s)=bs\in\R^n\rtimes H_{z}$,
\begin{align*}
 (S_{z} f_{z})(ahbs)&=  (S_{z} f_{z})(a+h^\dag[b],hs)\\
  &=\sqrt{\alpha(h^{-1})\alpha(s^{-1})}\ e^{2\pi i
  \scal{hs[y_0]}{a+h^\dag[b]}} \,\,T_{y_0,hs}^{-1}( f_{z,hs[y_0]}) \\
&= \sqrt{\alpha(s^{-1})}\ e^{2\pi i \scal{h[y_0]}{h^\dag[b]}}
\,\,T_{y_0,s}^{-1}(S_{z} f_{z})(a,h) \\
& =\sqrt{\alpha(s^{-1})}\ e^{2\pi i \scal{y_0}{b}}
\,\,\Lambda_{y_0,s}^{-1}(S_{z} f_{z})(a,h)
\end{align*}
by definition of $h^\dag$ and $\Lambda_{z}$. Further,
\begin{align*}
   \int_{Y} \lVert
      (S_{z} f_{z})(h(y))\rVert^2_{\cK_{z}}\,\alpha(h(y)) d\tau_{z}(y) & = \int_{Y}  \lVert
      T_{y_0,h(y)}^{-1}( f_{z,h(y)[y_0]})
\rVert^2_{\cK_{z}} d\tau_{z}(y)  \\
  & = \int_{Y}  \lVert  f_{z,y}
\rVert^2_{\nu_y} d\tau_{z}(y)  = \int_X \lvert f(x)\rvert^2 d\mu_{z}(x),
\end{align*}
whence (K3). This also shows that  $S_{z}$ is an isometry from $L^2(X,\mu_{z})$ into
$\cH_{z}$. \\
Finally we prove that $S_{z}$ is surjective.  Given $F\in
\cH_{z}$, for all $h\in H$ define
$$ 
f_{z, h}=  \sqrt{\alpha(h)} \,\,T_{y_0,h}( F(h))\in L^2(X,\nu_{h[y_0]}).
$$
Since $F$ satisfies (K2), it follows that $f_{z,
  hs}=f_{z, h}$. For $\varphi\in C_c(X)$ the map
$$
h\mapsto \sqrt{\alpha(h)}  \scal{T_{y_0,h}(
  F(h))}{\varphi}_{\nu_{h[y_0]}}=
\sqrt{\alpha(h)}  \scal{F(h)}{\varphi^{h}}_{\cK_{z}}
$$
is $dh$-measurable since $h\mapsto F(h)$ is $dh$-measurable from $H$
into $\cK_{z}$ and the map $h\mapsto\sqrt{\alpha(h)}\varphi^h$ is continuous from $H$ into
$\cK_{z}$. Therefore
$$ \int_Y  \lVert f_{z, h(y)}\rVert^2_{\nu_y}\,d\tau_{z}(y)=
\int_{Y} \lVert
       F(h(y))\rVert^2_{\cK_{z}}\,\alpha(h(y)) d\tau_{z}(y)<+\infty.$$
It follows that $f_{z}=\int_Y f_{z,h(y)}\,d\tau_{z}(y)$
is in $\int_Y L^2(X,\nu_y)\,d\tau_{z}(y)=L^2(X,\mu_{z})$ and,
by construction, $S_{z}f_{z}=F$.\\
Finally, we check the intertwining property on the dense subset $C_c(X)$ of $L^2(X)$.  
If $g=a\in\R^n$,
 for any $\varphi\in C_c(X)$  and for almost every $h\in H$
\begin{align*}
  (S_{z} (U_a\varphi))(h) & = \sqrt{\alpha(h^{-1})}\,
  T_{y_0,h}^{-1}\left(e^{-2\pi i\scal{\Phi(\cdot)}{a}}\varphi\right) \\
& =  \sqrt{\alpha(h^{-1})} \,e^{-2\pi i\scal{h[y_0]}{a}} T_{y_0,h}^{-1}\varphi\\
  & = ( S_{z}\varphi)(-a,h)=( S_{z}\varphi)(a^{-1}h)
\end{align*}
where, in the second line, we have used $\Phi(x)=h[y_0]$ for $\nu_{h[y_0]}$-almost
every $x\in X$. If $g=k\in H$, 
 \begin{align*}
  (S_{z} (U_k\varphi))(h) & = \sqrt{\alpha(h^{-1})}\,
  T_{y_0,h}^{-1}\left(\sqrt{\beta(k^{-1})}\varphi^{k}\right) \\
& =  \sqrt{\alpha(h^{-1})}\sqrt{\alpha(k)} T_{y_0,h}^{-1}
(T_{k^{-1}h[y_0],k}\varphi) \\
& =  \sqrt{\alpha((k^{-1}h)^{-1})} \left(
  T_{k^{-1}h[y_0],k}^{-1}T_{y_0, h}\right)^{-1}\varphi \\
& =\sqrt{\alpha((k^{-1}h)^{-1})} \left(
  T_{h[y_0],k^{-1}}T_{y_0, h}\right)^{-1}\varphi \\
& =\sqrt{\alpha((k^{-1}h)^{-1})} \left(
  T_{y_0, k^{-1}h}\right)^{-1}\varphi= (S_{z}\varphi)(k^{-1}h).
\end{align*}
Since two functions in $\cH_{z}$ that are equal for almost every $h\in H$,  
are equal  almost everywhere in  $G$, the intertwining  is proved.
\end{proof}
Recall that $L^2(X)=\int_{ Z} L^2(X,\mu_{z})\,dz$,
where the direct integral is defined by the measurable structure
associated with any fixed dense countable family $\{\varphi_k\}$  in
$C_c(X)$. Clearly,  $z\mapsto \{S_{z}\varphi_k\}$ is a measurable
structure for the family $\{\cH_{z}\}$, and we define the direct
integral 
$\cH=\int_{ Z} \cH_{z}\,dz$. 
\begin{theorem}\label{intertwine}
 The map $S:L^2(X)\to \cH$
$$ S f =  \int_{ Z} S_{z} f_{z} \,dz
\qquad f = \int_{ Z} f_{z} \,dz$$ 
is a unitary map intertwining the mock metaplectic representation $U$
with the unitary representation $W$ of $G$ acting on $\cH$ given by
$$
W=\int_{ Z} W_{z}\,dz.
$$
\end{theorem}
\begin{proof}
The statement follows from
the definition of the measurable structure for the direct
  integral $\int_{ Z} \cH_{z}\,dz$,  from Lemma~\ref{Udirect} and
  Lemma~\ref{S}.
 \end{proof}

\subsection{Admissible vectors}
We are in a position to state our main result. We need, however, a last disintegration formula, sometimes referred to as Weil's formula (see e.g.~\cite{fol95}),  a rather straightforward consequence of the theory of quasi-invariant measures on homogeneous spaces.
The easiest way of formulating it is perhaps that  for any $\varphi\in C_c(H)$ the following integral formula holds
\begin{equation}
\int_H \varphi(h) \alpha(h^{-1}) dh = \int_{Y} \left(\int_{H_{z}}
\varphi(h(y)s) ds\right)d\tau_{z}(y),
\label{weil1}
\end{equation}
where $ds$ is a suitable Haar measure on the stabilizer $H_{z}$ and where as before  $h(y)\in H$ is any element  that satisfies  $h(y)[y_0]=y$ for  $\tau_{z}$-almost every $y\in Y$. We  interpret~\eqref{weil1} along the same lines of thought  that we have followed for the other formulae by writing
\begin{equation}
  \label{weil2}
\alpha^{-1}\cdot dh = \int_{Y} (ds)^{h(y)^{-1}} d\tau_{z}(y)  
\end{equation}
as an equality of measures on $H$. This time $ds$ is regarded as a measure on $H$ concentrated on $H_{z}$, so that the translated measure $(ds)^{h(y)^{-1}}$ is concentrated on $h(y)H_{z}$. As usual, we shall extend \eqref{weil1} to $L^1$-functions by means of Theorem~\ref{th:2}. By Theorem~2 (and the comments below)  in Ch.~VII \S~3.5 of \cite{bourbaki_intII},
for all $s\in H_{z}$ the modular functions of $H$ and $H_{z}$ are related by the formula
\begin{equation}
  \label{weil3}
  \alpha^{-1}(s)=\frac{\Delta_{H_{z}}(s)}{\Delta_{H}(s)}.
\end{equation}

Theorem~\ref{intertwine} establishes that $U$ and $W$ are equivalent. Therefore, we formulate our necessary and sufficient condition for the existence of admissible vectors of $U$ for those of $W$. 
Thus, any admissible vector $F\in\cH$ for $W$ is to be thought of as the image under $S: L^2(X)\to\cH$ of an analyzing wavelet $\eta$. 

\begin{theorem}\label{adm2} The function 
 $F=\int F_{z}\,d\lambda(z)$ is an admissible vector for $W$ if and only if
for almost every $z\in  Z$ and for every $u\in \cK_{z}=L^2(X,\nu_{y_0})$
\begin{equation}
\|u\|_{\cK_{z}}^2=\int_{Y}\Bigl(\int_{H_{z}}|\scal{u}{\Lambda_{z,s}
\left(F_{z}\Delta_G^{-1/2}\right)(h(y))}_{\cK_{z}}|^2\,ds\Bigr)
\alpha(h(y))\,d\tau_{z}(y).
\label{SECOND}
\end{equation}
\end{theorem}
\begin{proof}  %As a particular instance of the cocycle identity \eqref{cociclo} we have
%\begin{equation}
%T_{h^{-1}[y_0],h}T_{y_0,h^{-1}}=T_{y_0,e}={\id},
%\label{inverse}
%\end{equation}
%the identity of $L^2(X,\nu_{y_0})=\cK_{z}$. 
By the definition of $T$ given in  Lemma~\ref{T}, for every $h\in H$ and $y_0\in Y$ 
$$
T_{y_0,h^{-1}}(\eta^h)_{z,y_0}(x)
=\sqrt{\alpha(h)\beta(h)}(\eta^h)_{z,h^{-1}[y_0]}(h.x)
=\sqrt{\alpha(h)\beta(h)}\eta_{z,h^{-1}[y_0]}(x)
$$
holds for any $\eta=\int_Z\int_Y\eta_{z,y}\,d\tau_z(y)\,d\lambda(z)\in L^2(X)$ and hence
%applying \eqref{inverse}  we obtain
$$
(\eta^h)_{z,y_0}
%=\sqrt{\alpha(h)\beta(h)}T_{h^{-1}[y_0],h}\eta_{h^{-1}[y_0]}
=\sqrt{\alpha(h)\beta(h)}\left(T_{y_0,h^{-1}}\right)^{-1}\eta_{z,h^{-1}[y_0]}.
$$
Suppose now  that $\eta$ is an admissible vector for $U$ or, equivalently, that $F=S\eta$ is such for $W$. By Theorem~\ref{adm1}, what we have just established and the definition of $S$ given in Lemma~\ref{S},  for almost every $z\in Z$ and any fixed   $y_0\in\pi^{-1}(z)$
\begin{align}
\|u\|^2_{\cK_{z}}&=\int_H|\scal{u}{(\eta^h)_{z,y_0}}|^2\frac{dh}{\alpha(h)\beta(h)}\label{firstline}\\
&=\int_H|\scal{u}{\sqrt{\alpha(h)\beta(h)}
\left(T_{y_0,h^{-1}}\right)^{-1}\eta_{z,h^{-1}[y_0]}}|^2\frac{dh}{\alpha(h)\beta(h)}\nonumber\\
&=\int_H|\scal{u}{S_{z}\eta_{z}(h^{-1})}|^2\frac{dh}{\alpha(h)}\nonumber\\
&=\int_H|\scal{u}{F_{z}(h^{-1})}|^2\frac{dh}{\alpha(h)}\nonumber\\
(h\mapsto h^{-1})\hskip0.3truecm&=\int_H|\scal{u}{F_{z}(h)}|^2\Delta_H(h^{-1})\alpha(h)\,dh.\nonumber
\end{align}
Hence, applying~\eqref{weil1}, the covariance property (K2),  \eqref{weil3} and \eqref{modularG} we obtain
\begin{align*}
\|u\|^2_{\cK_{z}}
&=\int_{Y}\left(\int_{H_{z}}|\scal{u}{F(h(y)s)}|^2\frac{\alpha^2(h(y)s)}{\Delta_H(h(y)s))}\,ds\right)
\,d\tau_{z}(y)\\
&=\int_{Y}\left(\int_{H_{z}}|\scal{u}{\sqrt{\alpha(s^{-1})}\Lambda_{{z},s^{-1}}F(h(y))}|^2\frac{\alpha^2(h(y)s)}{\Delta_H(h(y)s))}\,ds\right)
\,d\tau_{z}(y)\\
&=\int_{Y}\left(\int_{H_{z}}|\scal{u}{\Lambda_{{z},s^{-1}}F(h(y))}|^2
\frac{\alpha^2(h(y))}{\Delta_H(h(y))}\Delta_{H_{z}}(s^{-1})\,ds\right)
\,d\tau_{z}(y)\\
(s\mapsto s^{-1})\hskip0.3truecm&=\int_{Y}\left(\int_{H_{z}}|\scal{u}{\Lambda_{{z},s}F(h(y))}|^2\frac{1}{\Delta_G(h(y))}\,ds\right)
\,\alpha(h(y))d\tau_{z}(y),
\end{align*}
which is \eqref{SECOND}. Conversely, if \eqref{SECOND} holds for some $F\in\cH$, then reading the above strings of equalities backwards yields the first line in \eqref{firstline}.
Therefore, by Theorem~\ref{adm1}, $\eta$ is admissible for $U$, hence $F$ is such for $W$.
\end{proof}

\begin{corollary}\label{finiteness}
Assume that $U$ is a reproducing representation and suppose that  $z\in Z$ is such that  \eqref{SECOND} holds true. Then:
\begin{enumerate}
\item[(i)] if $\Phi^{-1}(y_0)$ is a finite set for some $y_0\in\pi^{-1}(z)$,   then the stabilizer $H_{y}$ is compact for every $y\in\pi^{-1}(z)$;
\item[(ii)] if $G$ is unimodular and the stabilizer $H_{y}$ is compact, then  $\Phi^{-1}(y)$ is a finite set, hence $n=d$.
\end{enumerate}
\end{corollary}
\begin{proof} Clearly, it is enough to prove (i) and (ii) for the origin $y_0$. Take a (countable) Hilbert basis $\{u_i\}$ of $\cK_{z}$. Apply \eqref{SECOND} to each element of the basis and sum
\begin{align}
\dime\cK_{z}
&=\int_{Y}\Bigl(\int_{H_{z}}\sum_i|\scal{u_i}{\Lambda_{z,s}
\left(F_{z}\Delta_G^{-1/2}\right)(h(y))}_{\cK_{z}}|^2\,ds\Bigr)
\alpha(h(y))\,d\tau_{z}(y)\nonumber\\
&=\int_{Y}\Bigl(\int_{H_{z}}
\lVert\Lambda_{z,s}\left(F_{z}\Delta_G^{-1/2}\right)(h(y))\rVert_{\cK_{z}}^2\,ds\Bigr)
\alpha(h(y))\,d\tau_{z}(y)\nonumber\\
&=\Bigl(\int_{H_{z}}ds\Bigr)
\int_{Y}\lVert F_{z}\Delta_G^{-1/2}(h(y))\rVert_{\cK_{z}}^2
\alpha(h(y))\,d\tau_{z}(y).\label{formaldegree}
\end{align} 
Now, if  $\Phi^{-1}(y_0)$ is a finite set, then the left hand side is finite and strictly positive, hence so is the right hand side, so $H_{z}$ has finite volume. This proves (i).
If $\Delta_G=1$ and  $H_{z}$ has finite volume, then the right hand side is finite and strictly positive by (K3). Hence $\Phi^{-1}(y_0)$ is a finite set and since it is a regular submanifold of dimension $d-n$, necessarily $n=d$. Thus (ii) holds. 
\end{proof}

%%%%%%%%%%%%%%%%%%%%%%%%%%%%%%%%%%%%%%%%%%%%%
% ulteriori definizioni

\newcommand{\mf}{{m_{\mh}}}
\newcommand{\IS}{\cN}

\subsection{Compact stabilizers} \label{sec-compatto}
As a preliminary step, we  assume that the stabilizer $H_{z}$ of a given $z\in Z$  is
compact, hence such is any other stabilizer in the same orbit.  Later we shall assume that this is the case for almost every orbit.\\
The compactness of the stabilizer allows us to use Schur's orthogonality relations for computing the inner integral over $H_{z}$ in~\eqref{SECOND}. Indeed, 
since $H_{z}$ is compact, the representation $\Lambda_{z}$ is completely reducible.
Hence, for each equivalence class ${\hat s}$
in the dual group  $\widehat{H}_{z}$,  we can choose a closed subspace
$\cK_{z,{\hat s}}\subset \cK_{z}$ such that the restriction
$\Lambda_{z,{\hat s}}$ of $\Lambda_{z}$ to $\cK_{z,{\hat s}}$ belongs to ${\hat s}$, and we
denote by $m_{\hat s}$ the multiplicity of ${\hat s}$ in  $\Lambda_{z}$ (with the convention that $\cK_{z,{\hat s}}=0$  if $m_{\hat s}=0$).
The following direct decomposition in primary  inequivalent representations holds true
\begin{equation}
\cK_{z}\simeq\bigoplus_{{\hat s}\in \widehat{H}_{z}} \cK_{z,{\hat s}}
\otimes \mathbb C^{m_{\hat s}}
\qquad
\Lambda_{z} \simeq \bigoplus_{{\hat s}\in \widehat{H}_{z}}
\Lambda_{z,{\hat s}}\otimes \id,\label{dec_inducente} 
\end{equation}
where we interpret $\C^{m_{\hat s}}=\ell^2$ whenever $m_{\hat s}=\aleph_0$. Furthermore, for any cardinal $m\in\{1,\ldots,\aleph_0\}$,   we denote by $\{e_j\}_{j=1}^{m}$  the canonical basis of $\C^m$.

Mackey's theorem on induced representations of semi-direct products \cite{mackey52} guarantees that each induced representation 
$\operatorname{Ind}_{\R^d\rtimes H_{z}}^G (e^{-2\pi i\scal{y_0}{\cdot}}\,  \Lambda_{z,{\hat s}})$ is  irreducible on $\cH_{z,{\hat s}}$ and gives the following direct decomposition in primary inequivalent  representations for $W_{z}$:
\begin{equation}
\cH_{z}\simeq\bigoplus_{{\hat s}\in \widehat{H}_{z}}\cH_{z,{\hat s}} \otimes \mathbb C^{m_{\hat s}}\qquad  W_{z}\simeq \bigoplus_{{\hat s}\in \widehat{H}_{z}}
\operatorname{Ind}_{\R^d\rtimes H_{z}}^G(e^{-2\pi
  i\scal{y_0}{\cdot}}\,  \Lambda_{z,{\hat s}}) \otimes \id.\label{dec_indotta}
\end{equation}
By \eqref{dec_inducente} and
\eqref{dec_indotta}, respectively, we have
\begin{align*}
F_{z}
&=\sum_{{\hat s}\in \widehat{H}_{z}}\sum_{i=1}^{m_{\hat s}}{F_{z,{\hat s},i}}\otimes e_i,
\qquad F_{z}\in\cH_{z}\\
u
&=\sum_{{\hat s}\in \widehat{H}_{z}}\sum_{i=1}^{m_{\hat s}}{u_{{\hat s},i}}\otimes e_i,
\qquad u\in\cK_{z}.
\end{align*}
We write $\vol{H_{z}}$ for the mass of  $H_{z}$ relative to the unique Haar measure $ds$ that makes formula~\eqref{weil1} work. Note that $\vol{H_{z}}$ is not necessarily one. 
\begin{proposition}\label{compatto}
Let $z\in Z$ be such that the stabilizer $H_{z}$  is compact. Given $F_{z}\in\cH_{z}$ the following facts are equivalent:
\begin{enumerate}
\item[(i)]  equality~\eqref{SECOND} holds true for all $u\in\cK_{z}$;
\item[(ii)] for all ${\hat s}\in \widehat{H}_{z}$ such that $m_{\hat s}\neq 0$, and for all $i,j=1,\ldots,m_{\hat s}$
\begin{align}\label{degreeOP}
  \int_{Y} \scal{F_{z,{\hat s},i}(h(y))}{F_{z,{\hat s},j}(h(y))}_{\cK_{z,{\hat s}}}
  \frac{\alpha(h(y))}{\Delta_G(h(y))} d\tau_{z}(y) =
  \frac{\dim{\cK_{z,{\hat s}}}}{\vol{H_{z}}} \delta_{ij}.
\end{align}
\end{enumerate}
\end{proposition}
\begin{proof}
Take $u\in\cK_{z}$. We compute the inner integral in \eqref{SECOND} using Schur's orthogonality relations. For $\tau_{z}$-almost every $y\in Y$ 
\begin{align*}
\int_{H_{z}}|\scal{u}{\Lambda_{z,s} F_{z}(h(y))}_{\cK_{z}}|^2\,ds 
&=\sum_{{\hat s}\in \widehat{H}_{z}} \sum_{i,j=1}^{m_{\hat s}}
\scal{u_{{\hat s},i}}{u_{{\hat s},j}}_{\cK_{z,{\hat s}}}\times\\  
&\hskip0.4truecm\times
\scal{F_{z,{\hat s},j}(h(y))}{F_{z,{\hat s},i}(h(y))}_{\cK_{z,{\hat s}}} \frac{\vol{H_{z}}}{\dim{\cK_{z,{\hat s}}}}.
\end{align*}
Choosing $u=u_{{\hat s},i}$,~\eqref{SECOND}
is equivalent to 
\begin{equation}
 \int_{Y} \lVert F_{z,{\hat s},i}(h(y))\rVert^2_{\cK_{z,{\hat s}}}
  \frac{\alpha(h(y))}{\Delta_G(h(y))} d\tau_{z}(y) =
  \frac{ \dim{\cK_{z,{\hat s}}}}{\vol{H_{z}}}.
\label{ciccio}
\end{equation}
Choose next  $j\neq i$ and $u=u_{{\hat s},i}\oplus u_{{\hat s},j}$. Taking~\eqref{ciccio} into account,~\eqref{SECOND} is equivalent to
$$ \int_{Y} \scal{F_{z,{\hat s},i}(h(y))}{F_{z,{\hat s},j}(h(y))}_{\cK_{z,{\hat s}}}
  \frac{\alpha(h(y))}{\Delta_G(h(y))} d\tau_{z}(y) =0.
$$
 Hence (i) is equivalent to (ii). 
\end{proof}

Equation~\eqref{degreeOP} has the following interpretation in terms of the
abstract theory developed by F\"uhr \cite{fuhr05}. Indeed, for each
irreducible representation of $G$ in~\eqref{dec_indotta},
we can define the (possibly unbounded) operator $d_{z,{\hat s}}$ on $\cH_{z,{\hat s}}$
\begin{equation}
d_{z,{\hat s}}F_{z,{\hat s}}(g)=\frac{\dim \cK_{z,{\hat s}}}{\vol H_{z}}\Delta_G(g)\,F_{z,{\hat s}}(g),
\label{formaldegree2}
\end{equation}
which  satisfies (K2) precisely because the stabilizer is compact. The
operator $d_{z,{\hat s}}$ is a positive self-adjoint injective operator
semi-invariant with weight $\Delta_G^{-1}$ \cite{dumo76}. 
Now, \eqref{degreeOP} says that $F_{z,{\hat s},i}$ is in the domain of $d_{z,{\hat s}}^{-1/2}$ and 
\begin{equation}
\scal{d_{z,{\hat s}}^{-1/2}F_{z,{\hat s},i}}{d_{z,{\hat s}}^{-1/2}F_{z,{\hat s},j}}_{\cH_{z,{\hat s}}}
=\delta_{ij},
\qquad
i,j=1,\dots,m_{{\hat s}}.
\label{DEGOP2}\end{equation}
One should compare this  with Theorem~4.20 and equations (4.15) and (4.16) of \cite{fuhr05}.

\begin{corollary}\label{compatto-s}
 Let $z\in Z$ be such that the stabilizer $H_{z}$  is
 compact.  The following are equivalent:
 \begin{itemize}
 \item[(i)] there  exists $F_{z}\in\cH_z$ such that equality~\eqref{SECOND} holds true for all $u\in\cK_{z}$;
  \item[(ii)] $m_{\hat s}\leq \dim(\cH_{z,{\hat s}})$ for all ${\hat s}\in \widehat{H_{z}}$.
  \end{itemize}
 If $G$ is non-unimodular, this last condition is always satisfied.
\end{corollary}
\begin{proof}
Fix ${\hat s}\in \widehat{H_{z}}$ such that $m_{\hat s}\neq 0$.
 If $G$ is unimodular, $d_{z,{\hat s}}$ is the identity up to a
  multiplicative constant, so that  the families $\{F_{z,{\hat s},i}\}_{i=1}^{m_{\hat s}}$
  satisfying~\eqref{DEGOP2} are precisely the orthogonal families in
  $\cH_{z,{\hat s}}$ with square norm equal to
  $\dim \cK_{z,{\hat s}}/\vol H_{z}$, whose existence is
  equivalent to $m_{\hat s}\leq \dim(\cH_{z,{\hat s}})$. 
  If $G$ is non-unimodular,   $d_{z,{\hat s}}$ is a semi-invariant
  operator with weight $\Delta_G^{-1}$.  Therefore its spectrum is unbounded
  (see formula~(2) of \cite{dumo76}), so that $\dime\cH_{z,{\hat s}}=+\infty$, provided that $m_{\hat s}\neq 0$.  Hence the families
  $\{F_{z,{\hat s},i}\}_{i=1}^{m_{\hat s}}$ satisfying~\eqref{DEGOP2} are
the families in the domain of~$d_{z,{\hat s}}^{-1/2}$ that
  are orthonormal with respect to the inner product induced by
  $d_{z,{\hat s}}^{-1/2}$.
\end{proof}

If $G$ is unimodular,  (ii)~of Corollary~\ref{finiteness} implies that
$\cK_{z}$ is finite-dimensional, so that $m_{\hat s}=0$ for all but
finitely many ${\hat s}\in \widehat{H}_{z}$ for which $m_{\hat s}$ is
finite. Furthermore,   the orbit $\pi^{-1}(z)$ is often infinite, so that $\dime\cH_{z,{\hat s}}=+\infty$ and the
requirement $m_{{\hat s}}\leq\dime\cH_{z,{\hat s}}$ is trivially satisfied
for every ${\hat s}\in\widehat{H}_{z}$.

From now on we assume that almost every stabilizer $H_{z}$ is
compact. For each $z$ we can thus apply Proposition~\ref{compatto}. 
Theorem~\ref{pazzini}   provides an explicit decomposition of the
representation $W$, hence of $U$, as a direct integral of its irreducible components,
each of which is realized as induced representation of the
restriction of $\Lambda_{z}$ to a suitable (irreducible)
subspace. The result does not depend on the fact that $U$ is
reproducing.  To state the theorem, we fix a Borel (hence $\lambda$) measurable
section $o:  \pi(Y)\to Y$ whose existence is ensured by Assumption~2 and by  Theorem~2.9 in \cite{effros65},  
thereby choosing $o(z)$ as the  origin of the  orbit  $\pi^{-1}(z)$.
We then extend $o:Z\to Y$ measurably. Thus,
for all $z\in  Z$,  we have  $\cK_{z}=L^2(X,\nu_{o(z)})$. 
\begin{lemma}\label{rosso}
 The field  of Hilbert spaces  $z \mapsto \cK_{z}$ is
 $\lambda$-measurable with respect to the measurable structure
 induced by $C_c(X)\subset K_{z}$ and the corresponding direct
 integral $\cK=\int_{ Z} \cK_{z}\, d\lambda(z)$ is a separable Hilbert space.
\end{lemma}
\begin{proof}
For any   $\varphi,\varphi'\in
C_c(X)$ the map $z\mapsto
\int_X\varphi(x)\overline{\varphi(x)'}d\nu_{o(z)}(x)$ is $\lambda$-mea\-sura\-ble because $y\mapsto
\int_X\varphi(x) \overline{\varphi(x)'}d\nu_y(x)$ is continuous (see
(iii) of Theorem~\ref{corcoar}) and
$o$ is Borel measurable. Since $ Z$ is second countable,
Corollary of Proposition~6 Ch.~II \S~1.5 in \cite{dix57} implies that $\cK$ is separable.
\end{proof}

\begin{theorem}\label{pazzini} 
Assume that   for $\lambda$-almost every $z\in Z$ the stabilizer $H_{z}$ is compact. There exist a countable family $\{z\mapsto \cK^{n}_{z}\}_{n\in\cN}$
of $\lambda$-measurable fields of Hilbert subspaces $\cK^{n}_{z}$ of $\cK_{z} $, and a family of cardinals
$\{{m_n}\}_{n\in \cN}\subset\{1,\ldots,\aleph_0\}$ such that,  for  almost every $z\in  Z$,
\begin{align}
  {\cK_{z}} & =\bigoplus_{{n\in\cN}} \cK_{z}^{n}\otimes \C^{{m_n}},\label{decomK} \\
  \Lambda_{z} &= \bigoplus_{{n\in\cN}} \,\Lambda_{z}^{n}\otimes \id,\label{decomL}
\end{align}
where \eqref{decomL} is the decomposition of $\Lambda_{z}$ into irreducibles.
\end{theorem}
Before the proof,   some remarks are in order.

\begin{remark}
In~\eqref{decomK} it is understood that, for each $n\in\cN$ and
$j=1,\ldots,m_n$, the field of Hilbert subspaces $z\mapsto \cK_{z}^n\otimes \C\{e_j\}$
is $\lambda$-measurable.
\end{remark}

\begin{remark}\label{dec-induco} For each ${n\in\cN}$ and for almost every $z\in Z$ we denote by $\cH_{z}^{n}$ the Hilbert space carrying the induced
representation $\text{Ind}_{\R^d\rtimes H_{z}} (e^{-2\pi
  i\scal{o(z)}{\cdot}}\, \Lambda_{z}^{n})$. Reasoning as in
the proof of Theorem~10.1 of
\cite{mackey52}, for each $n \in \cN$,
$z\mapsto \cH_{z}^{n}$ is a $\lambda$-measurable field of Hilbert
subspaces, $\cH_{z}^{n}\subset \cH_{z}$, and
\begin{align}
  \cH & =\bigoplus_{{n\in\cN}}\int_{ Z} \cH_{z}^{n} \ d\lambda(z)  \otimes\C^{m_n}\label{decomH}\\
   W & =\bigoplus_{{n\in\cN}} \int_{ Z} \operatorname{Ind}_{\R^d\rtimes H_{z}} (e^{-2\pi  i\scal{o(z)}{\cdot}}\, \Lambda_{z}^{n})\ d\lambda(z) \otimes\id\label{decomW}
\end{align}
where, by Theorem~14.1 of \cite{mackey52}, each component $\operatorname{Ind}_{\R^d\rtimes H_{z}} (e^{-2\pi  i\scal{o(z)}{\cdot}}\, \Lambda_{z}^{n})$ is irreducible and
two of them are inequivalent provided that they are
different from zero (see the next remark).
\end{remark}

\begin{remark}\label{zero}
In the statement of Theorem~\ref{pazzini}, given $n\in\cN$,  it is
possible that for some $z \in Z$ the Hilbert space
$\cK_{z}^{n}$ reduces to zero as well as  $\cH_{z}^{n}$.
If this is the case, then clearly
$\Lambda_{z}^{n}$  and $\operatorname{Ind}_{\R^d\rtimes H_{z}} (e^{-2\pi  i\scal{o(z)}{\cdot}}\, \Lambda_{z}^{n})$ can be removed from
the corresponding integral decompositions of $\Lambda_{z}$ and
$W$. 
\end{remark}
\begin{remark}\label{IS}
Fix $z$ and compare \eqref{dec_inducente} with \eqref{decomL}.
The set $\cN$ is a parametrization of the relevant elements in the dual group
$\widehat{H_{z}}$ defined by the direct 
decomposition of $\Lambda_{z}$ into its irreducible components
$\Lambda_{z}^{n}$. In other words, for each $n\in\cN$ for which $\cK_{z}^{n}\not=0$ there exists $\hat s_n\in\widehat{H_{z}}$ such that $\Lambda_{z}^{n}=\Lambda_{z,\hat s_n}$ and ${m_n}=m_{\hat s}$ is its multiplicity, which is independent of $z$ by its very construction.
\end{remark}
\begin{remark}\label{treauno}
  As a consequence of Theorem~\ref{pazzini} and general results on
  direct integrals, for each ${n\in\cN}$ there exists a $\lambda$-measurable field
  $\{z\mapsto \eps^{n}_{z,\ell}\}_{\ell\geq 1}$ of Hilbert bases for each field
  $z\mapsto \cH_{z}^{n}$ and,  for any $F\in\cH$,
  \begin{align}
    F & = \sum_{{n\in\cN}}\sum_{j=1}^{{m_n}} \int_{ Z} 
    F^{n}_{z,j}\,d\lambda(z)\otimes e_{j} \label{lunga} \\
    F^{n}_{z,j}& =\sum_{\ell\geq 1} f^{n}_{j,\ell}(z)\, \eps^{n}_{z,\ell}\nonumber
  \end{align}
  where $z\mapsto f^{n}_{j,\ell}(z)$ is a $\lambda$-measurable complex
  function and
\begin{align}
  \lVert F\rVert_{\cH}^2= \sum_{{n\in\cN}} \sum_{j=1}^{{m_n}}  \int_{ Z}\lVert F^{n}_{z,j}\rVert_{\cH^{n}_{z}}^2
  d\lambda(z)= \sum_{{n\in\cN}}\sum_{j=1}^{{m_n}}  \sum_{\ell\geq
    1} \int_{ Z}  \lvert f^{n}_{j,\ell}(z)\rvert^2
  d\lambda(z).\label{norma-lunga}
\end{align}
Conversely, if $\{z\mapsto f^{n}_{j,\ell}(z)\}_{n,j,\ell}$ is
a family of $\lambda$-measurable complex functions such that
$$
\sum_{{n\in\cN}}\sum_{j=1}^{{m_n}}  \sum_{\ell\geq
    1} \int_{ Z}
\lvert f^{n}_{j,\ell}(z)\rvert^2 d\lambda(z) <+\infty,
$$
then~\eqref{lunga} defines an element $F\in\cH$.
\end{remark}

\begin{proof}[of Theorem~\ref{pazzini}] 
 We claim that there exists a sequence  of
Borel measurable functions $\xi_k:Y\to H$ such that, for any $y\in Y$, the set
$\{\xi_k(y)\}_{k\in \N}$ is dense in  $H_{y}$.  To this end, define $\Xi: Y\times H\to Y\times Y$ by $\Xi(y,h)=(h[y],y)$, a continuous map, hence Borel
measurable. Now, the diagonal $D=\{(y,y)\mid y\in Y\}$ is a
Borel set and $H_y=\{h\in H\mid \Xi(y,h)\in D\}$ for any $y\in Y$.  
By Aumann's measurable selection principle (see 
e.g. Theorem~III.23 of \cite{cava77})  the desired sequence exists.

For all $z\in  Z$,  let ${M_{z}}\subset{\mathcal L}(\cK_{z})$ denote the
von Neumann algebra on $\cK_{z}$ generated by the representation $\Lambda_{o(z)}$
of $H_{o(z)}$.  We show that $z\mapsto  M_{z}$ is a $\lambda$-measurable
field of von Neumann algebras.  For each $z\in  Z$ the continuity of  $s\mapsto \Lambda_{z,s}$ implies that  the family $\{ \Lambda_{z,\xi_{k}(o(z))}\}_{k\in\N}$ generates ${M_{z}}$. Hence, it is enough to prove that 
for any $k\in \N$  the field of operators $z\mapsto
\Lambda_{z,\xi_{k}(o(z))}$ is $\lambda$-measurable. This means  that  for any $\varphi,\varphi'\in C_c(X)$, the map 
$$
z\mapsto  \int_X
\sqrt{\alpha(\xi_k(o(z))^{-1})\beta(\xi_k(o(z))^{-1})}\varphi(\xi_k(o(z))^{-1}.x)
\overline{\varphi'(x)}d\nu_{o(z)}(x)
$$
is $\lambda$-measurable.  First we claim that
\[
(y,h)\mapsto \int_X 
\varphi(h^{-1}.x) \overline{\varphi'(x)}d\nu_{y}(x)
\]
is continuous on $Y\times H$. Fix $(y_0,h_0)\in Y \times H$ and
$\eps>0$. By (iii) of Theorem~\ref{corcoar} applied to
$\varphi^{h_0}\overline{\varphi'}\in C_c(X)$  there exists a compact
neighbourhood $U$ of $y_0$ such that for all $y\in U$
\[
\lvert  \int_X 
\varphi(h_0^{-1}.x) \overline{\varphi'(x)}d\nu_{y}(x)- \int_X 
\varphi(h_0^{-1}.x) \overline{\varphi'(x)}d\nu_{y_0}(x) \rvert\leq \eps/2.
\]
Choose a compact neighbourhood $V$ of $h_0$ e define 
$K=V.\supp{\varphi}$, which is a compact subset of $X$. 
The map $y\mapsto (\nu_y)_{K}$ is continuous from $U$ to
$M(K)=C(K)^*$ with respect to the weak* topology, so that $\sup_{y\in U}\nu_y(K)$ is bounded
(Corollary II.4 of \cite{bre83}).  Now, 
the map $h\mapsto \varphi^h$ is uniformly continuous.  Hence there is  a compact
neighbourhood $V'\subset V$ of $h_0$ such that, for all $h\in V'$,
$\varphi(h^{-1}.x)=\varphi(h_0^{-1}.x)=0$ if $x\not\in K$ and
\[
\sup_{x\in X}\lvert\varphi(h^{-1}.x)-\varphi(h_0^{-1}.x)\rvert\leq
\frac{\eps}{2(1+\sup_{x\in X}\lvert\varphi'(x)\rvert \sup_{y\in U}\nu_y(K))}.
\]
The triangular inequality gives that for all $(y,h)\in U\times V$
\[
\lvert  \int_X 
\varphi(h^{-1}.x) \overline{\varphi'(x)}d\nu_{y}(x)- \int_X 
\varphi(h_0^{-1}.x) \overline{\varphi'(x)}d\nu_{y_0}(x) \rvert\leq \eps,
\]
so that the claim is proved. Since $h\mapsto \sqrt{\alpha(h)^{-1}\beta(h^{-1})}$ is continuous and 
${z \mapsto (o(z),\xi_k(o(z)))}$ is Borel measurable
from $ Z$ to $Y\times
H$,  it follows that  $z\mapsto \Lambda_{z,\xi_{k}(o(z))}$ is a Borel measurable  field of operators and, hence, $\lambda$-measurable.

 Proposition~1, Ch~II \S~3.2 of \cite{dix57}  shows that
 $M:=\int_{ Z} {M_{z}}\, d\lambda(z)$ is a 
 von~Neumann algebra acting on $\cK$. Since $ Z$ is second countable,
Theorem~4, Ch~II \S~3.3 of \cite{dix57} implies that
\begin{align*}
  M' & = \int_{ Z} {M_{z}}'\, d\lambda(z), \\
  M\cap M' & = \int_{ Z} {M_{z}}\cap {M_{z}}'\,d\lambda(z).
\end{align*}
Further, both $M$ and $M'$ are
type I von Neumann algebras. Indeed, for almost every $z\in Z$,
$H_{z}$ is a group of type I, hence $\Lambda_{z}$ is a representation of type
I, that is, $M_z$ is a type I von
Neumann algebra. Corollary~2, Ch~II \S~3.5  of \cite{dix57} implies that $M$ is of type I , again because $Z$ is second countable. Finally, $M'$ is of type I  by Theorem~1 Ch.~1 \S~8 of \cite{dix57}.

By applying twice (A50) of \cite{dix64} we infer that there exists a countable
family $\{P^i\}_{i\in I}$ of non-zero pairwise orthogonal projections in $M\cap M'$ with sum
the identity such that both the reduced algebra\footnote{The algebra of operators obtained by restricting to the subspace $W_i=P_i\cK$ and then projecting back to $W_i$, hence a von Neumann algebra on $W_i$.}
${M^i:=M_{P^i}}$  and the reduced algebra $(M^i)'=(M')_{P^i}$
are homogeneous. Since $M$ and $M'$ are decomposable, for each $i\in I$ the decomposition
$ P^i =\int_{ Z} P^i_{z}\,d\lambda(z)$ holds,
where, for all $z\in Z$, $P^i_{z}$ is a projection in $M_{z}\cap M_{z}'$ and $z\mapsto P_{z}^i$ is a $\lambda$-measurable field of operators. 
Proposition~6 Ch.~II
\S~3.5 of \cite{dix57} implies that
\[ M^i=\int_{ Z} M^i_{z}\,d\lambda(z)\qquad
{M^i}'=\int_{ Z} {M^i_{z}}'\,d\lambda(z)\] 
where, for all $z\in Z$,  $M^i_{z}$ is the reduced algebra associated with $P^i_{z}$.  \\
Furthermore,
for almost every $z\in Z$ , the  family
$\{P^i_{z}\}_{i\in I}$ is pairwise orthogonal with sum the identity.
Indeed, given $i,j\in I$ with $i\neq j$, Proposition~3 Ch.~2
  \S~2.3 in \cite{dix57} gives $0=P_iP_j=\int
  P_{z}^iP_{z}^jd\lambda(z)$, hence the Corollary of the cited section ensures that
  $P_{z}^iP_{z}^j=0$ for almost all $z$. Since $I$ is countable, then the  above equality holds almost
  everywhere for all $i,j\in I$. Given such a $z$, 
$\{P^i_{z}\}_{i\in I}$ is a family of pairwise orthogonal
projections, so that $\sum_i P_{z}^i$ converges to a projection
$P_{z}$ with respect to the strong operator topology, and so does
$\sum_i P^i$ converge to the identity. Proposition~4 Ch.~2 in \cite{dix57}
  \S~2.3 and the uniqueness of the limit imply that  $P_{z}=\id$
  for almost every $z$.

Fix $i\in I$.  Since $M^i$ and ${M^i}'$ are
homogeneous,  the very definition of homogeneous  von~Neumann algebra (see Ch.~3, \S~3.1 in \cite{dix57}) and the canonical isomorphism given by Proposition~5
Ch.~1, \S~2.4 in~\cite{dix57},  give
\begin{align*}
  P^i\cK & = \C^{d_i} \otimes \C^{m_i}\otimes \mathcal T^i \\
  M^i  & = \mathcal{L}( \C^{d_i} )\otimes  \C\,
{\id}_{\C^{m_i} }\otimes
\mathcal{A}^i \\
{M^i}' & =\C\,
{\id}_{\C^{d_i}}\otimes \mathcal{L}( \C^{m_i}) \otimes
\mathcal{A}^i,
\end{align*}
where $\mathcal{A}^i$
is a maximal abelian algebra acting on a suitable closed subspace $\mathcal
T^i\subset\cK$.   Denote by $Q$  the orthogonal projection onto
$\C e_1\otimes \C^{m_i}\otimes\mathcal T^i\in M^i$, where  $e_1$ is  the first 
element of the canonical basis of any $\C^p$.  The corresponding
reduced algebra of $M^i$ is ${\id}_{\C^{d_i} }\otimes \mathcal{A}^i$.
 Furthermore, if $\widehat{Q}$ is
the orthogonal projection onto $\C e_1\otimes\mathcal T^i\in (M^i_Q)'$, 
the corresponding reduced algebra of $(M^i_Q)'$ is $\mathcal{A}^i$. 
Hence, reasoning as before, $\mathcal T^i$ is a direct integral of
a $\lambda$-measurable
field $z\mapsto \cT_{z}^i$ of Hilbert subspaces of $P_{z}^i\cK_{z}$ and
$\mathcal{A}^i$ is a decomposable algebra, so that 
\[ \mathcal{A}^i= \int_{ Z} \mathcal A_{z}^i\,d\lambda(z)\]
where, for almost every $z$,  $\mathcal A_{z}^i$ is a maximal
von Neumann algebra on $\mathcal T_{z}^i$.
Proposition~3 Ch.~2 \S~3.4 in \cite{dix57} gives
\[ M^i= \int_{ Z} \mathcal{L}( \C^{d_i} )\otimes  \C\,
{\id}_{\C^{d_i} }\otimes
\mathcal A_{z}^i \,d\lambda(z),\]
so that (ii) of Proposition~1 Ch.~2 \S~3.4 implies, for almost
every $z$,
\begin{equation}
P_{z}^i\cK_{z}  = \C^{d_i} \otimes \C^{m_i}\otimes \mathcal
T_{z}^i  \qquad
M^i_{z}= \mathcal{L}( \C^{d_i} )\otimes  \C\, 
{\id}_{\C^{m_i} }\otimes
\mathcal{A}^i_{z}.
\label{dima}
\end{equation}
Hence,  Proposition~3 Ch.~III \S~3.2  and Theorems~1 and~2 of Ch.~1
\S.~7.3 give the existence of a unitary operator $J^i_{z}$ from $P^i_{z}\cK_{z}$ onto $\C^{d_i}\otimes \C^{m_i}\otimes L^2(\Omega_{z}^i,\omega^i_{z})$ such that 
\begin{equation}
J^i_{z}\mathcal{M}^i_{z} {J^i_{z}}^{-1}= \mathcal{L}(\C^{d_i})\otimes  \C\,
{\id}_{\C^{m_i}}\otimes L^\infty(\Omega_{z}^i,\omega_{z}^i)
\label{dsa}
\end{equation}
where $\Omega_{z}^i$ is a locally
compact second countable space and $\omega_{z}^i$ is a measure
with support $\Omega_{z}^i$.

The previous arguments and the compactness assumption imply that there exists a negligible set $N\subset Z$ such that for all $z\in Z\setminus N$, \eqref{dima} holds true for any $i\in I$ and $H_z$ is compact.
 Hence,  with the notation
used in~\eqref{dec_inducente},  for  $z\notin N$ we put 
\[ n^i_{z}=\text{card}\{ \hat{s}\in\widehat{H_{z}}: \dim{\cK_{z,\hat{s}}}=d_i,\, m_{\hat s}=m_i\}.\]
Hence in~\eqref{dsa} the set $\Omega^i_{z}$ can be chosen as $\{1,\ldots,n^i_{z}\}$ if
$1\leq n^i_{z}<+\infty$, $\N$ if $n^i_{z}=\aleph_0$,  $\emptyset$
if $n^i_{z}=0$, and the measure $\omega^i_{z}$ as the
corresponding counting measure. 
By construction, $z \mapsto \mathcal T^i_{z}$ is a
measurable field of Hilbert spaces and Proposition~1
Ch.~II,~\S~1.4 of \cite{dix57} implies that for any cardinal $p$ the
set  ${ Z}^{ip}=\{z\in Z\setminus N: n_{z}^i=p\} $
is $\lambda$-measurable and, for each $i\in I$
\begin{equation}
  \label{diptem}
  \bigcup_{p}   Z^{ip} = Z\setminus N.
\end{equation}
Clearly for all $z\in  Z^{ip}$ we have $\Omega^i_{z}=\Omega^{ip}$, where
\[
\Omega^{ip}:=
\begin{cases}
  \{1,\ldots,p\} & 1\leq p<+\infty \\
  \N & p=\aleph_0 \\
  \emptyset &  p=0,
\end{cases}
\]
so that  the von Neumann algebra $L^\infty(\Omega_{z}^i,\omega_{z}^i)$ is equal to $\ell^{\infty}(\Omega^{ip})$, 
independently of $z$. Lemma~2 Ch.~2 \S~3. of \cite{dix57}
implies that the unitary operator $J^i_{z}:\C^{d_i}\otimes \C^{m_i}\otimes\C^p\to P^i_{z}\cK_{z}$ can be chosen in such
a way that $z \mapsto J^i_{z}$ is
$\lambda$-measurable.  The previous arguments show that the relevant indices $n=(i,p,k)$ run  on a  countable set that will be denoted $\mathcal N$. Define
$m_n=m_i$ and
\[ \cK_{z}^n=
\begin{cases}
  {J^i_{z}} \bigl(\C^{d_i}\otimes \C \{e_1\} \otimes\C\{e_k\} \bigr)& z\in Z^{ip},\ 1\leq k\leq p,\, p>0 \\
 \{0\} & \text{otherwise.}
\end{cases}
\]
Summarizing, we finally obtain the following facts, which entail the result.
\begin{enumerate}[i)]
\item For each $n\in \mathcal N$, the map $z\mapsto \cK^n_{z}$ is
  $\lambda$-measurable since $z \mapsto J_{z}^i$ is
  $\lambda$-measurable field of operators.
\item For almost all $z\in  Z$ and for each $n
  \in\mathcal N$ the Hilbert space $\cK^n_{z}$ is
  invariant with respect to $\Lambda_{z}$, the
  corresponding restriction is irreducible  and the restriction to
$J^i_{z} \bigl(\C^{d_i}\otimes \C^{m_i} \otimes \C\{e_k\}\}
\bigr)$ is a factor representation, see~\eqref{dsa}. Thus
\[ J^i_{z}\bigl(\C^{d_i}\otimes \C^{m_i} \otimes \C\{e_k\}\bigr)=
\cK_{z}^n \otimes  \C^{m_n} \qquad \Lambda_{|\cK_{z}^n
  \otimes  \C^{m_n}}=\Lambda_{|\cK_{z}^n}\otimes \id.\]
In particular, for each $j=1,\ldots,m_n$ the field $z\mapsto
J^i_{z}\bigl(\C^{d_i}\otimes \C\{e_j\} \otimes \C\{e_k\})$ is $\lambda$-measurable.
\item For almost all $z\in  Z$, for each $n\neq n'$ the
  restriction of $\Lambda_{z}$ to $\cK^n_{z}$ and
  $\cK^{n'}_{z}$ are inequivalent, provided that both spaces are
  different from zero \eqref{dsa}.
\item For almost all $z\in  Z$, $\{\cK_{z}^n
  \otimes  \C^{m_n} \}_{n\in\mathcal N}$ is a family of pairwise orthogonal closed
  subspace with sum $\cK_{z}$, by \eqref{diptem} and
  the definition of $\cK_{z}^n$.
\end{enumerate}
\end{proof}

By means of the intertwining operator $S$ given by
Theorem~\ref{intertwine}, the direct decomposition~\eqref{decomW}
gives rise to a corresponding decomposition of the
mock-metaplectic representation $U$.  Hence, the abstract theory of
\cite{fuhr05}  applies and one can characterize the admissible vectors for
$U$. However, we can apply directly Corollary~\ref{compatto}. 
We need a last technical lemma concerning the measurability of the
map $z\mapsto \vol(H_{z})$ (compare with Lemma~18 of \cite{fu09}).
\begin{lemma}\label{luglio}
Assume that for almost every $y_0\in  Y$ the stabilizer $H_{y_0}$ is compact and
define 
\[
\vol(H_{y_0}) = \int_{H_{y_0}} ds
\]
where $ds$ is the unique Haar measure of $H_{ y_0}$ such that
\[
\int_H \varphi(h) \alpha(h^{-1}) dh = \int_{Y} \left(\int_{H_{ y_0}}
\varphi(h(y)s) ds\right)d\tau_{\pi(y_0)}(y)\qquad \varphi\in C_c(Y).
\]
Then:
\begin{enumerate}[(i)]
\item for all $y_0$ and $h\in H$, $\vol(H_{h[y_0]})=\Delta_G(h^{-1}) \vol(H_{y_0})$;
\item the map $y_0\mapsto \vol(H_{y_0})$ is Lebesgue measurable.
\end{enumerate}
Furthermore, given a Borel measurable section $o: Z\to Y$, the map 
\[z\mapsto \frac{\dim{\cK_{z}^{ n}}}{\vol(H_{z})}
\] 
is $\lambda$-measurable; if $G$ is unimodular,  it is independent
of the choice of~$o$. 
\end{lemma}

\begin{proof}
Fix a  continuous   $f\in L^1(Y)$ such that $f(y)>0$ for all $y\in Y$. The definition  of $\tau_{z}$ (see Theorem~\ref{Tmackey}) and~(iii) of Theorem~\ref{th:2} imply that  $f$ is $\tau_{z}$-integrable for $\lambda$-almost every  $z\in Z$. Clearly, the function $(y_0,h)\mapsto f(h[y_0]) \alpha(h^{-1})$ is continuous on $Y\times H$. Given $y_0\in Y$, let $z=\pi(y_0)$. Hence we can choose $y_0$ as the origin of $\pi^{-1}(z)$ and define $ds$ as the unique Haar measure of $H_{y_0}=H_{z}$ for which~\eqref{weil1} holds true. By (ii)~of Theorem~\ref{th:2}, for almost all $y_0\in Y$, 
\begin{align}
  0<\int_H f (h[y_0]) \alpha(h^{-1})\,dh & = \int_Y\left( \int_{H_{y_0}}
    f(h_ys[y_0])\,ds\right)d\tau_{\pi(y_0)}(y)\nonumber \\ 
    & =
  \vol(H_{y_0})\int_Y f(y) d\tau_{\pi(y_0)}(y)<+\infty\label{hofame}
\end{align}
since $h_ys[y_0]=y$; the first inequality is due to the fact $f>0$ and the last    follows from $f\in L^1(Y)$. Clearly $y_0\mapsto \int_H f (h[y_0]) \alpha(h^{-1})\,dh$ is Lebesgue-measurable as well as $y_0\mapsto \int_Y f(y) d\tau_{\pi(y_0)}(y)$ is Lebesgue measurable and strictly positive, so that $y_0\mapsto \vol(H_{y_0})$ is $\lambda$-measurable, too. The fact that  the map $z\mapsto\dim{\cK_{z}^{ n}}$ is $\lambda$-measurable for all $ n\in\cN$ is a consequence of Proposition~1,
Ch. 2 \S~1.4 of \cite{dix57}.\\
If $y_1=\ell[y_0]$ for some $\ell\in H$, whence $\pi(y_0)=\pi(y_1)$, then by \eqref{hofame}
\begin{align*}
\vol(H_{y_1})\int_Y f(y) d\tau_{\pi(y_0)}(y) & = \int_H f (h[y_1]) \alpha(h^{-1})\,dh \\ 
 (~h\mapsto h\ell^{-1}~)& = \Delta_H(\ell^{-1})\alpha(\ell) \int_H f (h[y_0]) \alpha(h^{-1})\,dh \\
 & = \Delta_G(\ell^{-1}) \vol(H_{y_0})\int_Y f(y) d\tau_{\pi(y_0)}(y).
\end{align*}
The second half of the lemma is clear.
\end{proof}

We are ready to state our main result on the admissible vectors of $G$. We distinguish according as to weather $G$ is unimodular or not. We consider first the unimodular case, compare with Eq.~(4.14) of
Theorem~4.22 in \cite{fuhr05}.
\begin{theorem}\label{MAIN1} Assume that $G$ is unimodular and that for almost every $z\in  Z$ 
the stabilizer $H_{z}$ is compact. The representation  $U$ is
reproducing if and only if the following  two conditions hold true: 
  \begin{enumerate}[(i)]
\item the integral
  \begin{equation}
  \int_{ Z} \dfrac{\card\Phi^{-1}(o(z))}{\vol{H_{z}}}\, d\lambda(z)\label{unimodular} 
    \end{equation}
    is finite;
\item for all $ n\in \cN$ and for almost every $z\in Z$ for which $\cK_{z}^{ n}\neq 0$
  \begin{equation}
{m_n} \leq \dim{\cH_{z}^{ n} } \label{uni-dim}
\end{equation}
where the notation is as  in~\eqref{decomH} and~\eqref{decomW}.
\end{enumerate}
Under the above equivalent conditions, $\eta$ is an admissible vector for $U$ if and only if
\[ S\eta=  \sum_{{ n\in\cN}}\sum_{j=1}^{{m_n}}\int\limits_{ Z} \sqrt{ \frac{\dim{\cK_{z, n}}}{\vol{H_{z}}}} \eps_{z,j}^{ n}\, d\lambda(z) \otimes e_{j}  ,\]
where  $ \{z\mapsto\eps^{ n}_{z,j}\}_{j\geq 1}$ is any measurable field of Hilbert bases for
  $z\mapsto \cH_{z}^{ n} $.
\end{theorem}

\begin{proof} 
We use the same notation as in Remark~\ref{treauno}. Theorem~\ref{adm2} and
Corollary~\ref{compatto} with $\Delta_G(h_y)=1$ give that $\eta\in
L^2(X)$ is an admissible vector for $U$ if and only if $F=W\eta \in
\cH$ satisfies the   condition that follows. Given $ n\in \cN$, 
for almost every $z\in Z$ for which  $\cK_{z}^{ n}\neq\{0\}$
(see Remarks~\ref{zero} and~\ref{IS}),  for all $i,j=1,\ldots,{m_n}$
\[ \scal{F^{ n}_{z,i}}{F^{ n}_{z,j}}_{ \cH_{z}^{ n} }=\delta_{i,j} \frac{\dim{\cK_{z, n}}}{\vol{H_{z}}},\]
that is,  the family $\{F^{ n}_{z,i}\}_{i=1}^{{m_n}}$ is orthogonal in $\cH_{z}^{ n}$ and  normalized with square norm equal to $\dim{\cK_{z, n}}/{\vol{H_{z}}}$. \\
As a consequence, if $\eta$ is an admissible vector, then clearly~\eqref{uni-dim} holds true and, by~\eqref{norma-lunga}, we have  that
\[ \lVert F\rVert^2_{\cH} = \int_{ Z} \left(\sum_{ n\in \cN}\sum_{i=1}^{{m_n}}   \frac{\dim{\cK_{z, n}}}{\vol{H_{z}}}   \right)d\lambda(z)=\int_{ Z}
\dfrac{\card\Phi^{-1}(y_0)}{\vol{H_{z}}}\, d\lambda(z),\]
and \eqref{unimodular} follows.
Conversely, define $F\in\cH$ such that, for all $j=1,\ldots,{m_n}$ and $\ell\geq 1$
\[
 f^{ n}_{j,\ell}(z)=\delta_{j,\ell}\sqrt{ \frac{\dim{\cK_{z, n}}}{\vol{H_{z}}}}  
\qquad \text{a.e.}z\in Z,
\]
which is possible due to~\eqref{uni-dim}. All the functions
$ f^{ n}_{j,\ell}$ are $\lambda$-measurable by Lemma~\ref{luglio}. Finally, \eqref{unimodular} and the last  string of
equalities imply $\lVert F\rVert^2_{\cH}<+\infty$.  
\end{proof}

We now consider the non-unimodular case.  For all $ n\in \cN$ and for
almost every $z\in Z$ we define the positive
self-adjoint injective operator $d_{z, n}$ acting on $\cH_{z}^{ n}$ by multiplication as 
in~\eqref{formaldegree2}, namely
\[
(d_{z, n}F_{z, n})(g)
=\frac{\dim \cK_{z, n}}{\vol H_{z}}\Delta_G(g)\,F_{z, n}(g) \qquad g\in G.
\]
\begin{theorem}\label{MAIN2}
Assume that $G$ is non-unimodular and that the stabilizer $H_{z}$ is
compact for almost every $z\in  Z$ . Then $U$ is reproducing and $\eta\in L^2(X)$ is an admissible vector for $U$ if and only if
$S\eta=  \sum_{{ n\in\cN}}\sum_{j=1}^{{m_n}}\int\limits_{ Z}
F^{ n}_{z,j} \,d\lambda(z) \otimes e_{j} $ is such that
\begin{enumerate}[(i)]
\item for all  $ n\in\cN$ and $i=1,\ldots,{m_n}$, the map  $z\mapsto F_{z,
     n,i}$ is a measurable field of vectors for $\{ \cH_{z}^{ n} \}$; 
\item for all$ n\in \cN$  and for almost all $z\in Z$ for which  $\cK_{z}^{ n}\neq 0$
\[
\scal{d_{z, n}^{-1/2}F^{ n}_{z,i}}{d_{z, n}^{-1/2}F^{ n}_{z,j}}_{\cH_{z, n}}
=\delta_{ij} \qquad i,j=1,\dots,m_{ n}
\]
\item 
$\displaystyle{
\sum_{{ n\in\cN}}  \sum_{j=1}^{{m_n}} \int_{ Z} \lVert F_{z.j}^{ n}\rVert_{ \cH_{z}^{ n} }^2
  d\lambda(z)<+\infty}$.
\end{enumerate}
\end{theorem}
\begin{proof} The fact that $\eta$ is admissible if and only if (i), (ii) and (iii) hold true is similar to the proof of Theorem~\ref{MAIN1}.
The non-trivial part is the existence of an admissible
vector. This fact is a consequence of Theorem~4.23 of \cite{fuhr05},
whose proof can be repeated in our  setting.  We report the main ideas. \\
Fix a strictly positive sequence such that $\sum_{ n\in\cN}\sum_{i=1}^{m_n} a_{ n,i}<+\infty$. 
For almost every $z\in Z$ the stability subgroup $H_z$ is compact, hence the modular function $\Delta_G$ defines a continuous  surjective $\hat\Delta_z:\pi^{-1}(z)\to(0,+\infty)$ by
$\hat\Delta_z(y)=\Delta_G(h(y))$,
where $h(y)[o(z)]=y$. Therefore there exists a subset $Y_{z, n,i}$ of $\pi^{-1}(z)$ with
strictly positive $\tau_{z}$-measure such that  for all $y\in Y_{z, n,i}$
\[
\sup_{y\in  Y_{z, n,i}}\hat\Delta_z(y)\leq \frac{a_{ n,i} \vol  H_{z}}{\dim \cK_{z, n}}.
\]
By Lemma~\ref{luglio} we may select a family of $\lambda$-measurable fields
$\{z\mapsto F^{ n}_{z,j}\}_{j=1}^{{m_n}}$ of vectors in $\operatorname{dom}d_{z, n}^{-1/2}$, that are 
orthonormal with respect to the scalar product induced by $d_{z, n}^{-1/2}$ with the property that the support with
respect to $\tau_{z}$ of the map $y\mapsto \lVert F^{ n}_{z,j}(h(y))\rVert_{\cK_{z,n}}^2$ is contained in $Y_{z, n,i}$. Thus, (iii)  is satisfied because
\[ \lVert F^{ n}_{z,j}\rVert_{ \cH_{z}^{ n} }^2\leq \sup_{y\in O_{z, n,i}}\frac{\dim \cK_{z, n}\Delta_G(h(y)) }{\vol  H_{z}}\leq a_{ n,i}. \]
Finally, (i) and (ii) are true by construction.
\end{proof}

\section{Examples}\label{examples}
We now discuss  the examples introduced in Section~2.

\subsection{Example~\ref{fuhr}}

Here the map $\Phi$ is the identity so that the set of critical points reduces to the
empty set and Assumption~1  is satisfied with the choice $X=Y=\R^d$
(recall that $n=d$) and $\alpha(h)\beta(h)=1$ for all $h\in H$. Assumption~2 is the fact that the
semi-direct product $\R^d\rtimes H$ is regular. In general, nothing more specific can be said on the parameter space $Z$ and the measure $\lambda$ on it,  other than what was said in the comments following Assumption~2.
Clearly, for all $y\in \R^d$, $\Phi^{-1}(y)$ is a singleton, the
corresponding measure $\nu_y$ is  trivial, so that
Theorem~\ref{adm1} states that $\eta\in L^2(X)$ is admissible for $U$, for $\lambda$-almost $z\in Z$ if and only if
\[
\int_{H} \lvert\eta(h^{-1}[y_0])\rvert^2 \,dh =1 
\]
where $y_0$ is a fixed origin in $\pi^{-1}(z)$.  Since the above
equation holds true for any other point in $\pi^{-1}(z)$, it
follows that $\eta$ is a weak admissible vector in the sense of
Definition~7 of~\cite{fu09}. 
 Theorem~6 of the cited
paper proves that Assumption~2 is essentially necessary to have
weak admissible vectors, (see the comment at the end of Section~\ref{sec-ass2}). 
Corollary~\ref{finiteness} guarantees that the stabilizers $H_{z}$
are compact for almost every $z\in Z$.  Hence the
results of Section~\ref{sec-compatto} hold true. Clearly, for almost every $\cK_z=\C$, $\cN$ is a singleton and ${m_n}=\dim(\cK_{z}^{ n})=1$, so
that $U$ is always reproducing if $G$ is non-unimodular. Otherwise, it is such if and only if
$\int_{ Z}(\vol{H_{z}})^{-1} d\lambda(z)$ is finite, which is precisely the content of Theorem~19 of
    \cite{fu09}. See also Section~5 of~\cite{fuhr05}. The presence of $\vol{H_{z}}$ is due to a different normalization of the Haar measures on the stabilizers.

\subsection{Example~\ref{scro}}
In this example $n=2$ and $d=1$ so that $U$ is not reproducing. This
fact is well known since $G$ has a non-compact center and $U$ is
irreducible. 

 \subsection{Example~\ref{eclass}}

The main result here is about groups of the form~\eqref{parabolic} with
  $n=d$, namely:
  \begin{theorem}\label{neqd}  Let $n=d$.
If the $H$-orbits of $\Phi(\R^d)$ are locally closed, the restriction
of the metaplectic representation to $G$ is reproducing if and only if $G$ is non-unimodular and  $H_y$ is compact for almost every $y\in
    \Phi(\R^d)$.
  \end{theorem}
In order to prove Theorem~\ref{neqd}, which could be stated under the slightly more general hypothesis that $\Phi$ is a homogeneous polynomial without referring to the symplectic group,  we need  an auxiliary  result which is of some interest by itself and  whose main idea goes back to \cite{LWWW}.

\begin{proposition}\label{weiss} Let $n\leq d$.
  Assume that $\Phi$ is a homogeneous map of degree $p>0$  and that the action on $\R^d$ is
  linear. If $U$ is a reproducing
  representation, then $G$ is non-unimodular.
\end{proposition}
The proof is based on the following lemma.
\begin{lemma}  Let $n\leq d$.
 Assume that $\Phi$ is a homogeneous map of degree $p>0$ and that  the action on $\R^d$ is
  linear. If $\eta$ is an admissible
 vector for $U$, then for any $\delta\in\R_+$, the dilated vector $\sqrt{\delta^{np-d}}\eta^\delta$ is also admissible.
\end{lemma}
\begin{proof} Put  $q=np-d$. The assumption of $\Phi$ implies
  that  for all $x\in {\R^d} $, $a\in \R^n$ and $\delta\in \R_+$,
  \begin{equation}
\scal{\Phi(\delta x)}{\delta^{-p}a}=\scal{\Phi(x)}{a}.\label{scaling}
\end{equation}
Clearly, $\sqrt{\delta^{q}}\eta^\delta\in L^2(X)$ and, for all $f\in
L^2(X)$, the linearity of $x\mapsto h.x$ gives 
\begin{align*}
 \int_G \lvert\scal{f}{U_g \sqrt{\delta^{q}}\eta^\delta}\rvert^2\,dg
 & = \delta^{q} \int_H\int_{\R^d} \Bigl| \int_{\R^d}  f(x)\beta(h)^{-\frac12}\times\\
 &\hskip0.4truecm\times
 e^{2\pi i \scal{\Phi(x)}{a}}\overline{\eta}(h^{-1}.(\delta^{-1}x))dx\Bigr|^2 \frac{dadh}{\alpha(h)} \\
(x\mapsto \delta x,\,a\mapsto \delta^{-p}a,\eqref{scaling}) 
& = \delta^{q+2d-np}\int_G
\lvert\scal{f^{\delta^{-1}}}{U_{g}\eta}\rvert^2 dg  \\
(\text{reproducing formula})&  = \delta^{q+2d-np} \int_{\R^d} \lvert
f(\delta x)\rvert^2dx \\
(x\mapsto \delta^{-1} x) & = \delta^{q+d-np} \ \lVert f\rVert^2 =  \lVert f\rVert^2,
\end{align*}
so that $\sqrt{\delta^{q}}\eta^\delta$ is an admissible vector for $U$.
\end{proof}
\begin{proof}[Proof of Proposition~\ref{weiss}] By contradiction,
  assume that $G$ is unimodular. 
Fix $\delta\in\R_+$. Choose an admissible vector $\eta\in L^2(X)$. Then
\begin{align*}
\int_X\lvert \eta(x)\rvert^2dx & = \delta^{-d} \int_X\lvert
\eta^{\delta}( x)\rvert^2dx \\
(\text{reproducing formula for }\eta~) & = \delta^{-d}\int_H\int_A \lvert
\scal{\eta^{\delta}}{U_{ah}\eta}\rvert^2 \frac{dadh}{\alpha(h)} \\
(a\mapsto -a,\, h\mapsto h^{-1})& = \delta^{-d}\int_H\int_A
\lvert\scal{U_{(h^\dag[a],h)}\eta^{\delta}}{\eta}\rvert^2
\frac{\alpha(h^{})\,dadh}{\Delta_H(h^{})} \\
(a\mapsto (h^\dag)^{-1}[a]) & = \delta^{-d}\int_H\int_A
\lvert\scal{U_{(a,h)}\eta^{\delta}}{\eta}\rvert^2 \Delta_G(h^{-1})
\frac{dadh}{\alpha(h)} \\
(q=np-d)&=\delta^{-q-d}  \int_G \lvert\scal{\eta}{U_{g}\sqrt{\delta^q}\eta^{\delta}}\rvert^2
dg\\
(\text{reproducing formula for }\sqrt{\delta^q}\eta) &
=\delta^{-np} \int_X\lvert \eta(x)\rvert^2dx= \delta^{-np} \lVert\eta\rVert^2.
\end{align*}
Since $\lVert\eta\rVert\neq 0$ and $np\neq 0$, this is a contradiction.
\end{proof}

  \begin{proof}[of Theorem~$\ref{neqd}$] Clearly Assumption~2 is satisfied.
    Suppose that $U$ is reproducing. Since $\Phi$ is quadratic,
    Proposition~\ref{weiss} implies that $G$ is non-unimodular and
    Theorem~\ref{nleqd} gives that the set $\cC$ of critical points is
    negligible.  The Jacobian criterion    implies that for all $y\in
    \Phi(\cR)$ the fiber $\Phi^{-1}(y)\cap\R$ is finite (see Appendix~\ref{aldo}). Theorem~\ref{adm2}
    implies that for almost all $y\in\Phi(\cR)$ equality~\eqref{SECOND}
    holds true and, as a consequence of (i) of
    Corollary~\ref{finiteness}, the corresponding stabilizer $H_{y}$ is compact. Conversely, if $G$ is non-unimodular and almost
    every stabilizer is compact, the set of critical points is a proper
    Zariski closed subset of $\R^d$, so that it is
    negligible. Theorem~\ref{MAIN2} implies that $U$ is reproducing.
  \end{proof}
  
Theorem~\ref{MAIN2}  characterizes the
admissible vectors. However, one can also  apply directly
Theorem~\ref{adm1}, taking into account that $\Phi^{-1}(y)$ is a finite set. 

\begin{corollary}
 A function $\eta\in L^2(X)$ is an admissible
vector for $U$ if and only if for $\lambda$-almost every $z\in Z$, there exists  $y\in \pi^{-1}(z)$ such that
for all points $x_{1},\ldots x_{M}\in\Phi^{-1}(y)$ 
\[
\int_H \eta(h^{-1}.x_{i})\overline{ \eta(h^{-1}.x_{j})}
\frac{dh}{\alpha(h)\beta(h)}= (J\Phi)(x_{i}) \ \delta_{ij}\qquad
i,j=1,\ldots M. 
\]
If the above equation is satisfied for a pair $x_{i},x_{j}\in
\Phi^{-1}(y)$, then it holds true for any pair $s.x_{i},s.x_{j}\in \Phi^{-1}(y)$ with $s\in H_{y}$.
\end{corollary}
\begin{proof}
We apply Theorem~\ref{adm1}. Given $z\in Z$ and
$y\in\pi^{-1}(z)$ for which~\eqref{FIRST} holds true,
formula~\eqref{eq:20} gives that  
\[
\nu_{y}=\sum_{i=1}^M
  \frac{\delta_{x_i}}{(J\Phi)(x_i)}.
  \]
Arguing as in the proof of
Proposition~\ref{compatto},~\eqref{FIRST} is equivalent to
\[
\int_H \eta(h^{-1}.x_{i})\overline{ \eta(h^{-1}.x_{j})}
\frac{dh}{\alpha(h)\beta(h)}= (J\Phi)(x_{i}) \delta_{ij}\qquad
i,j=1,\ldots N_{y}.\]
The last claim is clear because $H_{y}$ is compact so that for all
$s\in H_{z}$ we have $\alpha(s)=\beta(s)=1$ and hence  the equality
\[ (J\Phi)(h.x)= (J\Phi)(x) \alpha(h)^{-1}\beta(h)^{-1}\qquad h\in H.\]
\end{proof}

As an example, we apply the above corollary to the metaplectic
representation restricted to the shearlet group $G=TDS(2)$. % with $\gamma=1$.
Notice that
\[ (J\Phi)(x_1,x_2)=x_1^2/2\qquad \alpha(\ell,t)=t^{1+\gamma}\qquad\beta(\ell,t)=t^{-\gamma}.\]
We set $X=\{(x_1,x_2)\in\R^2\mid x_1\neq 0\}$, which is an
$H$-invariant open set with full Lebesgue measure and
$Y=\Phi(X)=\R_-\times \R$,  is a transitive free $H$-space. We
choose as origin the point $y_0=(-1/2,0)$ so that
$\Phi^{-1}(y_0)=\{(\pm 1,0)\}$. Since for any $h=(\ell,t)\in H$ 
\[ h^{-1}.(1,0)=(t^{\frac 12},  t^{\gamma-\frac 12}\ell),\]
a function $\eta\in L^2(X)$ is an admissible vector if and only if
\begin{align}
   & \int_{(0,+\infty)\times\R}\lvert \eta( t^{\frac 12},
   t^{\gamma-\frac 12}\ell)\rvert^2\,\frac{dt
     d\ell}{t^{3-\gamma}}  = \frac{1}{2}\label{shear1}\\
& \int_{(0,+\infty)\times\R}\lvert \eta( -t^{\frac 12},
   -t^{\gamma-\frac 12}\ell)\rvert^2\,\frac{dt
     d\ell}{t^{3-\gamma}}  =  \frac{1}{2}\label{shear2}\\
   & \int_{(0,+\infty)\times\R}  \eta(t^{\frac 12},  t^{\gamma-\frac 12}\ell) \overline{\eta(-t^{\frac 12}, - t^{\gamma-\frac 12}\ell)}\,\frac{dt d\ell}{t^{3-\gamma}} = 0.\label{shear3}
\end{align}
To recover the usual admissibility condition, put
$X_{\pm}=\{(x_1,x_2): \pm x_1>0\}$ and define  the unitary operator
$R_{\pm}:L^2(Y)\to L^2(X_{\pm})$
\[ (R_{\pm}\hat f)(x_1,x_2)=\hat f(\Phi(x_1,x_2) )\lvert
J\Phi(x_1,x_2)\rvert^{\frac 12},\]
so that 
\[ R_{\pm}^{-1}U_{(a;\ell,t)} R_{\pm} \hat f (y)= t^{(1+\gamma)/2}
e^{-2\pi i
  \scal{y}{a} } \hat f(t y_1,t^{\gamma}(\ell y_1+y_2)),\]
which clarifies the connection with the shearlet representation, see
\cite{gukula06}. 
Denote by $\hat \eta_{\pm}= R_{\pm}^{-1} \eta_{|X_{\pm}}$,  equations \eqref {shear1} and \eqref {shear2} become
\[ 
\int_{Y} \lvert\hat\eta_{\pm}(-\frac12 t,-\frac12
t^{\gamma}\ell)\rvert^2\frac{t^{\gamma -2}}{2}dt d\ell
=\frac 12.
\]
With the change of variables $\omega_1= -\frac12 t$ and
$\omega_2=-\frac12t^{\gamma} \ell$, whose Jacobian is $\frac{1}{4}t^{\gamma}$, they become
\begin{align*}
& \int_{\R_+\times\R}\lvert \hat\eta_{\pm}( \omega_1,\omega_2)\rvert^2\,\frac{d\omega_1 d\omega_2}{\omega_1^2}  =  1 .
\end{align*}
Similarly, \eqref {shear3} becomes
\begin{align*}
%& \int_{\R_+\times\R}\lvert \hat\eta_{-}( \omega_1,\omega_2)\rvert^2\,\frac{d\omega_1 d\omega_2}{\omega_1^2}  =  1 \\
  & \int_{\R_+\times\R}\hat\eta_+(\omega_1, \omega_2) \overline{\eta_-(\omega_1, \omega_2) }\, \frac{d\omega_1 d\omega_2}{\omega_1^2}   = 0.
\end{align*}
One should compare this with formula (2.1) in \cite{kula09}.
Note that $U$ is equivalent to two copies of the irreducible
representation $\operatornamewithlimits{Ind}_{\R^2}^{G}(\chi)$, where
$\chi$ is the character of $\R^2$ $(a_1,a_2)\mapsto e^{\pi i a_1} $. 

\subsection{Example~\ref{n>d}}
With the choice $X=\R^2\setminus\{0\}$ and $Y=\Phi(X)=(0,+\infty)$ Assumption~1 is satisfied because $X$ is an $H$-invariant open set whose complement has zero Lebesgue
  measure. The group $H$ acts freely on $Y$ so that Assumption~2 holds true and 
  $Z$ reduces to a singleton.  
We choose  $y_0=1$ as the origin of the orbit, whose stabilizer is the compact group
$H_1=\T$. Since $G$ is non-unimodular, $U$ is
reproducing by Theorem~\ref{MAIN2}. In order to characterize its
admissible vectors note that in Theorem~\ref{Tmackey} the
relatively invariant measure on $Y$ is $\tau_1=dy$. Furthermore, the map
$  \xi\mapsto (\cos\xi, \sin\xi)$
is  diffeomorphism of $S^1$ onto the Riemannian submanifold
$\Phi^{-1}(1)=\{x_1^2+x_2^2=1\}$. The Riemannian measure on $S^1$ is
$d\xi$ so that, for all $\varphi\in C_c(X)$ 
$$ \int_X \varphi(x_1,x_2) d\nu_1(x_1,x_2) 
= \int_0^{2\pi} \varphi(\cos\xi,
\sin\xi) \frac{d\xi}{2}.$$
Put $h(y)=(\sqrt{y},0)$ so that $h(y)[1]=y$. Then~\eqref{weil1} says that  the Haar measure 
on  $\T$ is $d\theta/4\pi$ because
$$ \int_H \varphi(t,\theta)  t dt \frac{d\theta}{2\pi} =\int_0^{+\infty}
\left(\int_0^{2\pi} \varphi(\sqrt{y},\theta)
  \frac{d\theta}{4\pi}\right)dy, $$
so that $\operatorname{vol}{\T}=\frac{1}{2}$.\\
The representation $\Lambda_1$ of $\T$
  on $L^2(X,\nu_1)\simeq L^2(S^1,d\xi/2)$ is the  regular
  representation, and
  \begin{align*}
    L^2(X,\nu_1) & \simeq \bigoplus_{n\in\mathbb Z} \C\,\{ e^{ i n\xi} \}\\
    \Lambda_{1,\theta} & \simeq \bigoplus_{n\in\mathbb Z} e^{ -i n\theta},
  \end{align*}
where each component is irreducible and any  two of them are inequivalent.

Since any $g=(a,t,\theta)$ can be written as  $g=(0,t,0)(t^2
  a,0,\theta)$, any function $F\in \cH$ can be
  identified with its restriction to $\R_+$ due to~(K2). Further, (K3)
becomes
$$ 
\int_0^\infty \lvert F(\sqrt{y})\rvert^2 y^{-1} dy =  \int_0^\infty
\lvert F(t)\rvert^2 2t^{-1} dt <+\infty.
$$
Hence  we have the following unitary identifications
$$\cH\simeq L^2(\R_+,2 t^{-1} dt, L^2(S^1,d\xi/ 2)) \simeq L^2(\R_+\times S^1,
t^{-1}dt d\xi ).$$
The unitary map $S:L^2(X)\to\cH$ is given explicitly by
$$ 
(S f)(t,\xi)=t \, T_{t^2,t^{-1}}( f_{1,t^2})(\xi)= t
f(t\cos\xi,t\sin\xi) .
$$
For $n\in\mathbb Z$, the space $\cH_{n}$ carrying the
  representation induced by $e^{-2\pi i a-i n\theta}$ is
$$
\cH_{n}=\{ F\in L^2(\R_+\times \T,
t^{-1}dt d\xi) \mid F(t,\xi)=F_n(t) e^{i n\xi}, F_n\in L^2(\R_+,t^{-1}dt)\}.
$$
If $\eta\in L^2(X)$,  then $S\eta=\sum_{n\in\mathbb Z} F_n e^{i
  n\xi} $ with  $F_n\in  L^2(\R_+,t^{-1}dt)$. It follows that  
$\eta$ is an admissible vector if and only if, for any $n\in\mathbb Z$,
$$\int_0^{+\infty}   \left(\int_{S^1}\lvert F_n(\sqrt{y}) e^{in\xi}\rvert^2\frac{d\xi}{2} \right)y^{-2} dy=\frac{\dim{\cK_{n}}}{\operatorname{vol}\T}
=2,$$
since $\dime{\cK_{n}}=1$.  By the change of variable $t=\sqrt{y}$, this is equivalent to 
$$\int_0^{+\infty}   \lvert F_n(t)\rvert^2 t^{-3} dt =\frac{1}{\pi}. $$
Finally, since
$$ F_n(t)=\frac{1}{2\pi}\int_0^{2\pi} t \eta(t\cos \xi,t\sin\xi)
e^{-in\xi}d\xi =:t \hat{\eta}(t,n),$$
the set of admissible vectors consists of  the Lebesgue measurable
functions $\eta:\R^2\to\C$ such that
\begin{align*}
& \sum_{n\in\mathbb Z} \int_0^{+\infty} \lvert \hat{\eta}(t,n)\rvert^2
t dt<+\infty\iff \eta\in L^2(\R^2)\\
& \int_0^{+\infty}   \lvert \hat{\eta}(t,n)\rvert^2 t^{-1} dt
=\frac{1}{\pi}\qquad \forall n\in\mathbb Z. 
\end{align*}

\subsection{Example~\ref{n<dnocom}}
 In this example Assumption~1 is satisfied with the choice 
$X=\R^2\setminus\{x_2=0\}$ and $Y=\Phi(X)=\R\setminus\{0\}$, because
$X$ is a $H$-invariant open set whose complement has zero Lebesgue
  measure. The group $H$ acts freely on $Y$ so that Assumption~2 holds true and 
  $Z$ reduces to a singleton.
We choose  $y_0=1$ as the origin of the orbit so that the
corresponding stabilizer is the non-compact group
$H_1=\R^*$.  To prove that $G$ a reproducing group, we use Theorem~\ref{adm2}. 
In Theorem~\ref{Tmackey} the
relatively invariant measure on $Y$ is $\tau_1=dy$. Furthermore, the map
$\xi\mapsto (\xi,  1)$
is  a diffeomorphism of $\R$ onto the Riemannian submanifold
$\Phi^{-1}(1)=\{x_2=1\}$. The Riemannian measure on $\R$ is
$d\xi$ and $(J\Phi)(x)=1$, so that  \eqref{eq:20} gives  for all $\varphi\in C_c(X)$ 
$$ \int_X \varphi(x_1,x_2) d\nu_1(x_1,x_2) 
= \int_\R \varphi(\xi,
1) d\xi.$$
Put $h(y)=(y,0)$ so that $h(y)[1]=y$. Then \eqref{weil1} says that  the Haar measure  on
  $\R$ is $db$ because
$$ \int_H \varphi(t,b)  \lvert t\rvert \frac{dt}{\lvert t\rvert } db=\int_{\R^*}
\left(\int_\R \varphi(y,b)\,db\right)dy.$$
The representation $\Lambda_1$ of $\R$
  on $L^2(X,\nu_1)\simeq L^2(\R,d\xi)$ is the  regular
  representation, and
    \begin{align*}
    L^2(X,\nu_1) & \simeq \int_{\R} \C \,d\omega\\
    \Lambda_{1,b} & \simeq \int_{\R} e^{ -2\pi i\omega b} \,d\omega\\
  \end{align*}
where each component is irreducible, any two of them are inequivalent and the intertwining operator is given by the  Fourier transform.

Since any $g=(a,t,b)\in G$ can be written as  $g=(0,t,0)(ta,0,b)$,  any function $F\in \cH$ can be
  identified with its restriction to $\R^*$ due to~(K2) and we have the following unitary identifications
$$\cH\simeq L^2(\R^*, t^{-1} dt, L^2(\Phi^{-1}(1),\nu_1)) \simeq L^2(\R^2,
y^{-1}dy d\xi ).$$
The unitary map $S:L^2(X)\to L^2(\R^2,y^{-1}dy d\xi )$ is given explicitly by
$$ 
(S f)(y,\xi)=\lvert t\rvert^{\frac12} f(y,\xi) .
$$
Theorem~\ref{adm2} implies that $\eta\in L^2(X)$ is an admissible vector if and only if for all $u\in L^2(\R,d\xi)$
\begin{align*}
  \int_\R \lvert u(\xi)\rvert^2\,d\xi & = \int_\R \left(\int_\R \lvert\scal{u}{\lvert y\rvert^{-\frac12}\Lambda_{1,b} (S\eta)(y,\cdot)}\rvert^2\,db\right) \lvert y\rvert^{-1 }dy \\
& = \int_\R \left(\int_\R  \lvert \hat{u}(\omega)\rvert^2 \lvert\hat{\eta}(y,\omega)\rvert^2 d\omega\right)  \lvert y\rvert^{-1 }dy 
  \end{align*}
where we use that $\Delta(h(y))= \alpha(h(y))^{-1}=\lvert y \rvert$ and where $\hat{}$ denotes the Fourier transform with respect to $\xi$.
It follows that the set of admissible vectors is the set of Lebesgue measurable
functions $\eta:\R^2\to\C$ such that
\begin{align*}
&  \int_\R \left(\int_\R \lvert\hat{\eta}(y,\omega)\rvert^2 d\omega\right)  dy <+\infty\iff \eta\in L^2(\R^2)\\
&\int_\R \lvert\hat{\eta}(y,\omega)\rvert^2 \lvert y\rvert^{-1 }dy =1 \qquad \text{for almost every }\omega\in\R.
\end{align*}
This set is clearly non empty: take for example  any strictly positive
continuous function $\sigma\in L^1(\R)$ and define 
\[
\widehat{\eta}(y,\omega)=\left(\frac{1}{\sqrt{2\pi}\sigma(\omega)} \lvert y\rvert e^{-\frac{y^2}{2\sigma(\omega)^2}}\right)^{\frac12}.
\]

%%%%%%%%%%%%%%%%%%%%%%%%%%%%%%%%
\appendix
\section{Appendix: some measure theory revisited}\label{MT}
In this Appendix we review some known facts that are somehow hard to locate in the literature in a way that is both easily accessible and stated under the assumptions that we are making.  The  spaces $X$ and $Y$ are as in Section~\ref{mainresults} and are regarded as measure spaces with respect to the Lebesgue measure, denoted $dx$ and $dy$ respectively.
\subsection{Disintegration of measures}\label{DIS}We start by adapting to our setting some facts from integration theory on general locally compact spaces.  The main reference for the issues at hand is \cite{bourbaki_int}. 
Hereafter, $C_c(X)$ denotes the space of compactly supported continuous functions on $X$, endowed with the usual locally convex (separable) inductive limit topology, for which a sequence $(\varphi_n)_{n\in\mathbb N}$ in $C_c(X)$ converges to zero if there exists a compact set $K$ such that $\text{supp}\,\varphi_n\subset K$ for all $n$  and $\lim_{n\to\infty}\sup_{x\in
  K}\lvert\varphi_n(x)\rvert=0$. We  denote by $M(X)$ the topological dual of $C_c(X)$; when equipped  with the $\sigma(M(X),C_c(X))$-topology,  the  topological dual of $M(X)$  is again $C_c(X)$ (\cite{RS80}, Th.~IV.20).  Since $X$ is second countable, the Riesz-Markov representation theorem uniquely identifies the measures with the positive elements of $M(X)$. By the word measure on a locally compact second countable topological space, we mean a positive measure defined on the Borel $\sigma$-algebra, which is finite on compact subsets.

The following theorem, in some sense a version of Fubini's theorem,
summarizes the main properties of the kind of disintegration of
measures we are concerned with. The main point here, though, is the
possibility of extending the disintegration from $C_c$ to $L^1$. We
state it for $X$ and $Y$, but it also  holds {\it verbatim} if we replace $X$ and $Y$ with two arbitrary locally compact second countable topological spaces.
\begin{theorem}\label{th:2}
Suppose that $\omega$ is a  measure on $X$ and  $\rho$ a measure on $Y$ and let $\Psi:X\to Y$ be a $\omega$-measurable map. Assume further that $\{\omega_y\}$ is a family of
measures on $X$ such that
\begin{itemize}
\item[(a)] $\omega_y$ is concentrated on $\Psi^{-1}(y)$ for all $y\in Y$;
\vskip0.2truecm
\item[(b)] ${\displaystyle\int_X \varphi(x)d\omega(x)
=\int_Y\left(\int_X\varphi(x)d\omega_y(x)\right)d\rho(y)}$
for all $\varphi\in C_c(X)$.
\end{itemize}
Then, for any $\omega$-measurable function $f:X\to\mathbb C$ the following facts hold true:
\begin{itemize}
\item[(i)]  $f$ is  $\omega_y$-measurable  for almost every $y\in Y$;
\item[(ii)] $f$ is $\omega$-integrable if and only if  $\displaystyle{ \int_Y\left(\int_X\lvert
    f(x)\rvert d\omega_y(x)\right)d\rho(y)}$ is finite;
\item[(iii)] if $f$ is $\omega$-integrable, then $f$ is $\omega_y$-integrable for $\rho$-almost every $y\in Y$, the function (defined almost everywhere) $y\mapsto\int_X f(x)d\omega_y(x)$ is $\rho$-integrable, and
\begin{equation}
\int_X f(x)d\omega(x)=\int_Y\left(\int_X f(x)d\omega_y(x)\right)d\rho(y)\label{eq:3};
\end{equation}
\item[(iv)] if  $\{\omega'_y\}$  is another family of measures on $X$ satisfying (a) and (b), then $\omega'_y=\omega_y$ for $\rho$-almost all $y\in Y$.
\end{itemize}
\end{theorem}
\begin{proof}
The theorem is essentially contained in \cite{bourbaki_int}, scattered in several statements.  For the proof of (i), (ii) and (iii) we quote from Chapter~5, and for the proof of (iv) from Chapter~6.

Statement (i) is the content of~a) Prop.~4, \S~3.2, taking into account that, since it is second countable, $X$ is $\sigma$-compact and, a fortiori, $\omega$-moderated (a subset is $\omega$-moderated if  it is contained into the union of a countable sequence of compact subsets  and a $\omega$-negligible set).

As for (ii), since $X$ is second countable, Prop.~2, \S~3.1,  guarantees
that the family $\int_X\varphi(x)d\omega_y(x)$ is  $\rho$-adequate in the sense of Def.~1, \S~3.1. The equivalence of the two conditions in (ii)  is then the content of the Corollary at the end of
\S~3.2. 

As for (iii), it is just Th.~1, \S~3.3,  observing that any function is $\omega$-moderated since $X$ is $\omega$-moderated (a function is $\omega$-moderated if it is null on the complement of a $\omega$-moderated subset).

Finally, for (iv), by assumption $\int_Y \omega_yd\rho(y)=\int_Y \omega'_yd\rho(y)$, where the integral is a scalar integral of vector valued functions taking values in $M(X)$.  Now Lemma~1, \S~3.1 ensures that $C_c(X)$ has a countable subset  which is dense\footnote{\label{separa} It is proved there that there exists a  countable subset $S\subset C_c(X)$ such that for every $\varphi\in C_c(X)$ there is a sequence $(\varphi_n)_{n\in\mathbb N}$ in $S$ converging to $\varphi$ uniformly and $\lvert \varphi_n\rvert\leq \lvert \varphi_0\rvert$.} in $C_c(X)$ with respect to the  $\sigma(C_c(X),M(X))$ topology, so that, by Remark~2 in \S 1.1,  it is enough to show that for any $\varphi\in C_c(X)$ and for $\rho$-almost every $y\in Y$
$$
\int_X \varphi(x)d\omega_y(x)=\int_X \varphi(x)d\omega'_y(x).
$$
This   is in turn equivalent to proving that 
 \begin{equation}
\int_Y\left(\int_X \varphi(x)d\omega_y(x)\right)\xi(y)d\rho(y)=
  \int_Y\left(\int_X \varphi(x)d\omega'_y(x)\right)\xi(y)d\rho(y)\label{eq:4c}
\end{equation}
holds for all $\varphi\in C_c(X)$ and $\xi\in C_c(Y)$.   Fix then   $\varphi\in C_c(X)$ and $\xi\in C_c(Y)$, and put $f(x)=\xi(\Psi(x))\varphi(x)$. This function is $\omega$-measurable since $\Psi$ is $\omega$-measurable and $\xi$ and $\varphi$ are continuous, it is bounded since both $\xi$ and $\varphi$ are bounded, and it has a compact support since $\varphi$ is compactly
supported. Hence $f$ is $\omega$- integrable. Applying twice~\eqref{eq:3} we get
\begin{equation}
\int_Y \left(\int_X
  \xi(\Psi(x))\varphi(x)d\omega_y(x)\right)d\rho(y)
= \int_Y\left(\int_X\xi(\Psi(x))\varphi(x)d\omega'_y(x)\right)d\rho(y).\label{eq:22}
\end{equation}
Given $y\in Y$, (a) implies that $\xi(\Psi(x))=\xi(y)$ for $\omega_y$-almost all $x\in X$, so that
$$
\int_Y \left(\int_X
  \xi(\Psi(x))\varphi(x)d\omega_y(x)\right)d\rho(y)=\int_Y
\left(\int_X \varphi(x)d\omega_y(x)\right)\xi(y)d\rho(y),
$$
and similarly for the right hand side of~\eqref{eq:22}. Hence~\eqref{eq:4c} is true and the claim is proved.
\end{proof}

The integral formula (b) will be written for short
\begin{equation}
d\omega=\int_Y\omega_y\,d\rho(y).
\label{generaldis}
\end{equation}

%%%%%%%%%%%%%%%%%%%%%%%%%%%%%%%%%%%%%%%%%%%%
\subsection{Direct integrals}\label{DI}
Next we recall the definition of direct integral, following \cite{fol95}. Hereafter we assume that  the hypotheses of Theorem~\ref{th:2} are satisfied.
 Fix a countable family $\{\varphi_k\}_{k\in\N}$ dense in $C_c(X)$, and hence also in every $L^2(X,\omega_y)$,  with $y\in Y$. The map $y\mapsto \scal{\varphi_k}{\varphi_{\ell}}_{\omega_y}$ is $\rho$-measurable since it is $\rho$-integrable by  hypothesis (b) of Theorem~\ref{th:2}. 
 Under these circumstances, $\{\varphi_k\}_{k\in\mathbb N}$ is called a $\rho$-measurable structure for the family of Hilbert spaces $\left\{L^2(X,\omega_y)\right\}$.  The direct integral $\int_Y L^2(X,\omega_y)dy$ is defined as the set consisting of all the families $\{f_y\}$  satisfying: 
\begin{itemize}
\item[(D1)] $f_y\in L^2(X,\omega_y)$ for all $y\in Y$;
\vskip0.2truecm
\item[(D2)] ${\displaystyle\int_Y \lVert f_y\rVert^2_{\omega_y} d\rho(y)<+\infty}$;
\vskip0.2truecm
\item[(D3)] $y\mapsto \scal{f_y}{\varphi_k}_{\omega_y}$ is $\rho$-measurable for all $k\in\N$.
\end{itemize}
\vskip0.2truecm
Two families $\cF=\{f_y\}$ and $\cG=\{g_y\}$ are identified if for almost every  $y\in Y$  $f_y=g_y$ as elements in $L^2(X,\omega_y)$. The space $\int_Y L^2(X,\omega_y)d\rho(y)$ is a Hilbert space under
$$
\scal{\cF}{\cG}=\int_Y
\scal{f_y}{g_y}_{\omega_y}d\rho(y).
$$
Since $C_c(X)$ has a dense countable subset, see Footnote~\ref{separa},  (D3) is equivalent to 
\begin{itemize}
\item[(D3')]  $y\mapsto \scal{f_y}{\varphi}_{\omega_y}$ is $\rho$-measurable for all $\varphi\in C_c(X)$,
\end{itemize}
so that, as long as we choose the functions of $\{\varphi_k\}_{k\in\mathbb N}$ in $C_c(X)$, the measurable structure is independent of the choice of the particular family.
\begin{proposition}\label{directintegral}
 Given  $f\in  L^2(X,\omega)$, there exists a  unique family $\{f_y\}$ in  the Hilbert space direct integral $\int_Y L^2(X,\omega_y)d\rho(y)$ such that, for almost every $y\in Y$, the equality $f_y(x)=f(x)$ holds for $\omega_y$-almost every $x\in X$. Furthermore, the map $f\mapsto \{f_y\}$ is a unitary operator from $L^2(X,\omega)$ onto $\int_Y L^2(X,\omega_y)d\rho(y)$.  
\end{proposition}
\begin{proof}%[Proof of Prop.~\ref{prop:direct integral}]
By hypothesis (b) of Theorem~\ref{th:2},  for every $\varphi\in C_c(X)$ we have
$$
\int_X\varphi(x)d\omega(x)
=\int_Y\left(\int_{X} \varphi(x)d\omega_y\right)d\rho(y).
$$
 Given a function\footnote{Here it is important that $f$ is a function, and not an equivalence class modulo a.e. equality.} $f:X\to \mathbb C $ which is square-integrable with respect to $\omega$, hence in particular $\omega$-measurable, (i) of  Theorem~\ref{th:2}  implies that $f$ is $\omega_y$-measurable for  almost every $y\in Y$. Further, since $|f|^2$ is integrable with respect to $\omega$,  (iii) of  the same theorem
ensures that $|f|^2$ is  $\omega_y$-integrable for almost all $y\in Y$, the map $y\mapsto \int_X\lvert f(x)\rvert^2d{\omega_y}(x) $ is  integrable, and  
\begin{equation}\label{eq:4}
\int_X \lvert f(x)\rvert^2\,d\omega(x)=\int_Y \left(\int_X\lvert f(x)\rvert^2 d{\omega_y}(x)\right)d\rho(y).
\end{equation}
Hence there is a $\rho$-full set $Y'\subset Y$ such that, if $y\in Y'$,  $f$ is square-integrable with respect to $\omega_y$. For $y\in Y'$ define $f_y$ to be the equivalence class of $f$ in $L^2(X,\omega_y)$ and, for $y\not\in Y'$, put $f_y=0$. 

We claim that $\cF=\{f_y\}$ is in $\int_Y L^2(X,\omega_y)d\rho(y)$. By~\eqref{eq:4}, conditions (D1)  and (D2) are clearly satisfied. To prove (D3'), take $\varphi\in C_c(X)$. Clearly, $f\overline{\varphi} $ is $\omega$-integrable and hence, by (iii) of  Theorem~\ref{th:2}, it is $\omega_y$-integrable for almost every $y\in Y$ and 
$$
y\mapsto \int_X f(x)\overline{\varphi(x)}d\omega_y(x)=\scal{f_y}{\varphi}_{\omega_y}
$$
is integrable, hence measurable. Therefore  $f\mapsto \cF$ is a well defined map from the space of  square-integrable functions on $X$ to $\int_YL^2(X,\omega_y)d\rho(y)$, it is linear and,  by \eqref{eq:4}, 
\begin{equation}
\int_X\lvert f(x)\rvert^2d\omega(x)=\int_Y \lVert f_y\rVert^2_{\omega_y}d\rho(y).
\label{elltwo}
\end{equation}
Hence, it defines an isometry from $L^2(X,\omega)$ into $ \int_YL^2(X,\omega_y)d\rho(y)$ and, by construction,  for almost every $y\in Y$, the equality $f_y(x)=f(x)$ holds for $\omega_y$-almost every $x\in X$.

We claim that the isometry  $f\mapsto \cF$ is surjective. It is enough to prove that for any family $\cF$ whose members $f_y$ are positive, there exists a positive   $f\in L^2(X,\omega)$ such that, for almost every $y\in Y$, the equality $f_y(x)=f(x)$ holds for $\omega_y$-almost every $x\in X$. Take then such an $\cF$.
First of all, we show  that the family of measures $\{f_y\cdot\omega_y\}$ is scalarly integrable with respect to $\rho$.  This is equivalent to saying that for all $\varphi\in C_c(X)$ the function $y\mapsto F_\varphi(y)=\int_X\varphi(x)f_y(x)d\omega_y(x)$, certainly well defined  because  (D1) implies that  $\varphi f_y$ is $\omega_y$-integrable for every $y\in Y$, is $\rho$-integrable. Indeed,~(D3') says that $F_\varphi$ is 
$\rho$-measurable, whereas H\"older's inequality and Cauchy-Schwartz give
\begin{align*}
\int_Y \lvert F_\varphi (y)\rvert d\rho(y)
&\leq \int_Y\lVert\varphi\rVert_{\omega_y}\lVert f\rVert_{\omega_y} d\rho(y)\\
&\leq \left(\int_Y \lVert\varphi\rVert^2_{\omega_y}d\rho(y)\right)^{1/2}
\left( \int_Y\lVert f\rVert^2_{\omega_y} d\rho(y)\right)^{1/2}
\end{align*}
so that by (D2) and~\eqref{elltwo} applied to $\varphi$ yield
$$
\int_Y \lvert F_\varphi (y)\rvert d\rho(y)
\leq C \lVert\varphi\rVert<+\infty.
$$
Hence the claim is proved and  $\mu=\int_Y (f_y\cdot\omega_y)d\rho(y)$ defines a measure. We show next that $\mu$ is a measure with base\footnote{A measure which is the product  $\psi\cdot\cL$ of a measure $\cL$ by a locally $\cL$-integrable positive function $\psi$ is called a measure with base $\cL$ (see Def.~2, \S~5.2, Ch.~V in \cite{bourbaki_int}).}$\omega$. This will produce the
required $f$ that maps to $\cF$. The Lebesgue-Nikodym theorem (see Th.~2 ,\S~5.5, Ch.~5 of \cite{bourbaki_int}) ensures that it is enough to prove that  any compact subset $K\subset X$ for which $\omega(K)=0$ satisfies $\mu(K)=0$. Take such a $K$. Item (iii) of Theorem~\ref{th:2} applied to the characteristic function $\chi_K$ gives that for almost every $y\in Y$, $K$ is $\omega_y$-negligible and, {\it a fortiori}, $f_y\cdot\omega_y$-negligible. Thus, \eqref{eq:3}  with $\omega=\mu$, $\omega_y=f\cdot\omega_y$ and $f=\chi_K$ yields
$$
\mu(K)=\int_Y  \left(\int_K f_y(x)d\omega_y(x)\right)d\rho(y)=0.
$$
Hence there exists a locally integrable positive function $f$ such that $f\cdot\omega=\mu$. Moreover, if $\varphi\in C_c(X)$, $\varphi f$ is integrable, so that again (iii) of Theorem~\ref{th:2} tells us  that, for almost every $y\in Y$, $\varphi f$ is $\omega_y$-integrable, the map $y\mapsto \int_X\varphi(x)f(x)\,d\omega_y(x)$ is integrable and by definition of $\mu$
\begin{align*}
 \int_Y \left(\int_X\varphi(x) f_y(x)\,d\omega_y(x)\right)d\rho(y)
&=\int_X \varphi(x) \,d\mu(x)\\
&= \int_Y \left(\int_X\varphi(x) f(x)\,d\omega_y(x)\right)d\rho(y).
\end{align*}
By the above equality, (iv) of Theorem~\ref{th:2}  may be applied to infer that for almost every $y\in Y$ the equality $f=f_y$ holds  $\omega_y$-almost everywhere. Finally, (D2) gives
$$
\int_Y\left(\int_X\lvert f(x)\rvert^2 d\omega_y(x)\right)d\rho(y)
= \int_Y\left(\int_X\lvert f_y(x)\rvert^2 d\omega_y(x)\right)d\rho(y)<+\infty.
 $$
Hence (iii) of Theorem~\ref{th:2} implies that $f$ is square integrable. The  equivalence class of $f$ in $L^2(X,\omega)$ is then the  element required to prove surjectivity.
\end{proof}

Both  $L^2(X,\omega)$ and each of  the spaces $L^2(X,\omega_y)$ can be identified with subspaces of $M(X)$ simply by  viewing their elements as continuous linear functionals on $C_c(X)$ via  integration with respect to $\omega$ and $\omega_y$, respectively. Further, (iv) of Theorem~\ref{th:2} implies that saying that for almost every $y\in Y$ the equality $f_y(x)=f(x)$ holds for $\omega_y$-almost every $x\in X$ is equivalent to 
$$
f\cdot\omega=\int_Y(f_y\cdot\omega_y)\,d\rho(y),
$$
in the sense that the map $Y\to M(X)$, $y\mapsto f_y\cdot\omega_y$ is  $\rho$-scalarly-integrable.
These remarks together with Proposition~\ref{directintegral} imply that
\begin{equation}
L^2(X,\omega)=\int_Y L^2(X,\omega_y)d\rho(y)
\label{L2}
\end{equation}
by means of the equality in $M(X)$
\begin{equation}
f=\int_Y f_y d\rho(y),
\label{L22}
\end{equation}
where the integral is a scalar integral.
%%%%%%%%%%%%%%%%%%%%%%%%%%%%%%%%%%%%%%%%%%%%%

\subsection{The coarea formula for submersions}\label{CF}  Below we give a simple proof of the Coarea Formula for submersions; the general case is due to Federer~\cite{Federer69}.
Suppose that  $n\leq d$ and let  $X\subset\R^d$ be an open set. Recall that a  $C^1$-map $\Phi:X\to\R^n$ is called a submersion if its differential $\Phi_{*x}$ is surjective for all $x\in X$. For every  $y\in Y=\Phi(X)$,  let $dv^y(x)$ denote the volume element of the Riemannian submanifold  $\Phi^{-1}(y)$ and by $J\Phi$ the Jacobian. We introduce the measure  $\nu_y$ on $X$ by
\begin{equation}
\label{eq:20}
\nu_y(E)=
\int_{\Phi^{-1}(y)\cap E} \frac{dv^y(x)}{(J\Phi)(x)},
\qquad E\in\cB(X).
\end{equation}
It is worth observing that $\nu_y$ is finite on compact sets and   concentrated on $\Phi^{-1}(y)$.

\begin{theorem}[Coarea formula for submersions]\label{coar-form}
Suppose that  $\Phi:X\to\R^n$ is a submersion. Then 
\begin{equation}
  \label{coareformula}
dx = \int_{Y}d\nu_y\,dy ,
\end{equation}
where  $dx$ and $dy$ are the Lebesgue measures on $\R^d$ and $\R^n$, respectively.
\end{theorem}
\begin{proof}  We must show that
$$
\int_X f(x)\, dx = \int_{Y}\left( \int_{X}f(x)\frac{dv^y(x)}{(J\Phi)(x)}\right)dy 
$$
holds for every $f\in C_c(X)$. Fix $x_0\in X$. Since $\Phi_{*x_0}$ is surjective,  the Inverse Mapping Theorem implies (Corollary 5.8 in \cite{lang95}) that there exists a diffeomorphism $\Psi:U\times V \mapsto W$ such that
\begin{equation}
\Phi(\Psi(z,y))=y\qquad z\in U,\, y\in V,\label{eq:10}
\end{equation}
where $U$ is an open subset of $\R^{d-n}$,  $ V$ is an open subset of $\R^n$ and $W$ is an open neighborhood of $x_0$.

Take $f\in C_c(X)$. For any such $f$,  since $\supp{f}$ is compact, by choosing a suitable finite covering if necessary, we can always assume that $\supp{f}\subset W$. The change of variables formula and Fubini's Theorem give
\begin{align}
  \label{eq:18}
  \int_W f(x)\, dx & =\int_{V}\left( \int_{U}
  f(\Psi(z,y)){(J\Psi)(z,y)}\,dz\right)dy.
\end{align}
To obtain the coarea formula we simply  compute the Jacobian $J\Psi$. Observe that for any given $y\in V$, $\Psi^y=\Psi(\cdot,y)$ is a diffeomorphism from $U$ onto $W\cap \Phi^{-1}(y)$, regarded as a submanifold. In particular, using
this local chart, the volume element at the point $x=\Psi(z,y)$
is given by 
\begin{equation}
  \label{eq:14}
  dv^y(x)%\big|_{x=\Psi(z,y)}
  =\sqrt{\det{\left[^t(\Psi^y)_{*z}(\Psi^y)_{*z}\right]} }\, dz.
\end{equation}
Taking the derivatives of\eqref{eq:10} with respect to $z$ and $y$ separately, we obtain
\begin{equation}
\Phi_{*\Psi(z,y)} \,D_1\Psi_{(z,y)}   = 0,
\qquad
\Phi_{*\Psi(z,y)}\, D_2\Psi_{(z,y)}   = I_{n\times n} 
\label{eq:13}.
\end{equation}
Fix $(z,y)\in U\times V$ and let $P_1$ denote the orthogonal projection from $\R^d$ onto $\ker\Phi_{*\Psi(z,y)}$, and $P_2=I-P_1$  the orthogonal projection onto $[\ker\Phi_{*\Psi(z,y)}]^\perp$, which is a subspace of dimension $n$ because $\Phi$ is a submersion. From~\eqref{eq:13}  it follows that
\begin{equation}
P_2 (D_1\Psi)_{(z,y)}=0,
\qquad
P_2 (D_2\Psi)_{(z,y)}=(\Phi_{*\Psi(z,y)}\circ\iota)^{-1},
\label{eq:16}
\end{equation}
where $\iota:[\ker\Phi_{*\Psi(z,y)}]^\perp\to\R^d$ is the natural injection.
Let $R\in\text{O}(d)$ be the rotation that takes $\ker\Phi_{*\Psi(z,y)}$ onto the $z$-hyperplane (first $d-n$ coordinates)  and its orthogonal complement onto the $y$-hyperplane (last $n$ coordinates), so that $R P_1(z,y)=z$ and $RP_2(z,y)=y$. Then \eqref{eq:16} imply
$$
R \Psi_{*(z,y)} =
\begin{bmatrix}
  A  & B \\
   0   & C 
\end{bmatrix}
$$
where $A=R (D_1\Psi)_{(z,y)} $, $B=R P_1 (D_2\Psi)_{(z,y)} $ and 
$C=R P_2 (D_2\Psi)_{(z,y)}$. Therefore
$$
(J\Psi)(z,y)
=|\det{ R \Psi_{*(z,y)}}|
=|\det{A}|\,|\det{C}|
= \frac{\sqrt{\det{\left[^t(\Psi^y)_{*z}(\Psi^y)_{*z}\right]}}}{\sqrt{\det{ \left[\Phi_{*\Psi(z,y)}  
{^t}\Phi_{*\Psi(z,y)}\right]} }},
$$
where we have used~\eqref{eq:16}. Taking~\eqref{eq:14} into account,  for
$x=\Psi(z,y)$ we have
$$
(J\Psi)(z,y) \,dz = \left.\frac{dv^y(x)}{(J\Phi)(x)}\right.,
$$
which inserted in  \eqref{eq:18} yields the result.
\end{proof}

%%%%%%%%%%%%%
\section{The Jacobian criterion}\label{aldo}
We show below that Theorem~16.19 in \cite{eis95} implies that   for all $y\in
\Phi(\cR)$ the fiber $\Phi^{-1}(y)$ is finite. First of all, we can view $\Phi$ as a polynomial map  from $\C^d$ into itself, so we write $\Phi=(f_1,\dots, f_d)$. Without loss of generality we assume further that $y=0$. Following  \cite{eis95}, we write $S=\C[X_1,\dots,X_d]$ and we denote by $I$ the ideal in $S$  generated by $f_1,\dots, f_d$.
We are interested in its radical $\sqrt I$, which decomposes
as an intersection, unique up to order, of prime ideals $\sqrt I=P_1\cap\dots\cap P_s$. Hence
$V(I)=\{w\in\C^d:f_1(w)=\dots=f_d(w)=0\}=V_1\cup\dots\cup V_s$, the corresponding decomposition into irreducible components, namely $V_i=Z(P_i)$. Under the present circumstances, $\dim V_i=d-{\mathrm{codim}}(P_i)$, where the latter is the Krull codimension of $P_i$. Clearly, ${\mathrm{codim}}(P_i)=d$ if and only if $V_i$ is a singleton.
Suppose that $\dim V_j>0$ for some $j$. We will show 
that at the points $w\in V_j$ the Jacobian determinant 
\[
J\Phi(w)=\det\left(\frac{\partial f_i}{\partial w_j}(w)\right)
\]
vanishes. Suppose by contradiction that $J\Phi(w)\neq0$. Now, the codimension  of $I_{P_j}$
in $S_{P_j}$ is equal to ${\mathrm{codim}}(P_j)$ because $P_j$ is a minimal prime of $I$. By assumption, this is strictly smaller than $d$. By the Jacobian criterion, the Jacobian matrix taken modulo $P_j$ has rank strictly less than $d$. This means that $J\Phi\in P_j$. But $w\in V_j$, which implies that $J\Phi(w)=0$, a contradiction. Therefore $\Phi^{-1}(0)\cap\cR$ does not intersect irreducible components with positive dimension, hence it is a finite set.

\section*{acknowledgement} 
It is a pleasure to thank Aldo Conca  for indicating to us Theorem~16.19 in \cite{eis95} and the argument  used in Appendix~\ref{aldo}. We also thank Pietro Celada for useful discussions on the coarea formula.
%\end{acknowledgement}

%\bibliography{biblio}

\begin{thebibliography}{10}

\bibitem{ernie98}
P.~Aniello, Gianni Cassinelli, Ernesto de~Vito, and Alberto Levrero.
\newblock Square-integrability of induced representations of semidirect
  products.
\newblock {\em Rev. Math. Phys.}, 10(3):301--313, 1998.

\bibitem{bourbaki_GT}
Nicolas Bourbaki.
\newblock {\em General topology. {C}hapters 1--4}.
\newblock Elements of Mathematics (Berlin). Springer-Verlag, Berlin, 1998.
\newblock Translated from the French, Reprint of the 1989 English translation.

\bibitem{bourbaki_int}
Nicolas Bourbaki.
\newblock {\em Integration. {I}. {C}hapters 1--6}.
\newblock Elements of Mathematics (Berlin). Springer-Verlag, Berlin, 2004.
\newblock Translated from the 1959, 1965 and 1967 French originals by Sterling
  K. Berberian.

\bibitem{bourbaki_intII}
Nicolas Bourbaki.
\newblock {\em Integration. {II}. {C}hapters 7--9}.
\newblock Elements of Mathematics (Berlin). Springer-Verlag, Berlin, 2004.
\newblock Translated from the 1963 and 1969 French originals by Sterling K.
  Berberian.

\bibitem{bre83}
Ha{\"{\i}}m Brezis.
\newblock {\em Analyse fonctionnelle}.
\newblock Collection Math\'ematiques Appliqu\'ees pour la Ma\^\i trise.
  [Collection of Applied Mathematics for the Master's Degree]. Masson, Paris,
  1983.
\newblock Th{\'e}orie et applications. [Theory and applications].

\bibitem{cava77}
C.~Castaing and M.~Valadier.
\newblock {\em Convex analysis and measurable multifunctions}.
\newblock Lecture Notes in Mathematics, Vol. 580. Springer-Verlag, Berlin,
  1977.

\bibitem{codenota06b}
E.~Cordero, F.~De~Mari, K.~Nowak, and A.~Tabacco.
\newblock Reproducing groups for the metaplectic representation.
\newblock In {\em Pseudo-differential operators and related topics}, volume 164
  of {\em Oper. Theory Adv. Appl.}, pages 227--244. Birkh\"auser, Basel, 2006.

\bibitem{codenota10}
E.~Cordero, F.~De~Mari, K.~Nowak, and A.~Tabacco.
\newblock Dimensional upper bounds for admissible subgroups for the metaplectic
  representation.
\newblock {\em Math. Nachr.}, 283(7):982--993, 2010.

\bibitem{codenota06a}
Elena Cordero, Filippo De~Mari, Krzysztof Nowak, and Anita Tabacco.
\newblock Analytic features of reproducing groups for the metaplectic
  representation.
\newblock {\em J. Fourier Anal. Appl.}, 12(2):157--180, 2006.

\bibitem{dakustte09}
Stephan Dahlke, Gitta Kutyniok, Gabriele Steidl, and Gerd Teschke.
\newblock Shearlet coorbit spaces and associated {B}anach frames.
\newblock {\em Appl. Comput. Harmon. Anal.}, 27(2):195--214, 2009.

\bibitem{dix57}
Jacques Dixmier.
\newblock {\em Les alg\`ebres d'op\'erateurs dans l'espace hilbertien
  ({A}lg\`ebres de von {N}eumann)}.
\newblock Cahiers scientifiques, Fascicule XXV. Gauthier-Villars, Paris, 1957.

\bibitem{dix64}
Jacques Dixmier.
\newblock {\em Les {$C^{\ast} $}-alg\`ebres et leurs repr\'esentations}.
\newblock Cahiers Scientifiques, Fasc. XXIX. Gauthier-Villars \& Cie,
  \'Editeur-Imprimeur, Paris, 1964.

\bibitem{dumo76}
M.~Duflo and Calvin~C. Moore.
\newblock On the regular representation of a nonunimodular locally compact
  group.
\newblock {\em J. Functional Analysis}, 21(2):209--243, 1976.

\bibitem{effros65}
Edward~G Effros.
\newblock Transformation groups and {\$}cspast {\$}-algebras.
\newblock {\em Ann. of Math. (2)}, 81:38--55, 1965.

\bibitem{eis95}
David Eisenbud.
\newblock {\em Commutative algebra}, volume 150 of {\em Graduate Texts in
  Mathematics}.
\newblock Springer-Verlag, New York, 1995.
\newblock With a view toward algebraic geometry.

\bibitem{Federer69}
Herbert Federer.
\newblock {\em Geometric measure theory}.
\newblock Die Grundlehren der mathematischen Wissenschaften, Band 153.
  Springer-Verlag New York Inc., New York, 1969.

\bibitem{fol95}
Gerald~B. Folland.
\newblock {\em A course in abstract harmonic analysis}.
\newblock Studies in Advanced Mathematics. CRC Press, Boca Raton, FL, 1995.

\bibitem{fuhr05}
Hartmut F{\"u}hr.
\newblock {\em Abstract harmonic analysis of continuous wavelet transforms},
  volume 1863 of {\em Lecture Notes in Mathematics}.
\newblock Springer-Verlag, Berlin, 2005.

\bibitem{fu09}
Hartmut F{\"u}hr.
\newblock Generalized {C}alder\'on conditions and regular orbit spaces.
\newblock {\em Colloq. Math.}, 120(1):103--126, 2010.

\bibitem{gukula06}
Kanghui Guo, Gitta Kutyniok, and Demetrio Labate.
\newblock Sparse multidimensional representations using anisotropic dilation
  and shear operators.
\newblock In {\em Wavelets and splines: {A}thens 2005}, Mod. Methods Math.,
  pages 189--201. Nashboro Press, Brentwood, TN, 2006.

\bibitem{Halmos50}
Paul~R. Halmos.
\newblock {\em Measure {T}heory}.
\newblock D. Van Nostrand Company, Inc., New York, N. Y., 1950.

\bibitem{HeRossII79}
Edwin Hewitt and Kenneth~A. Ross.
\newblock {\em Abstract harmonic analysis. {V}ol. {II}: {S}tructure and
  analysis for compact groups. {A}nalysis on locally compact {A}belian groups}.
\newblock Die Grundlehren der mathematischen Wissenschaften, Band 152.
  Springer-Verlag, New York, 1970.

\bibitem{klli72}
Adam Kleppner and Ronald~L Lipsman.
\newblock The plancherel formula for group extensions. i, ii.
\newblock {\em Ann. Sci. Ecole Norm. Sup. (4)}, 5:459--516; ibid. (4) 6 (1973),
  103--132, 1972.

\bibitem{kur92}
Casimir Kuratowski.
\newblock {\em Topologie. {I} et {II}}.
\newblock \'Editions Jacques Gabay, Sceaux, 1992.
\newblock Part I with an appendix by A. Mostowski and R. Sikorski, Reprint of
  the fourth (Part I) and third (Part II) editions.

\bibitem{kula09}
Gitta Kutyniok and Demetrio Labate.
\newblock Resolution of the wavefront set using continuous shearlets.
\newblock {\em Trans. Amer. Math. Soc.}, 361(5):2719--2754, 2009.

\bibitem{lang95}
Serge Lang.
\newblock {\em Differential and {R}iemannian manifolds}, volume 160 of {\em
  Graduate Texts in Mathematics}.
\newblock Springer-Verlag, New York, third edition, 1995.

\bibitem{LWWW}
R.~S. Laugesen, N.~Weaver, G.~L. Weiss, and E.~N. Wilson.
\newblock A characterization of the higher dimensional groups associated with
  continuous wavelets.
\newblock {\em J. Geom. Anal.}, 12(1):89--102, 2002.

\bibitem{mackey52}
George~W. Mackey.
\newblock Induced representations of locally compact groups. {I}.
\newblock {\em Ann. of Math. (2)}, 55:101--139, 1952.

\bibitem{mac57}
George~W. Mackey.
\newblock Borel structure in groups and their duals.
\newblock {\em Trans. Amer. Math. Soc.}, 85:134--165, 1957.

\bibitem{RS80}
Michael Reed and Barry Simon.
\newblock {\em Methods of modern mathematical physics. {I}}.
\newblock Academic Press Inc. [Harcourt Brace Jovanovich Publishers], New York,
  second edition, 1980.
\newblock Functional analysis.

\bibitem{sch97}
Laurent Schwartz.
\newblock {\em Analyse. {III}}.
\newblock Hermann, Paris, 1997.

\bibitem{st64}
Shlomo Sternberg.
\newblock {\em Lectures on differential geometry}.
\newblock Prentice-Hall Inc., Englewood Cliffs, N.J., 1964.

\end{thebibliography}
%\bibliographystyle{plain}

\end{document}